\documentclass[11pt,reqno]{amsart}

\usepackage{amsmath,amssymb,amsfonts,amsthm,mathtools,mathrsfs}
\usepackage{aliascnt}
\usepackage{bm}
\usepackage{enumitem}
\usepackage[margin=1in]{geometry}
\usepackage[dvipsnames]{xcolor}
\usepackage[colorlinks=true,linkcolor=RoyalBlue,citecolor=red,urlcolor=RoyalBlue]{hyperref}
\usepackage{microtype}

\numberwithin{equation}{section}
\allowdisplaybreaks

\newtheorem{theorem}{Theorem}[section]
\newaliascnt{proposition}{theorem}
\newtheorem{proposition}[proposition]{Proposition}
\aliascntresetthe{proposition}
\newaliascnt{lemma}{theorem}
\newtheorem{lemma}[lemma]{Lemma}
\aliascntresetthe{lemma}
\newaliascnt{corollary}{theorem}
\newtheorem{corollary}[corollary]{Corollary}
\aliascntresetthe{corollary}
\newaliascnt{definition}{theorem}
\newtheorem{definition}[definition]{Definition}
\aliascntresetthe{definition}
\newaliascnt{remark}{theorem}
\newtheorem{remark}[remark]{Remark}
\aliascntresetthe{remark}

\newcommand{\R}{\mathbb R}
\newcommand{\C}{\mathbb C}
\newcommand{\cS}{\mathcal S}
\newcommand{\cH}{\mathcal H}
\newcommand{\cG}{\mathcal G}
\newcommand{\cW}{\mathcal W}
\newcommand{\cL}{\mathcal L}
\newcommand{\cR}{\mathcal R}
\newcommand{\cT}{\mathcal T}
\newcommand{\cN}{\mathcal N}
\newcommand{\cQ}{\mathcal Q}
\newcommand{\fkX}{\mathfrak X}
\newcommand{\cP}{\mathcal P}
\newcommand{\cM}{\mathcal M}
\newcommand{\fM}{\mathfrak M}
\newcommand{\dd}{\,\mathrm d}
\newcommand{\la}{\langle}
\newcommand{\ra}{\rangle}
\newcommand{\norm}[1]{\lVert #1\rVert}
\newcommand{\abs}[1]{\lvert #1\rvert}
\DeclareMathOperator{\supp}{supp}
\DeclareMathOperator{\dist}{dist}
\renewcommand{\Re}{\operatorname{Re}}

\begin{document}

\title[Modified Scattering and Asymptotics for Perturbed One-Dimensional Cubic NLS]{Modified Scattering and Asymptotics for Perturbed One-Dimensional Cubic NLS}

\author[Vinh Nguyen]{Vinh Nguyen$^1$}
\address{
$^1$Department of Mathematics, University of California, Berkeley, CA 94704, USA}
\email{vnguyen26@berkeley.edu}

\author[Avy Soffer]{Avy Soffer$^2$}
\address{$^2$Department of Mathematics, Rutgers University, Piscataway, NJ 08854, USA}
\email{soffer@math.rutgers.edu}

\author[Truong Vu]{Truong Vu$^{3,4}$}

%\author[Truong Vu]{Truong Vu$^{2,3}$}
\address{$^3$Departments of Mathematics and Computational Medicine, University of California, Los Angeles, CA 90095, USA}
\address{$^4$Applied Mathematics and Computational Research (AMCR) division, Lawrence Berkeley National Laboratory, Berkeley, CA 94720, USA}
\email{truongvu@math.ucla.edu}

\date{\today}

\thanks{\textbf{Acknowledgment.} A.S. was partially supported by NSF grant DMS-2205931 and Simons Travel Grant. V.N. and T.V. were partially supported by an AMS--Simons travel grant.}

\begin{abstract}
We study the long-time dynamics of small solutions to the one-dimensional nonlinear Schr\"odinger equation
\[
i\partial_t v+\partial_x^2v-\beta\abs{v}^2v
+\cW(x)\abs{v}^4v+\gamma i\partial_xv=0,
\]
where $\cW$ is spatially localized. The cubic nonlinearity is long range and produces the logarithmic phase correction, whereas the localized quintic term is short range at leading order. We construct a global forward modified wave operator for small complex asymptotic profiles and prove quantitative final-state estimates. For small data in the weighted energy space, we also establish global existence, sharp $t^{-1/2}$ decay, and forward modified scattering with a unique asymptotic profile. The principal new phenomenon occurs beyond this leading law. The exact Duhamel tail generated by the localized quintic term admits a quantitative inner scaling limit at the distinguished frequency $\zeta=-\gamma/2$ on the scale $\abs{\zeta+\gamma/2}\sim t^{-1/2}$. Its universal shape is explicit and depends on the value of the scattering profile on the distinguished ray and on the zeroth moment of $\cW$. When both quantities are nonzero, the limit is nontrivial, belongs optimally to $C^{2,1}_{\mathrm{loc}}$, and is not $C^3$ at the center. Away from the corresponding self-similar ray $\xi=-\gamma$, we construct rigorously defined higher-order outer expansions through every integer order allowed by the decay of $\cW$, and to each fixed finite order when $\cW$ is rapidly decreasing.
\end{abstract}

\maketitle
\tableofcontents

\section{Introduction}\label{sec:intro}

The long-time dynamics of one-dimensional nonlinear dispersive equations are strongly influenced by borderline interactions. For the cubic nonlinear Schr\"odinger equation with $\lambda_{\mathrm c}\in\R$,
\begin{equation}\label{eq:cubic-NLS-intro}
 i\partial_tu+\partial_x^2u=\lambda_{\mathrm c}\abs{u}^2u,
\end{equation}
the linear decay rate is $t^{-1/2}$. The cubic interaction therefore has effective size $t^{-1}$, which is not integrable at infinity. Small solutions do not scatter to a free solution with a time-independent profile. Instead, the leading asymptotic state acquires a logarithmic nonlinear phase.

\smallskip
The modified scattering theory for \eqref{eq:cubic-NLS-intro} has several
complementary formulations. Ozawa constructed modified wave operators in
\cite{O1991}. Hayashi et al. developed an early scattering theory for the
cubic equation and related Hartree models in \cite{HKN1998}, and Hayashi and
Naumkin subsequently established precise large-time asymptotics and
asymptotic completeness in \cite{HN1998}. Shimomura and Tonegawa obtained
further long-range scattering results in one and two dimensions in
\cite{ST2004}. Lindblad and Soffer treated a translation-invariant
cubic--quintic model by a direct construction in self-similar variables,
constructing modified wave operators, proving small-data completeness, and
developing higher-order asymptotics \cite{LS2006}. Fourier-space and
space-time resonance ideas were developed further by Kato and Pusateri in
\cite{KP2011}. Ifrim and Tataru obtained global bounds, sharp decay, and
modified scattering through testing by wave packets \cite{IT2015}. Their
later work \cite{IT2024} places this method in a general framework for
one-dimensional dispersive equations with cubic long-range interactions.

\smallskip
Murphy and Van Hoose adapted modified scattering to a dispersion-managed
model in \cite{MVH2022}, while Murphy and Zheng treated a cubic equation
with a time-dependent dispersion map in \cite{MZ2026}. Kawamoto and
Mizutani constructed modified wave operators for the defocusing cubic
equation with large scattering data in \cite{KM2025b}. Jendrej and Salvi
derived arbitrary-order expansions for small solutions of one-dimensional
gauge-invariant polynomial Schr\"odinger equations in \cite{JS2026}. Chen
showed that the fixed-profile formulation can fail at the unweighted $L^2$
endpoint in \cite{C2026}. For the integrable defocusing problem, Deift and
Zhou developed nonlinear steepest descent for weighted Sobolev data in
\cite{DZ2003}, and Dieng and McLaughlin refined the method using a
$\bar\partial$ formulation in \cite{DM2008}. Murphy surveys PDE and
integrable approaches in \cite{M2021}.

\smallskip
Deift and Zhou developed a perturbation theory for infinite-dimensional
integrable systems and analyzed the defocusing cubic NLS under a small
higher-power perturbation in \cite{DZ2002}. Their method evolves nonlinear
scattering data and allows weighted data that need not be small, while the
perturbation parameter depends on the size of those data. Chen, Liu, and
Tian treated a localized Schwartz coefficient and powers above the cubic
one in \cite{CLT2025}. They proved persistence of the leading integrable
asymptotics, with a uniform $O(t^{-3/4})$ remainder, when the perturbation
parameter is sufficiently small relative to the weighted scattering data.

These inverse-scattering results do not isolate the shrinking
$T^{-1/2}$ spectral window or the tail of the exact cubic-renormalized
Fourier profile studied here. Our direct argument permits either sign of
the cubic coefficient and a fixed localized coefficient, but it requires
small initial data. This smallness closes the contraction and wave-packet
bootstrap and is not asserted to be intrinsic in the defocusing regime.

\smallskip
A second line of work concerns one-dimensional equations with spatially
inhomogeneous linear Hamiltonians. Modified scattering in the presence of
external potentials has been studied in
\cite{GPR2018,MMS2019,CP2024,KM2025a,S2024}. Related scattering and
asymptotic-stability results for localized or singular potentials appear in
\cite{SS2025,CV2009,L2016,FV2018,BV2016,MMS2023}. These problems involve a
linear inhomogeneity and different spectral mechanisms, whereas the present
perturbation is nonlinear and localized only in the quintic interaction.

\smallskip
Spatially dependent nonlinearities form a related but distinct class. Genoud and Stuart studied bound states generated by decaying nonlinear coefficients in \cite{GS2008}. Aoki et al. treated critical inhomogeneous nonlinearities and modified wave operators in \cite{AIMMU2021}. Aloui and Tayachi proved global existence and scattering results for decaying inhomogeneities in \cite{AT2024}. Cui et al. established decay estimates and scattering criteria for one-dimensional inhomogeneous nonlinear Schr\"odinger equations in \cite{CLZ2026}. Baker and Murphy obtained large-data scattering results for one-dimensional inhomogeneous nonlinearities in \cite{BM2025}. Harrop-Griffiths et al. analyzed scattering when nonlinear effects are concentrated in space in \cite{HGKV2026}. Xie proved small-data modified scattering for a cubic equation containing both a repulsive point interaction and a localized inhomogeneous coefficient, and used the scattering map in an inverse problem \cite{X2026}. These results show that spatial localization can improve time integrability while introducing distinguished spatial or spectral regions.

\medskip
In this work, we study the equation
\begin{equation}\label{eq:main}
 i\partial_t v+\partial_x^2v-\beta\abs{v}^2v+\cW(x)\abs{v}^4v+\gamma i\partial_xv=0,
 \qquad t\geq1,\quad x\in\R,
\end{equation}
where $\beta,\gamma\in\R$ and $\cW$ is real-valued and smooth. We assume
that there exist $p>2$ and constants $C = C(\cW,j)>0$ such that, for every integer
$j\geq0$,
\begin{equation}\label{eq:W-assumption}
 \abs{\partial_x^j\cW(x)}\leq C\la x\ra^{-p-j},
 \qquad x\in\R.
\end{equation}
With the convention in
\eqref{eq:main}, the positive part of $\cW$ is focusing in the stationary
energy, while the negative part is defocusing.

\smallskip
For the self-similar analysis, it is convenient to use the translated coordinate
\[
 y=x-\gamma t,
 \qquad
 u(t,y)=v(t,y+\gamma t).
\]
Then $u$ solves
\begin{equation}\label{eq:comoving}
 i\partial_tu+\partial_y^2u-\beta\abs{u}^2u+\cW(y+\gamma t)\abs{u}^4u=0.
\end{equation}
To identify the large-time dynamics, we introduce the self-similar variables
\begin{equation}\label{eq:selfsimilar-transform-intro}
 s=t,
 \qquad
 \xi=\frac{y}{t},
 \qquad
 u(t,y)=t^{-1/2}e^{iy^2/(4t)}U(s,\xi).
\end{equation}
The profile $U$ satisfies
\begin{equation}\label{eq:selfsimilar-intro}
 i\partial_sU+s^{-2}\partial_\xi^2U-\beta s^{-1}\abs{U}^2U
 +s^{-2}\cW\bigl(s(\xi+\gamma)\bigr)\abs{U}^4U=0.
\end{equation}
The homogeneous cubic interaction carries the nonintegrable factor $s^{-1}$
and produces the logarithmic phase correction. The localized quintic
interaction carries the integrable factor $s^{-2}$ and is short range at
leading order. This remains true on the ray $\xi=-\gamma$, where the argument
of $\cW$ stays at the origin. The ray becomes relevant at finer scales because
differentiation in $\xi$ acts on $\cW\bigl(s(\xi+\gamma)\bigr)$.

\smallskip
These observations lead first to a modified final-state problem. Given a small complex profile $A$ with $\la\xi\ra A\in H^3$, we take
\[
 U_A(s,\xi)
 =A(\xi)\exp\left(-i\beta\abs{A(\xi)}^2\log s\right).
\]
We work directly with the complex profile $A$ rather than separating it into
amplitude and phase. This avoids a phase-normalization difficulty at zeros of
$A$, where neither $\abs A$ nor a global argument need have the regularity
required by the analysis. In the original variables the corresponding leading
state is
\begin{equation}\label{eq:vA-intro}
 v_A(t,x)=t^{-1/2}e^{i(x-\gamma t)^2/(4t)}
 A\left(\frac{x-\gamma t}{t}\right)
 \exp\left[-i\beta\left|A\left(\frac{x-\gamma t}{t}\right)\right|^2\log t\right].
\end{equation}
In \autoref{thm:wave-operator} we show that, for each fixed $\beta$, every sufficiently small profile in this weighted Sobolev class determines a global solution converging to \eqref{eq:vA-intro}, with a remainder of order $t^{-1}(1+\abs\beta\log t)^2$ in $H^1\cap L^\infty$. This construction defines the wave operator at $t=1$. The smallness condition controls the $t^{-1}$ cubic linearization, and conservation of mass and energy after the gauge transformation
\begin{equation*}
 v(t,x)=e^{-i\gamma x/2+i\gamma^2t/4}Q(t,x),
\end{equation*}
with $Q$ solving \eqref{eq:Q-stationary}, gives global continuation.

\smallskip
We next address the converse problem. For small initial data in the weighted
energy space $\Sigma$, \autoref{thm:forward-scattering} gives a unique global solution for
$\varepsilon=\norm{u(1)}_\Sigma$ sufficiently small, decay at the rate
$\norm{u(t)}_{L^\infty}\lesssim \varepsilon t^{-1/2}$, and a unique complex
profile $A$ describing the forward modified scattering. The resulting
scattering map is continuous, injective, and gauge equivariant on a sufficiently
small $\Sigma$-ball. We also prove the cubic-order linearization
\[
 A[u_1](\xi)
 =\frac{e^{-i\pi/4}}{2\sqrt\pi}
 e^{i\xi^2/4}\widehat{u_1}(\xi/2)
 +O_{L^2\cap L^\infty}(\norm{u_1}_\Sigma^3).
\]
This makes the condition $A(-\gamma)\neq0$ explicit and stable on an open
set of small data. The proof combines the Galilean vector field with the
wave-packet testing method of Ifrim and Tataru \cite{IT2015,IT2024}. The term
$\cW(y+\gamma t)\abs{u}^4u$ contributes integrable errors to the vector-field
estimates and the packet ordinary differential equation.

The localized coefficient reappears beyond the leading modified-scattering law. For $r\in\R$, suppressing the dispersive term in \eqref{eq:selfsimilar-intro} formally produces
\[
 \mathcal I_{\cW}(r)=\int_1^\infty s^{-2}\cW(sr)\dd s.
\]
As $r\to0$, this integral satisfies
\[
 \mathcal I_{\cW}(r)=\cW(0)+\cW'(0)r\log(1/\abs r)+O(\abs r).
\]
Thus $\mathcal I_{\cW}$ has an $r\log(1/\abs r)$ singularity when
$\cW'(0)\neq0$.
This formal expression does not describe the phase of a solution. On the
coefficient layer $\abs r\sim s^{-1}$, one has $\partial_\xi\sim s$, so the
term $s^{-2}\partial_\xi^2U$ is not perturbative.

Instead, in \autoref{sec:distinguished} we introduce the cubic-renormalized
interaction profile
\[
 \cP_u(t,\zeta)
 :=e^{it\zeta^2}\widehat u(t,\zeta)
 +i\beta\int_1^t e^{is\zeta^2}
 \widehat{\abs{u(s)}^2u(s)}(\zeta)\dd s.
\]
We prove that it converges and denote its limit by $\cP_{u,+}$. The profile
compensates the full cubic Duhamel increment and satisfies an
exact identity in which its tail $\cP_{u,+}-\cP_u(T)$ is given by the
localized quintic forcing. If $A$ is the modified scattering
profile from \autoref{thm:forward-scattering}, $A_0=A(-\gamma)$,
$\alpha=\beta\abs{A_0}^2$, and $M_0=\int_\R\cW(x)\dd x$, then near the distinguished
frequency $\zeta_0=-\gamma/2$ we prove the scaling limit
\[
 T^{3/2}e^{i\alpha\log T}
 \left[\cP_{u,+}-\cP_u(T)\right]
 \left(\zeta_0+\frac{\eta}{\sqrt T}\right)
 \longrightarrow
 iM_0\abs{A_0}^4A_0
 \int_1^\infty \tau^{-5/2-i\alpha}e^{i\eta^2\tau}\dd\tau.
\]
Thus the tail of $\cP_u$ admits a quantitative inner scaling limit of
width $T^{-1/2}$ in frequency, equivalently in self-similar velocity. When
$M_0A_0\neq0$, the limit is nontrivial and has a nonanalytic term
$\abs{\eta}^{3+2i\alpha}$. Its second derivative is locally Lipschitz, but its
third derivative does not extend continuously to $\eta=0$. The $L^2$ and
$L^\infty$ tail rates are $T^{-7/4}$ and $T^{-3/2}$, and both are optimal
when $M_0A_0\neq0$.

The terminal variable also has two complementary characterizations. It is
a fifth-order perturbation of the free interaction transform of the initial
datum, and it equals the stationary-phase image of $A$ plus a convergent
nonresonant cubic correction. If $p>4$ and the initial datum is Schwartz,
set
\begin{equation}\label{eq:C-alpha}
 C_\alpha
 :=\int_0^\infty \sigma^{-5/2-i\alpha}
 \left(e^{i\sigma}-1-i\sigma\right)\dd\sigma.
\end{equation}
We prove that $C_\alpha\neq0$ and that there are coefficients
$b_0,\ldots,b_3$ such that
\[
 \cP_{u,+}(\zeta_0+r)
 =\sum_{j=0}^3b_jr^j
 +iM_0\abs{A_0}^4A_0C_\alpha\abs r^{3+2i\alpha}
 +o(\abs r^3).
\]
Thus the actual terminal profile has a quadratic Peano approximation with
an $O(\abs r^3)$ remainder, but it has no third-order Taylor expansion when
$M_0A_0\neq0$. The small-data linearization above shows that this case
occurs for a relatively open set of Schwartz data.

\smallskip
This is the point at which our result differs from existing perturbed
modified-scattering theory. The works \cite{DZ2002,CLT2025} prove persistence
of the leading integrable asymptotics under higher-power perturbations, while
\cite{LS2006,JS2026} constructs higher-order outer expansions for
translation-invariant polynomial nonlinearities. They do not resolve a
time-dependent spectral window created by a localized nonlinear coefficient.
Results on the leading modified-scattering profile establish existence and
Sobolev or higher regularity under weighted hypotheses
\cite{O1991,HN1998,DZ2003,JS2026}, whereas \cite{C2026} concerns failure of
existence of a fixed profile at the unweighted $L^2$ endpoint. By contrast,
our exact cubic-renormalized terminal profile exists for small weighted data
but, even for Schwartz initial data, has no third-order Taylor expansion at
$\zeta=-\gamma/2$ when $p>4$ and $M_0A_0\neq0$. Its universal inner scaling
profile has the sharp regularity $C^{2,1}_{\mathrm{loc}}$ and is not $C^3$.
Thus the new issue is a localized, next-order loss of differentiability of an
existing asymptotic profile, not a failure of leading modified scattering.

The higher-order theorem concerns a different profile class. If $A$ is
compactly supported a positive distance from $-\gamma$, then the recursive
coefficients are exactly those of the translation-invariant cubic equation.
The decay of $\cW$ controls only the remainder. The construction reaches
every integer $N\leq\lfloor p\rfloor$, while a Schwartz coefficient is
invisible at every fixed algebraic order. The support condition forces
$A(-\gamma)=0$, so this global outer theorem and the nontrivial inner
theorem do not apply to the same profile. No matched inner--outer expansion
for one solution is claimed.

\medskip
The rest of the paper is organized as follows. In \autoref{sec:wave-operators} we
construct the modified wave operator and establish its basic mapping
properties. In \autoref{sec:forward-scattering} we prove small-data global
existence, decay, modified scattering, and the linearization of the
scattering map. In \autoref{sec:distinguished} we characterize the exact
cubic-renormalized profile, compute its inner tail, and prove the Peano
singularity of its terminal value. In \autoref{sec:higher} we identify the
outer coefficients with the pure-cubic Lindblad--Soffer expansion.

\smallskip
\textbf{Notations.}
All functions are complex-valued unless stated otherwise. We write
$\la x\ra=(1+x^2)^{1/2}$ and
$\norm{f}_q=\norm{f}_{L^q(\R)}$ for $1\leq q\leq\infty$. The relation
$X\lesssim Y$ means
$X\leq CY$ for a positive constant $C$. Such constants may change from line
to line, and their dependence is indicated when needed. We define
\begin{equation}\label{def:Sigma}
 \Sigma=\left\{f\in H^1(\R):xf\in L^2(\R)\right\},
 \qquad
 \norm{f}_\Sigma=\norm{f}_{H^1}+\norm{xf}_2.    
\end{equation}
Our Fourier transform convention is
\begin{equation}\label{eq:fourier-convention}
 \widehat f(\zeta)=\int_\R e^{-i\zeta y}f(y)\dd y.
\end{equation}
The space $C_0(\R)$ consists of continuous functions that vanish at infinity.
The space $\cS(\R)$ is the Schwartz class. We use
$\mathbb N=\{1,2,\ldots\}$, $\mathbb N_0=\{0,1,2,\ldots\}$, and write
$\mathbf 1_E$ for the indicator of a set $E$.
We write $C_{\mathrm{loc}}^{2,1}(\R)$ for the $C^2$ functions whose second
derivative is locally Lipschitz. For $\alpha\in\R$, complex powers are defined by
\[
 \abs{\eta}^{3+2i\alpha}
 :=\abs{\eta}^3e^{2i\alpha\log\abs{\eta}}\quad\text{for }\eta\neq0,
\]
and the value at $\eta=0$ is defined to be zero.

\section{A forward modified wave operator}\label{sec:wave-operators}
The final-state construction below adapts the self-similar formulation and
vanishing-final-data iteration used by Lindblad and Soffer for the
translation-invariant critical nonlinear Schr\"odinger equation
\cite[Sections~1--2]{LS2006}. We rederive the argument in the present
normalization and work directly with a complex asymptotic profile, without
requiring a global amplitude--phase decomposition. The moving coefficient
$\cW(y+\gamma t)$ is estimated in physical variables, since derivatives of
$\cW(t(\xi+\gamma))$ are not uniformly small near the distinguished ray
$\xi=-\gamma$.

\subsection{Local theory, conservation, and continuation}
\label{subsec:local-theory}

We use the definition of $\Sigma$ from \eqref{def:Sigma}, with $y$ as the
spatial variable. We introduce the Galilean vector field
\[
 L(t):=y+2it\partial_y.
\]
For smooth complex-valued functions $u$ and $F$, one has
\begin{align}
 [i\partial_t+\partial_y^2,L(t)]&=0,
 \label{eq:L-commute}\\
 L(t)\bigl(\abs{u}^2u\bigr)
 &=2\abs{u}^2L(t)u-u^2\overline{L(t)u},
 \label{eq:L-cubic}\\
 L(t)\bigl(\abs{u}^4u\bigr)
 &=3\abs{u}^4L(t)u
 -2\abs{u}^2u^2\overline{L(t)u},
 \label{eq:L-quintic}\\
 L(t)\left[\cW(y+\gamma t)F\right]
 &=\cW(y+\gamma t)L(t)F
 +2it\cW'(y+\gamma t)F.
 \label{eq:L-W}
\end{align}

We first record the local theory and the small-mass continuation criterion.

\begin{proposition}
\label{prop:local-theory}
Let $\beta,\gamma\in\R$, and let $\cW$ be real-valued and satisfy
\eqref{eq:W-assumption}.

\begin{enumerate}[label=\textup{(\roman*)}]
\item
For every $t_0\geq1$ and $u_0\in H^1(\R)$, equation
\eqref{eq:comoving} has a unique maximal mild solution
\[
 u\in C\bigl([t_0,T_+);H^1(\R)\bigr)
 \cap C^1\bigl([t_0,T_+);H^{-1}(\R)\bigr),
\]
where $T_+\in(t_0,\infty]$. The solution depends continuously on $u_0$ in $H^1$. If
$T_+<\infty$, then
\begin{equation}\label{eq:H1-blowup-alternative}
 \limsup_{t\uparrow T_+}\norm{u(t)}_{H^1}=\infty.
\end{equation}
The analogous local existence, uniqueness, and continuation statement
holds backward from $t_0$. Moreover,
\begin{equation}\label{eq:local-mass-conservation}
 \norm{u(t)}_2=\norm{u_0}_2
\end{equation}
throughout the lifespan.

\item
If, in addition, the datum in part \textup{(i)} satisfies $u_0\in\Sigma$, then
the corresponding solution satisfies
\[
 u\in C\bigl([t_0,T_+);\Sigma\bigr),
\]
with continuous dependence in $\Sigma$. The maximal $\Sigma$ lifespan
coincides with the maximal $H^1$ lifespan. 

\item
Consider the stationary-coefficient equation
\begin{equation}\label{eq:Q-stationary}
 i\partial_tQ+\partial_x^2Q-\beta\abs{Q}^2Q
 +\cW(x)\abs{Q}^4Q=0.
\end{equation}
Every $H^1$ solution conserves its mass and the energy
\begin{equation*}
 \operatorname{Mass}(Q)=\norm Q_2^2,
 \qquad
 E(Q)=\int_\R\left(
 \abs{\partial_xQ}^2+\frac{\beta}{2}\abs{Q}^4
 -\frac13\cW(x)\abs{Q}^6\right)\dd x.
\end{equation*}
There exists
\[
 m_0=m_0\bigl(\norm{\cW}_\infty\bigr)>0
\]
such that every $H^1$ solution of \eqref{eq:Q-stationary} satisfying
$\norm{Q(t_0)}_2\leq m_0$ at some time $t_0$ in its lifespan is global in
both time directions. More
precisely, throughout its lifespan,
\begin{equation}\label{eq:energy-coercive-small-mass}
 E(Q(t))
 \geq\frac12\norm{\partial_xQ(t)}_2^2
 -C\beta^2\norm{Q(t_0)}_2^6.
\end{equation}
\end{enumerate}
\end{proposition}

\begin{proof}
Write
\[
 \cW_t(y):=\cW(y+\gamma t),
 \qquad
 \mathfrak N_t(f):=\beta\abs{f}^2f-\cW_t\abs{f}^4f.
\]
The standing assumption implies that
$t\mapsto\cW_t$ is continuous in $W^{1,\infty}(\R)$. The algebra
property of $H^1(\R)$ and the embedding
$H^1(\R)\hookrightarrow L^\infty(\R)$ give, uniformly in $t$,
\begin{equation}\label{eq:local-nonlinear-H1}
 \norm{\mathfrak N_t(f)}_{H^1}
 \leq C_{\beta,\cW}
 \left(\norm f_{H^1}^3+\norm f_{H^1}^5\right).
\end{equation}
Indeed,
\[
 \norm{\partial_y(\abs{f}^2f)}_2
 \leq C\norm f_\infty^2\norm{\partial_yf}_2
\]
and
\[
 \norm{\partial_y(\cW_t\abs{f}^4f)}_2
 \leq \norm{\cW'}_\infty\norm f_\infty^4\norm f_2
 +C\norm{\cW}_\infty\norm f_\infty^4
 \norm{\partial_yf}_2.
\]
Factoring polynomial differences and applying the same estimates shows
that the map is locally Lipschitz. More precisely, if
$\max\{\norm f_{H^1},\norm g_{H^1}\}\leq R$, then
\begin{equation*}
 \norm{\mathfrak N_t(f)-\mathfrak N_t(g)}_{H^1}
 \leq C_{\beta,\cW}(R^2+R^4)\norm{f-g}_{H^1}.
\end{equation*}
Let $S(t):=e^{it\partial_y^2}$ and, for $T>0$, set
\[
 \mathcal X_T^{\mathrm{loc}}:=C\bigl([t_0,t_0+T];H^1(\R)\bigr),
 \qquad
 \norm f_{\mathcal X_T^{\mathrm{loc}}}
 :=\sup_{t_0\leq t\leq t_0+T}\norm{f(t)}_{H^1}.
\]
Define
\[
 \fM(f)(t)
 =S(t-t_0)u_0
 -i\int_{t_0}^tS(t-s)\mathfrak N_s(f(s))\dd s.
\]
Put $R_{\mathrm{loc}}=2\norm{u_0}_{H^1}$. The case $u_0=0$ is immediate.
Otherwise, choose $T>0$ such that
\begin{equation}\label{eq:local-time-choice}
 C_{\beta,\cW}T(R_{\mathrm{loc}}^2+R_{\mathrm{loc}}^4)\leq\frac14.
\end{equation}
Since $S(t)$ is unitary on $H^1$, equations
\eqref{eq:local-nonlinear-H1} and \eqref{eq:local-time-choice} imply
\[
 \norm{\fM(f)}_{\mathcal X_T^{\mathrm{loc}}}
 \leq \norm{u_0}_{H^1}
 +C_{\beta,\cW}T(R_{\mathrm{loc}}^3+R_{\mathrm{loc}}^5)
 \leq \norm{u_0}_{H^1}+\frac{R_{\mathrm{loc}}}{4}
 <R_{\mathrm{loc}}
\]
on the closed radius-$R_{\mathrm{loc}}$ ball in
$\mathcal X_T^{\mathrm{loc}}$. Likewise,
\[
 \norm{\fM(f)-\fM(g)}_{\mathcal X_T^{\mathrm{loc}}}
 \leq\frac14\norm{f-g}_{\mathcal X_T^{\mathrm{loc}}}.
\]
Banach's fixed-point theorem gives the asserted local solution. The
same estimates give uniqueness and continuous dependence. Since the
existence time depends only on an upper bound for the $H^1$ norm,
iteration gives \eqref{eq:H1-blowup-alternative}. Applying the same
argument on $[t_0-T,t_0]$ proves the backward statement. Finally,
mass conservation follows first for smooth solutions by multiplying
the equation by $\overline u$, integrating, and taking imaginary
parts. Approximation of $H^1$ data by smooth data and local continuous
dependence yield \eqref{eq:local-mass-conservation} for every mild
$H^1$ solution.

We next prove persistence of $\Sigma$. Suppose $u_0\in\Sigma$ and set
\[
 V_{L,0}=L(t_0)u_0=yu_0+2it_0\partial_yu_0\in L^2.
\]
For the already constructed $H^1$ solution, consider the real-linear
equation
\begin{equation}\label{eq:local-VL-equation}
\begin{aligned}
 (i\partial_t+\partial_y^2)V_L
 &=\beta\left(2\abs u^2V_L-u^2\overline {V_L}\right)
 -\cW_t\left(
 3\abs u^4V_L-2\abs u^2u^2\overline {V_L}\right)
-2it\cW_t'\abs u^4u,\\
  V_L(t_0)&=V_{L,0},
\end{aligned}    
\end{equation}
where $\cW_t'(y)=\cW'(y+\gamma t)$. Let $J$ be a compact subinterval of
the $H^1$ lifespan. The coefficients of $V_L$ and $\overline {V_L}$ belong
to $L^1_tL^\infty_y$, and the inhomogeneous term belongs to
$L^1_tL^2_y$. A short-time Duhamel contraction followed by
concatenation therefore gives a unique $V_L\in C(J;L^2(\R))$. Its $L^2$
energy inequality is
\begin{align}
 \frac{\dd}{\dd t}\norm{V_L(t)}_2
 &\leq C\big[
 \abs\beta\norm{u(t)}_\infty^2
 +\norm{\cW}_\infty\norm{u(t)}_\infty^4
 \big]\norm{V_L(t)}_2
 +2\abs t\norm{\cW'}_\infty
 \norm{u(t)}_\infty^4\norm{u(t)}_2.
 \label{eq:local-VL-energy}
\end{align}
Gronwall's inequality shows that $V_L$ remains finite on every compact
subinterval of the $H^1$ lifespan. More precisely, if $T_+<\infty$ and
$\sup_{t_0\leq t<T_+}\norm{u(t)}_{H^1}<\infty$, then the coefficients and the
source in \eqref{eq:local-VL-equation} are integrable on $[t_0,T_+)$. The
Duhamel formula and \eqref{eq:local-VL-energy} then show that $V_L(t)$ has an
$L^2$ limit as $t\uparrow T_+$.

For smooth data, identities
\eqref{eq:L-commute}--\eqref{eq:L-W} show directly that
$V_L=L(t)u$. For general $u_0\in\Sigma$, choose
$u_0^{(n)}\in\cS(\R)$ converging to $u_0$ in $\Sigma$. Local
continuous dependence gives
$u^{(n)}\to u$ in $C(J;H^1(\R))$ on a common interval $J$. The difference
version of \eqref{eq:local-VL-energy} then gives
$L(t)u^{(n)}\to V_L$ in $C(J;L^2(\R))$. Since
\[
 yu^{(n)}=L(t)u^{(n)}-2it\partial_yu^{(n)},
\]
we obtain, after passing to the limit,
\[
 yu=V_L-2it\partial_yu\in C(J;L^2(\R)).
\]
This proves persistence of $\Sigma$. The corresponding difference
estimates prove continuous dependence in $\Sigma$. If the $H^1$ norm stays
bounded at a finite endpoint, the endpoint argument above and
$yu=V_L-2it\partial_yu$ extend $yu$ continuously in $L^2$ to that endpoint.
Hence the maximal $\Sigma$ and $H^1$ lifespans coincide.

It remains to prove the stationary continuation assertion. The local
theory already established applies to \eqref{eq:Q-stationary} by
taking $\gamma=0$. For smooth solutions, direct differentiation gives
\[
 \frac{\dd}{\dd t}\operatorname{Mass}(Q(t))=0
\]
and
\[
 \frac{\dd}{\dd t}E(Q(t))
 =2\Re\int_\R
 \left(-Q_{xx}+\beta\abs{Q}^2Q-\cW\abs{Q}^4Q\right)
 \overline{Q_t}\dd x=0,
\]
because the expression in parentheses equals $iQ_t$. Both
conservation laws extend to $H^1$ solutions by smooth approximation. The
energy is continuous on $H^1(\R)$.

Set $m_Q=\norm{Q(t)}_2=\norm{Q(t_0)}_2$. The one-dimensional
Gagliardo--Nirenberg inequalities give
\[
 \norm Q_4^4\leq Cm_Q^3\norm{\partial_xQ(t)}_2,
 \qquad
 \norm Q_6^6\leq Cm_Q^4\norm{\partial_xQ(t)}_2^2.
\]
Consequently,
\[
 E(Q(t))\geq
 \left(1-C\norm{\cW}_\infty m_Q^4\right)
 \norm{\partial_xQ(t)}_2^2
 -C\abs\beta m_Q^3\norm{\partial_xQ(t)}_2.
\]
Choose $m_0>0$ so that
\[
 C\norm{\cW}_\infty m_0^4\leq\frac14.
\]
Young's inequality then gives, for $m_Q\leq m_0$,
\[
 E(Q(t))
 \geq\frac12\norm{\partial_xQ(t)}_2^2-C\beta^2m_Q^6,
\]
which is \eqref{eq:energy-coercive-small-mass}. Conservation of energy
therefore bounds $\norm{\partial_xQ(t)}_2$ uniformly throughout the
maximal lifespan. The blow-up alternative rules out either finite
endpoint and proves global existence in both time directions.
\end{proof}

\begin{remark}\label{rem:critical-quintic}
The localized quintic term is mass critical in one dimension. Let
$Q_{\mathrm{gs}}$ be the positive solution of
\[
 -Q_{\mathrm{gs}}''+Q_{\mathrm{gs}}-Q_{\mathrm{gs}}^5=0.
\]
The sharp Gagliardo--Nirenberg inequality \cite{W1983} is
\[
 \norm f_6^6
 \leq\frac{3}{\norm{Q_{\mathrm{gs}}}_2^4}
 \norm f_2^4\norm{f'}_2^2.
\]
Writing $\cW_+=\max\{\cW,0\}$, the stationary energy therefore satisfies
\[
 E(Q)\geq
 \left(1-\frac{\norm{\cW_+}_\infty\norm Q_2^4}
 {\norm{Q_{\mathrm{gs}}}_2^4}\right)\norm{Q_x}_2^2
 -C\abs\beta\norm Q_2^3\norm{Q_x}_2.
\]
When $\norm{\cW_+}_\infty>0$, this identifies the natural critical mass
scale
$\norm{Q_{\mathrm{gs}}}_2\norm{\cW_+}_\infty^{-1/4}$. The smaller number
$m_0$ used in \autoref{prop:local-theory} gives a convenient strict
coercivity constant and is not claimed to be the exact dynamical threshold
for a variable coefficient. If $\cW\leq0$, the quintic energy is
defocusing and this part of the continuation argument requires no mass
restriction.
\end{remark}

\begin{remark}\label{rem:general-localized-power}
The choice of a quintic perturbation is also tied to the vector-field
argument. For $m\in\mathbb N$ and a localized term
$\cW\abs{u}^{2m}u$, application of $L(t)$
produces $2it\cW'\abs{u}^{2m}u$. Under the linear decay rate, its $L^2$
norm is bounded by
\[
 Ct\norm u_\infty^{2m+1}
 \lesssim t^{-(2m-1)/2}.
\]
This is integrable by the present proof when $m\geq2$. The localized cubic
case $m=1$ needs a different vector-field estimate. When $m>2$, the
stationary energy is mass supercritical and the mass-only continuation
used above is no longer available. The quintic case $m=2$ is the unique
power for which both parts of the present argument close under the stated
hypotheses.

Once a corresponding global decay and modified-scattering theorem is
available, the inner calculation itself has a direct extension. For
$\cW\abs{u}^{2m}u$, the tail scale is $T^{-(2m-1)/2}$ and the limiting
shape is
\[
 \int_1^\infty
 \tau^{-(2m+1)/2-i\alpha}e^{i\eta^2\tau}\dd\tau.
\]
Its first nonsmooth term is a nonzero multiple of
$\abs\eta^{2m-1+2i\alpha}$. We restrict our results to $m=2$ because the
global arguments for the other powers are not the same.
\end{remark}

\subsection{The self-similar equation and the leading profile}

The transformation in \eqref{eq:selfsimilar-transform-intro} gives the identity
\begin{equation*}
 \bigl(i\partial_t+\partial_y^2\bigr)
 \left[t^{-1/2}e^{iy^2/(4t)}U(t,y/t)\right]
 =t^{-1/2}e^{iy^2/(4t)}
 \left[
 \left(i\partial_s+s^{-2}\partial_\xi^2\right)U(s,\xi)
 \right]_{(s,\xi)=(t,y/t)}.
\end{equation*}
Hence \eqref{eq:comoving} is equivalent to \eqref{eq:selfsimilar-intro}.

We define
\begin{equation*}
 \fkX:=\left\{A:\R\to\C:\la\xi\ra A\in H^3(\R)\right\},
 \qquad
 \norm A_{\fkX}:=\norm{\la\xi\ra A}_{H^3}.
\end{equation*}
Given $A\in\fkX$, define
\begin{equation*}
 U_A(s,\xi):=A(\xi)e^{-i\beta\abs{A(\xi)}^2\log s}
\end{equation*}
and
\begin{equation*}
 u_A(t,y):=t^{-1/2}e^{iy^2/(4t)}U_A(t,y/t).
\end{equation*}
The intrinsic size of the long-range $H^1$ linearization is measured by
\begin{equation}\label{eq:a-beta-A}
 \mathfrak a_\beta(A)
 :=\sup_{t\geq2}t^{1/2}
 \left(\norm{u_A(t)}_{L^\infty_y}
 +\norm{\partial_yu_A(t)}_{L^\infty_y}\right).
\end{equation}
This quantity is finite and
\begin{equation}\label{eq:a-beta-A-bound}
 \mathfrak a_\beta(A)
 \leq C_\beta\left(\norm A_{\fkX}+\norm A_{\fkX}^3\right).
\end{equation}
Indeed, differentiation of $u_A$ gives the exact formula
\begin{align}
 \mathfrak a_\beta(A)
 &=\norm A_\infty
 +\sup_{t\geq2}\left\|
 \frac{i\xi}{2}A+\frac1tA'
 -\frac{i\beta\log t}{t}(\abs A^2)'A
 \right\|_\infty.
 \notag
\end{align}
Since $\sup_{t\geq2}t^{-1}\log t<\infty$, the weighted Sobolev
embedding and the product rule imply that, for every $R>0$ and all
$A,B\in\fkX$ satisfying
$\norm A_{\fkX}+\norm B_{\fkX}\leq R$,
\begin{equation}\label{eq:a-beta-continuity}
 \left|\mathfrak a_\beta(A)-\mathfrak a_\beta(B)\right|
 \leq C_{\beta,R}\norm{A-B}_{\fkX}.
\end{equation}
The function $U_A$ solves the long-range profile equation
\begin{equation*}
 i\partial_sU_A-\beta s^{-1}\abs{U_A}^2U_A=0.
\end{equation*}
Therefore the residual of $u_A$ in \eqref{eq:comoving} is
\begin{equation}\label{eq:residual-def}
 \cR_A(t,y)=t^{-5/2}e^{iy^2/(4t)}
 \left[
 \partial_\xi^2U_A(t,\xi)
 +\cW\bigl(t(\xi+\gamma)\bigr)\abs{U_A(t,\xi)}^4U_A(t,\xi)
 \right]_{\xi=y/t}.
\end{equation}

For $t\geq2$, set
\begin{equation*}
 \ell_\beta(t)=1+\abs\beta\log t.
\end{equation*}

\begin{lemma}\label{lem:residual}
For every $A\in\fkX$, there is a constant
$C_{\mathrm{res}}=C_{\mathrm{res}}(A,\beta,\cW)$ such that
\begin{equation}\label{eq:residual-H1}
 \norm{\cR_A(t)}_{H^1_y}
 \leq C_{\mathrm{res}}\frac{\ell_\beta(t)^2}{t^2},
 \qquad t\geq2.
\end{equation}
Moreover
\begin{equation}\label{eq:uA-pointwise}
 \norm{u_A(t)}_{L^\infty_y}
 +\norm{\partial_yu_A(t)}_{L^\infty_y}
 \leq \mathfrak a_\beta(A)t^{-1/2},
 \qquad t\geq2.
\end{equation}
\end{lemma}

\begin{proof}
Set $\Theta_A(\xi)=\beta\abs{A(\xi)}^2.$
Direct differentiation gives
\begin{align}
 \partial_\xi U_A
 &=e^{-i\Theta_A\log t}
 \left[A'-i\Theta_A'A\log t\right],
 \notag\\
 \partial_\xi^2U_A
 &=e^{-i\Theta_A\log t}
 \left[A''-2i\Theta_A'A'\log t-i\Theta_A''A\log t
 -(\Theta_A')^2A(\log t)^2\right].
 \label{eq:UA-second-derivative}
\end{align}
Since $\la\xi\ra A\in H^3$, the product rule, the one-dimensional
Sobolev embedding, and a density argument give
\begin{align}
 \norm{\la\xi\ra^m\partial_\xi^kU_A(t)}_{L^2_\xi}
 &\leq C_{A,\beta,m,k}\ell_\beta(t)^k,
 &&m\in\{0,1\},\quad 0\leq k\leq3-m,
 \label{eq:UA-weighted-bounds}\\ \norm{\la\xi\ra^m\partial_\xi^kU_A(t)}_{L^\infty_\xi}
 &\leq C_{A,\beta,m,k}\ell_\beta(t)^k,
 &&(m,k)\in\{(0,0),(1,0),(0,1)\}.
 \notag
\end{align}
In particular, the formula for $\partial_yu_A$ below and the elementary
bound
\[
 \sup_{t\geq2}t^{-1}\log t<\infty
\]
imply \eqref{eq:a-beta-A-bound}.
For the residual write
\[
 B_A(t,\xi)=\partial_\xi^2U_A(t,\xi)
 +\cW\bigl(t(\xi+\gamma)\bigr)\abs{U_A(t,\xi)}^4U_A(t,\xi).
\]
Then \eqref{eq:residual-def} becomes
\[
 \cR_A(t,y)=t^{-5/2}e^{iy^2/(4t)}B_A(t,y/t).
\]
The change of variables $y=t\xi$ and \eqref{eq:UA-weighted-bounds} imply
\begin{equation}\label{eq:residual-L2-detailed}
 \norm{\cR_A(t)}_2=t^{-2}\norm{B_A(t)}_2
 \leq C_{\mathrm{res}} t^{-2}\ell_\beta(t)^2.
\end{equation}
For the derivative, it is preferable to differentiate in the physical variable before estimating. We obtain
\begin{equation}\label{eq:residual-y-derivative}
 \partial_y\cR_A
 =t^{-5/2}e^{iy^2/(4t)}
 \left[\frac{i\xi}{2}B_A+t^{-1}\partial_\xi B_A\right]_{\xi=y/t}.
\end{equation}
The potentially large chain-rule factor in $\partial_\xi\cW(t(\xi+\gamma))$ is canceled by the prefactor $t^{-1}$ in \eqref{eq:residual-y-derivative}:
\[
 t^{-1}\partial_\xi\cW\bigl(t(\xi+\gamma)\bigr)
 =\cW'\bigl(t(\xi+\gamma)\bigr).
\]
Thus no positive power of $t$ is produced. The factor $\xi B_A$ is
controlled by the weighted profile bounds. Since $\abs{U_A}=\abs A$, the
localized terms satisfy
\[
 \norm{\xi\cW\bigl(t(\xi+\gamma)\bigr)\abs{U_A}^4U_A}_2
 \leq\norm{\cW}_\infty\norm A_\infty^4\norm{\xi A}_2, \quad\text{ and}
\]

\[
 \norm{\cW'\bigl(t(\xi+\gamma)\bigr)\abs{U_A}^4U_A}_2
 \leq\norm{\cW'}_\infty\norm A_\infty^4\norm A_2.
\]
The full differentiated localized term obeys
\begin{align*}
 t^{-1}\norm{\partial_\xi\left[
 \cW\bigl(t(\xi+\gamma)\bigr)\abs{U_A}^4U_A\right]}_2& \leq \norm{\cW'}_\infty\norm A_\infty^4\norm A_2
 +Ct^{-1}\norm{\cW}_\infty\norm A_\infty^4
 \norm{\partial_\xi U_A}_2\\
 &
 \leq C_{A,\beta,\cW}\ell_\beta(t).
\end{align*}
The remaining derivatives of $U_A$ are bounded by
\eqref{eq:UA-weighted-bounds}. Consequently,
\[
 \norm{\partial_y\cR_A(t)}_2
 \leq C_{\mathrm{res}} t^{-2}\ell_\beta(t)^2.
\]
Indeed, the only third profile derivative occurs as
$t^{-1}\partial_\xi^3U_A$. Its nominal bound
$t^{-1}\ell_\beta(t)^3$ is $O_\beta(\ell_\beta(t)^2)$.
Together with \eqref{eq:residual-L2-detailed} this proves \eqref{eq:residual-H1}.

Finally,
\[
 \partial_yu_A(t,y)
 =t^{-1/2}e^{iy^2/(4t)}
 \left[\frac{i\xi}{2}U_A+t^{-1}\partial_\xi U_A\right]_{\xi=y/t}.
\]
Both $U_A$ and $\xi U_A$ are uniformly bounded, while the second term has an additional factor $t^{-1}$. This proves \eqref{eq:uA-pointwise}.
\end{proof}

\subsection{Final-state contraction}

Define
\[
 \cN_3(z)=\abs{z}^2z,
 \qquad
 \cN_5(z)=\abs{z}^4z.
\]
Set $u=u_A+w$. Since $\cR_A$ is the residual of $u_A$, the correction satisfies
\begin{equation*}
 \bigl(i\partial_t+\partial_y^2\bigr)w
 =\beta\bigl[\cN_3(u_A+w)-\cN_3(u_A)\bigr]
 -\cW(y+\gamma t)\bigl[\cN_5(u_A+w)-\cN_5(u_A)\bigr]
 -\cR_A.
\end{equation*}
For the equation
\[
 \bigl(i\partial_t+\partial_y^2\bigr)w=F
\]
with vanishing final data, the Duhamel formula is
\begin{equation*}
 w(t)=i\int_t^\infty e^{i(t-s)\partial_y^2}F(s)\dd s.
\end{equation*}

The following estimates retain the decay contributed by each background
factor.

\begin{lemma}\label{lem:nonlinear-H1}
Let $A\in\fkX$, set $a=\mathfrak a_\beta(A)$, and define
$\cW_t(y)=\cW(y+\gamma t)$. There is a constant $C_{\cW}>0$, depending
only on $\norm{\cW}_{W^{1,\infty}}$, such that for every
$t\geq2$ and $w\in H^1$,
\begin{align}
 \norm{\cN_3(u_A+w)-\cN_3(u_A)}_{H^1}
 &\leq Ca^2t^{-1}\norm{w}_{H^1}
 +Cat^{-1/2}\norm{w}_{H^1}^2
 +C\norm{w}_{H^1}^3,
 \label{eq:cubic-H1}\\
 \norm{\cW_t\bigl[\cN_5(u_A+w)-\cN_5(u_A)\bigr]}_{H^1}
 &\leq C_{\cW}a^4t^{-2}\norm{w}_{H^1}
 +C_{\cW}a^3t^{-3/2}\norm{w}_{H^1}^2
 \notag\\
 &\quad+C_{\cW}a^2t^{-1}\norm{w}_{H^1}^3
 +C_{\cW}at^{-1/2}\norm{w}_{H^1}^4
 +C_{\cW}\norm{w}_{H^1}^5.
 \label{eq:quintic-H1}
\end{align}
For fixed $t\geq2$ and
$w_1,w_2\in H^1$, set
\[
 \mu=\norm{w_1}_{H^1}+\norm{w_2}_{H^1}.
\]
Then
\begin{align}
 &\norm{\beta\left[\cN_3(u_A+w_1)-\cN_3(u_A+w_2)\right]
 -\cW_t\left[\cN_5(u_A+w_1)-\cN_5(u_A+w_2)\right]}_{H^1}
 \notag\\
 &\qquad\leq \Lambda(t,\mu)\norm{w_1-w_2}_{H^1},
 \label{eq:nonlinear-difference}
\end{align}
where
\begin{align}
 \Lambda(t,\mu)
 &\leq C\abs\beta\left[a^2t^{-1}+at^{-1/2}\mu+\mu^2\right]
 \notag\\
 &\quad+C_{\cW}\left[a^4t^{-2}+a^3t^{-3/2}\mu
 +a^2t^{-1}\mu^2+at^{-1/2}\mu^3+\mu^4\right].
 \label{eq:Lambda-bound}
\end{align}
\end{lemma}

\begin{proof}
Fix $t\geq2$ and abbreviate
\[
 u_{\mathrm{bg}}=u_A(t),\qquad \rho(t)=at^{-1/2}.
\]
By \eqref{eq:a-beta-A},
\begin{equation}\label{eq:background-W1inf-proof}
 \norm{u_{\mathrm{bg}}}_{W^{1,\infty}}
 =\norm{u_{\mathrm{bg}}}_{L^\infty}
  +\norm{\partial_yu_{\mathrm{bg}}}_{L^\infty}
 \leq \rho(t).
\end{equation}
We repeatedly use the one-dimensional Sobolev embedding
$H^1(\R)\hookrightarrow L^\infty(\R)$ and the product rule.

We first record the resulting mixed product estimate. Suppose that
$\Pi(u_{\mathrm{bg}},\overline {u_{\mathrm{bg}}},w,\overline w)$ is a
monomial containing $n_{\mathrm{bg}}$ factors chosen from
$u_{\mathrm{bg}},\overline {u_{\mathrm{bg}}}$ and
$n_{\mathrm{cor}}\geq1$ factors chosen from
$w,\overline w$. Then
\[
 \norm{\Pi(u_{\mathrm{bg}},\overline {u_{\mathrm{bg}}},w,\overline w)}_{L^2}
 \leq
 \norm{u_{\mathrm{bg}}}_{L^\infty}^{n_{\mathrm{bg}}}
 \norm{w}_{L^\infty}^{n_{\mathrm{cor}}-1}
 \norm{w}_{L^2}.
\]
When one derivative is applied to $\Pi$, it either falls on a background
factor or on a correction factor. The second term below is omitted when
$n_{\mathrm{bg}}=0$:
\[
\begin{aligned}
 \norm{\partial_y\Pi(u_{\mathrm{bg}},\overline {u_{\mathrm{bg}}},w,\overline w)}_{L^2}
 &\leq C_{n_{\mathrm{bg}},n_{\mathrm{cor}}}
 \norm{u_{\mathrm{bg}}}_{L^\infty}^{n_{\mathrm{bg}}}
 \norm{w}_{L^\infty}^{n_{\mathrm{cor}}-1}
 \norm{\partial_yw}_{L^2}
 \\
 &\quad
 +C_{n_{\mathrm{bg}},n_{\mathrm{cor}}}
 \norm{u_{\mathrm{bg}}}_{L^\infty}^{n_{\mathrm{bg}}-1}
 \norm{\partial_yu_{\mathrm{bg}}}_{L^\infty}
 \norm{w}_{L^\infty}^{n_{\mathrm{cor}}-1}
 \norm{w}_{L^2}.
\end{aligned}
\]
Using \eqref{eq:background-W1inf-proof} and
$\norm{w}_{L^\infty}\leq C\norm{w}_{H^1}$, we obtain
\begin{equation}\label{eq:mixed-monomial-proof}
 \norm{\Pi(u_{\mathrm{bg}},\overline {u_{\mathrm{bg}}},w,\overline w)}_{H^1}
 \leq C_{n_{\mathrm{bg}},n_{\mathrm{cor}}}
 \rho(t)^{n_{\mathrm{bg}}}\norm{w}_{H^1}^{n_{\mathrm{cor}}}.
\end{equation}
Complex conjugation causes no change in this estimate.

For the cubic map $\cN_3(z)=\abs{z}^2z=z^2\overline{z}$, the exact
expansion is
\[
\begin{aligned}
 \cN_3(u_{\mathrm{bg}}+w)-\cN_3(u_{\mathrm{bg}})
 =\abs{u_{\mathrm{bg}}}^2w
   +2\Re(\overline {u_{\mathrm{bg}}}w)u_{\mathrm{bg}}
 +2\Re(\overline {u_{\mathrm{bg}}}w)w
   +\abs{w}^2u_{\mathrm{bg}}+\abs{w}^2w.
\end{aligned}
\]
The first line is linear in $w$ and contains two background factors.
The first two terms on the second line are quadratic in $w$ and contain
one background factor, while the final term is cubic in $w$ and contains
no background factor. Therefore \eqref{eq:mixed-monomial-proof} gives
\[
\begin{aligned}
 \norm{\cN_3(u_{\mathrm{bg}}+w)-\cN_3(u_{\mathrm{bg}})}_{H^1}
 &\leq
 C\left(
 \rho(t)^2\norm{w}_{H^1}
 +\rho(t)\norm{w}_{H^1}^2
 +\norm{w}_{H^1}^3
 \right)
 \\
 &=
 C\left(
 a^2t^{-1}\norm{w}_{H^1}
 +at^{-1/2}\norm{w}_{H^1}^2
 +\norm{w}_{H^1}^3
 \right).
\end{aligned}
\]
This proves \eqref{eq:cubic-H1}.

For the quintic map
$\cN_5(z)=\abs{z}^4z=z^3\overline{z}^{\,2}$, the binomial formula yields
\begin{equation}\label{eq:quintic-binomial-proof}
 \cN_5(u_{\mathrm{bg}}+w)-\cN_5(u_{\mathrm{bg}})
 =
 \sum_{\substack{0\leq j\leq3,\;0\leq k\leq2\\j+k\geq1}}
 \binom{3}{j}\binom{2}{k}
 u_{\mathrm{bg}}^{3-j}\overline {u_{\mathrm{bg}}}^{\,2-k}
 w^j\overline w^{\,k}.
\end{equation}
A summand with $j+k=n$ contains precisely $n$ correction factors and
$5-n$ background factors. Applying \eqref{eq:mixed-monomial-proof} to
\eqref{eq:quintic-binomial-proof}, and then grouping the terms according
to $n=1,\ldots,5$, gives
\[
\begin{aligned}
 \norm{\cN_5(u_{\mathrm{bg}}+w)-\cN_5(u_{\mathrm{bg}})}_{H^1}
 \leq C\bigl(
 \rho(t)^4\norm{w}_{H^1}
 +\rho(t)^3\norm{w}_{H^1}^2
 +\rho(t)^2\norm{w}_{H^1}^3
 +\rho(t)\norm{w}_{H^1}^4
 +\norm{w}_{H^1}^5
 \bigr).
\end{aligned}
\]
Since $\cW_t(y)=\cW(y+\gamma t)$, translation invariance gives
\[
 \norm{\cW_t}_{W^{1,\infty}}
 =\norm{\cW}_{W^{1,\infty}}.
\]
Moreover, for every $G\in H^1(\R)$,
\[
\begin{aligned}
 \norm{\cW_tG}_{H^1}
 \leq
 \norm{\cW_tG}_{L^2}
 +\norm{(\partial_y\cW_t)G}_{L^2}
 +\norm{\cW_t\partial_yG}_{L^2}
 \leq
 C\norm{\cW}_{W^{1,\infty}}\norm{G}_{H^1}.
\end{aligned}
\]
Consequently, substituting $\rho(t)=at^{-1/2}$ gives
\[
\begin{aligned}
 \norm{\cW_t\bigl[\cN_5(u_{\mathrm{bg}}+w)
 -\cN_5(u_{\mathrm{bg}})\bigr]}_{H^1}
 &\leq C_{\cW}\bigl(
 a^4t^{-2}\norm{w}_{H^1}
 +a^3t^{-3/2}\norm{w}_{H^1}^2
 +a^2t^{-1}\norm{w}_{H^1}^3
 \\
 &\hspace{48mm}
 +at^{-1/2}\norm{w}_{H^1}^4
 +\norm{w}_{H^1}^5
 \bigr),
\end{aligned}
\]
which proves \eqref{eq:quintic-H1}.

It remains to prove the Lipschitz estimate. Set
\[
 h=w_1-w_2,\qquad
 w_\theta=(1-\theta)w_2+\theta w_1,\qquad
 B_\theta=u_{\mathrm{bg}}+w_\theta,\qquad 0\leq\theta\leq1.
\]
Then
\[
 \norm{w_\theta}_{H^1}
 \leq \norm{w_1}_{H^1}+\norm{w_2}_{H^1}=\mu.
\]
Regarding $\cN_3$ and $\cN_5$ as smooth maps from $\C\simeq\R^2$ to
$\C$, their real Fr\'echet derivatives are
\[
 D\cN_3(z)[h]=2\abs{z}^2h+z^2\overline h, \quad \text{ and}
\]
\[
 D\cN_5(z)[h]=3\abs{z}^4h+2z^3\overline z\,\overline h.
\]
The fundamental theorem of calculus therefore gives
\begin{equation}\label{eq:FTC-nonlinearity-proof}
 \cN_j(u_{\mathrm{bg}}+w_1)-\cN_j(u_{\mathrm{bg}}+w_2)
 =
 \int_0^1D\cN_j(B_\theta)[h]\dd\theta,
 \qquad j\in\{3,5\}.
\end{equation}

Every monomial in $D\cN_3(u_{\mathrm{bg}}+w_\theta)[h]$ contains one factor
chosen from $h,\overline h$ and two further factors chosen from
$u_{\mathrm{bg}},\overline {u_{\mathrm{bg}}},w_\theta,\overline w_\theta$.
The same product-rule
argument used in \eqref{eq:mixed-monomial-proof} shows that
\[
 \norm{D\cN_3(u_{\mathrm{bg}}+w_\theta)[h]}_{H^1}
 \leq
 C\left(\rho(t)^2+\rho(t)\mu+\mu^2\right)\norm{h}_{H^1}.
\]
Likewise, every monomial in $D\cN_5(u_{\mathrm{bg}}+w_\theta)[h]$
contains one factor chosen from $h,\overline h$ and
four further factors. Hence
\[
 \norm{D\cN_5(u_{\mathrm{bg}}+w_\theta)[h]}_{H^1}
 \leq
 C\left(
 \rho(t)^4+\rho(t)^3\mu+\rho(t)^2\mu^2+\rho(t)\mu^3+\mu^4
 \right)\norm{h}_{H^1}.
\]
Integrating these estimates in \eqref{eq:FTC-nonlinearity-proof},
multiplying the cubic estimate by $\abs{\beta}$, and using the
$H^1$ multiplier estimate for $\cW_t$, we obtain
\[
\begin{aligned}
 &\norm{
 \beta\bigl[\cN_3(u_{\mathrm{bg}}+w_1)
 -\cN_3(u_{\mathrm{bg}}+w_2)\bigr]
 -\cW_t\bigl[\cN_5(u_{\mathrm{bg}}+w_1)
 -\cN_5(u_{\mathrm{bg}}+w_2)\bigr]
 }_{H^1}
 \\
 &\quad\leq
 \Bigl\{
 C\abs{\beta}\left(\rho(t)^2+\rho(t)\mu+\mu^2\right)
 +C_{\cW}\left(
 \rho(t)^4+\rho(t)^3\mu+\rho(t)^2\mu^2+\rho(t)\mu^3+\mu^4
 \right)
 \Bigr\}\norm{w_1-w_2}_{H^1}.
\end{aligned}
\]
Finally, substituting $\rho(t)=at^{-1/2}$ gives
\[
\begin{aligned}
 \Lambda(t,\mu)
 \leq
 C\abs{\beta}\left[
 a^2t^{-1}+at^{-1/2}\mu+\mu^2
 \right]
 +C_{\cW}\left[
 a^4t^{-2}+a^3t^{-3/2}\mu+a^2t^{-1}\mu^2
 +at^{-1/2}\mu^3+\mu^4
 \right],
\end{aligned}
\]
which proves \eqref{eq:nonlinear-difference} and
\eqref{eq:Lambda-bound}.
\end{proof}

The $t^{-1}$ cubic linearization contributes
$C\abs\beta\mathfrak a_\beta(A)^2$ to the contraction constant. The next
theorem assumes that this quantity is small.

\begin{theorem}\label{thm:wave-operator}
Let $\beta,\gamma\in\R$, let $\cW$ be real-valued and satisfy
\eqref{eq:W-assumption}, and let $A\in\fkX$. There exist a universal constant
$\eta_0>0$ and a number $m_0=m_0(\norm{\cW}_\infty)>0$ with the following
property. If
\begin{equation}\label{eq:effective-smallness}
 \abs\beta\,\mathfrak a_\beta(A)^2\leq\eta_0,
 \qquad
 \norm A_2\leq m_0,
\end{equation}
then there is a global solution
\[
 u\in C([1,\infty);H^1(\R))
 \cap C^1([1,\infty);H^{-1}(\R))
\]
of \eqref{eq:comoving} satisfying
\begin{equation}\label{eq:wave-operator-est}
 \norm{u(t)-u_A(t)}_{H^1_y}
 +\norm{u(t)-u_A(t)}_{L^\infty_y}
 \leq C_{A,\beta,\cW}\frac{\bigl(1+\abs\beta\log t\bigr)^2}{t}
\end{equation}
for all sufficiently large $t$.
The solution is unique among all mild solutions
$\widetilde u\in C([1,\infty);H^1(\R))$ for which, on some asymptotic tail,
\begin{equation}\label{eq:wave-operator-uniqueness-class}
 \norm{\widetilde u(t)-u_A(t)}_{H^1_y}=O(t^{-3/4}).
\end{equation}
Consequently, $v(t,x)=u(t,x-\gamma t)$ solves \eqref{eq:main} and satisfies
\begin{equation*}
 \norm{v(t)-v_A(t)}_{H^1_x}
 +\norm{v(t)-v_A(t)}_{L^\infty_x}
 \leq C_{A,\beta,\cW}\frac{\bigl(1+\abs\beta\log t\bigr)^2}{t}
\end{equation*}
for all sufficiently large $t$, where $v_A$ is defined in
\eqref{eq:vA-intro}.
The assignment
\begin{equation}\label{eq:wave-operator-map}
 \Omega_+(A):=v(1)\in H^1(\R)
\end{equation}
is therefore a modified wave operator on the profile class determined by
\eqref{eq:effective-smallness}.

For each fixed $\beta$ and $\cW$, condition \eqref{eq:effective-smallness} is satisfied when $\norm A_{\fkX}$ is sufficiently small.
\end{theorem}

\begin{proof}
Fix $T_0\geq2$, to be chosen below, and define
\[
 \mathcal X_{T_0}=\left\{w\in C\bigl([T_0,\infty);H^1(\R)\bigr):
 \norm w_{\mathcal X_{T_0}}
 =\sup_{t\geq T_0}\frac{t}{\ell_\beta(t)^2}
 \norm{w(t)}_{H^1}<\infty\right\}.
\]
For $w\in\mathcal X_{T_0}$, define
\[
 \cM_A(w)(t)
 =\beta\bigl[\cN_3(u_A(t)+w(t))-\cN_3(u_A(t))\bigr]
 -\cW_t\bigl[\cN_5(u_A(t)+w(t))-\cN_5(u_A(t))\bigr].
\]
The fixed-point map is
\begin{equation}\label{eq:fixed-point-map}
 \cT w(t)=i\int_t^\infty e^{i(t-s)\partial_y^2}
 \left[\cM_A(w)(s)-\cR_A(s)\right]\dd s.
\end{equation}
The free group is unitary on $H^1$. By \autoref{lem:residual},
\begin{equation}\label{eq:source-X}
 \sup_{t\geq T_0}\frac{t}{\ell_\beta(t)^2}
 \int_t^\infty\norm{\cR_A(s)}_{H^1}\dd s
 \leq C_{\mathrm{res}}.
\end{equation}

We record the elementary tail estimates used repeatedly below. After increasing $T_0$, one has
\begin{align}
 \int_t^\infty\frac{\ell_\beta(s)^2}{s^2}\dd s
 &\leq C\frac{\ell_\beta(t)^2}{t},
 \label{eq:tail-L2-s2}\\
 \int_t^\infty\frac{\ell_\beta(s)^2}{s^3}\dd s
 &\leq C\frac{\ell_\beta(t)^2}{t^2},
\quad
 \int_t^\infty\frac{\ell_\beta(s)^4}{s^2}\dd s
 \leq C\frac{\ell_\beta(t)^4}{t}.
 \notag
\end{align}
For example, integration by parts in the first integral gives the boundary
term $\ell_\beta(t)^2/t$ and a term bounded by
$C\ell_\beta(t)/t$. The remaining
estimates follow similarly. In particular,
\begin{equation*}
 \sup_{t\geq T_0}\frac{t}{\ell_\beta(t)^2}
 \int_t^\infty\frac{\ell_\beta(s)^2}{s^3}\dd s\leq \frac{C}{T_0},
\end{equation*}
whereas the corresponding normalized cubic integral in \eqref{eq:tail-L2-s2} is only bounded, not small. This is the quantitative distinction between the long-range and short-range terms.

For $R_{\mathrm{tail}}>0$, let
$\mathbb B_{\mathrm{tail}}
=\{w\in\mathcal X_{T_0}:\norm w_{\mathcal X_{T_0}}\leq R_{\mathrm{tail}}\}$.
For $w_1,w_2\in \mathbb B_{\mathrm{tail}}$, set
$\mu(s)=\norm{w_1(s)}_{H^1}+\norm{w_2(s)}_{H^1}$. Then
\begin{equation}\label{eq:mu-ball}
 \mu(s)\leq 2R_{\mathrm{tail}}\frac{\ell_\beta(s)^2}{s}.
\end{equation}
Insert \eqref{eq:mu-ball} into \eqref{eq:Lambda-bound} and use
\[
 \norm{w_1(s)-w_2(s)}_{H^1}
 \leq \frac{\ell_\beta(s)^2}{s}
 \norm{w_1-w_2}_{\mathcal X_{T_0}}.
\]
The only long-range coefficient is the cubic linearization
$C\abs\beta\mathfrak a_\beta(A)^2s^{-1}$. It contributes
\[
 C\abs\beta\mathfrak a_\beta(A)^2
 \sup_{t\geq T_0}\frac{t}{\ell_\beta(t)^2}
 \int_t^\infty\frac{\ell_\beta(s)^2}{s^2}\dd s
 \leq C\abs\beta\mathfrak a_\beta(A)^2.
\]
Every other contribution gains an additional negative power of $s$. For example, the quadratic cubic term is bounded by
\begin{align*}
 &C_{A,R_{\mathrm{tail}}}\sup_{t\geq T_0}\frac{t}{\ell_\beta(t)^2}
 \int_t^\infty \frac{\ell_\beta(s)^4}{s^{5/2}}\dd s\leq C_{A,R_{\mathrm{tail}}}\sup_{t\geq T_0}
 \frac{\ell_\beta(t)^2}{t^{1/2}},
\end{align*}
and the pure cubic correction term is bounded by
\[
 C_{R_{\mathrm{tail}}}\sup_{t\geq T_0}\frac{t}{\ell_\beta(t)^2}
 \int_t^\infty\frac{\ell_\beta(s)^6}{s^3}\dd s
 \leq C_{R_{\mathrm{tail}}}\sup_{t\geq T_0}\frac{\ell_\beta(t)^4}{t}.
\]
The quintic linearization gives the bound
$C_{\cW}\mathfrak a_\beta(A)^4/T_0$, and its
higher-order terms are smaller still. Since
$\ell_\beta(t)^m t^{-\theta}\to0$ for every fixed $m$ and $\theta>0$,
the smallness condition \eqref{eq:effective-smallness} makes the long-range
cubic contribution smaller than $1/4$, and then $T_0$ can be chosen so
that the sum of all short-range contributions is below $1/4$. Therefore
\begin{equation*}
 \norm{\cT w_1-\cT w_2}_{\mathcal X_{T_0}}
 \leq\frac12\norm{w_1-w_2}_{\mathcal X_{T_0}},
 \qquad w_1,w_2\in \mathbb B_{\mathrm{tail}}.
\end{equation*}

To verify invariance of the ball, enlarge $C_{\mathrm{res}}$ if necessary
so that it bounds \eqref{eq:source-X}, and take
$R_{\mathrm{tail}}=4C_{\mathrm{res}}$. The linear cubic contribution is controlled by the
same effective smallness condition, while all terms at least quadratic in
$w$ contain one of the additional powers of $s^{-1/2}$ displayed in
\eqref{eq:cubic-H1}--\eqref{eq:quintic-H1}. Enlarging $T_0$ after fixing
$R_{\mathrm{tail}}$, we obtain
\[
 \norm{\cT w}_{\mathcal X_{T_0}}
 \leq C_{\mathrm{res}}+\frac12\norm w_{\mathcal X_{T_0}}
 \leq C_{\mathrm{res}}+2C_{\mathrm{res}}
 <4C_{\mathrm{res}}=R_{\mathrm{tail}},
 \qquad w\in \mathbb B_{\mathrm{tail}}.
\]
Banach's fixed-point theorem therefore yields a unique
$w\in \mathbb B_{\mathrm{tail}}$ solving
\eqref{eq:fixed-point-map}. The integral converges in $H^1$ for every
$t\geq T_0$. Moreover,
\[
 w\in C\bigl([T_0,\infty);H^1(\R)\bigr)
 \cap C^1\bigl([T_0,\infty);H^{-1}(\R)\bigr).
\]
The definition of $\mathcal X_{T_0}$ and the embedding
$H^1(\R)\hookrightarrow L^\infty(\R)$ give
\eqref{eq:wave-operator-est} on the tail.

It remains to continue the tail solution to the fixed time $t=1$. Set
\[
 v(t,x)=u(t,x-\gamma t),
 \qquad
 Q(t,x)=e^{i\gamma x/2-i\gamma^2t/4}v(t,x).
\]
Then $Q$ solves \eqref{eq:Q-stationary}. By
\eqref{eq:local-mass-conservation}, the $L^2$ norm of the tail solution
is constant. Since
\[
 \norm{u(t)-u_A(t)}_2\longrightarrow0,
 \qquad
 \norm{u_A(t)}_2=\norm A_2,
\]
and both changes of variables above are unitary on $L^2$, we have
\[
 \norm{Q(t)}_2=\norm A_2\leq m_0.
\]
The small-mass continuation assertion in
\autoref{prop:local-theory} therefore extends $Q$ uniquely through every
finite time, in particular backward to $t=1$. Moreover,
\[
 \norm{\partial_yu(t)}_2=\norm{\partial_xv(t)}_2
 \leq \norm{\partial_xQ(t)}_2+\frac{\abs\gamma}{2}\norm{Q(t)}_2,
\]
so the conserved-energy bound for $Q$ also rules out $H^1$ blow-up after
transforming back. Thus the inverse transformations give
the asserted global solutions $u$ and $v$ and make
\eqref{eq:wave-operator-map} well defined. The uniqueness assertion in
\eqref{eq:wave-operator-uniqueness-class} follows from
\autoref{lem:tail-uniqueness} below.
\end{proof}

We next prove uniqueness in a class larger than the contraction ball.

\begin{lemma}\label{lem:tail-uniqueness}
Assume that the universal constant $\eta_0$ in
\autoref{thm:wave-operator} has been chosen sufficiently small. Let
$\beta,\gamma\in\R$, let $\cW$ be real-valued and satisfy
\eqref{eq:W-assumption}, and let
$A\in\fkX$ satisfy
$\abs\beta\mathfrak a_\beta(A)^2\leq\eta_0$. For some $T_*\geq2$, let
$u^{(1)}$ and $u^{(2)}$
be two mild solutions in $C([T_*,\infty);H^1(\R))$ of
\eqref{eq:comoving} with the same final profile $A$, and suppose that for
some finite $C_*$,
\begin{equation}\label{eq:tail-uniqueness-class}
 \norm{u^{(j)}(t)-u_A(t)}_{H^1_y}
 \leq C_* t^{-3/4},
 \qquad t\geq T_*,\quad j=1,2.
\end{equation}
Then
\[
 u^{(1)}(t)=u^{(2)}(t),\qquad t\geq T_*.
\]
\end{lemma}

\begin{proof}
Set $d=u^{(1)}-u^{(2)}$ and $w_j=u^{(j)}-u_A$. For $T_1$ sufficiently
large, the difference nonlinearity belongs to
$L^1([T_1,\infty);H^1)$. The $H^1$ norm of its cubic linear part is
$O(s^{-1})\norm{d(s)}_{H^1}=O(s^{-7/4})$, and all remaining terms are
integrable with faster decay. Integrating from $t$ to a
finite $T_{\mathrm{fin}}$ and then sending $T_{\mathrm{fin}}\to\infty$,
using $d(T_{\mathrm{fin}})\to0$ in $H^1$, gives
\[
 d(t)=i\int_t^\infty e^{i(t-s)\partial_y^2}
 \Bigl\{\beta\bigl[\cN_3(u_A+w_1)-\cN_3(u_A+w_2)\bigr]
 -\cW_s\bigl[\cN_5(u_A+w_1)-\cN_5(u_A+w_2)\bigr]\Bigr\}\dd s,
\]
where $\cW_s(y)=\cW(y+\gamma s)$. Applying the difference estimates
\eqref{eq:nonlinear-difference}--\eqref{eq:Lambda-bound} and
\eqref{eq:tail-uniqueness-class} gives
\[
 \mu(s):=\norm{w_1(s)}_{H^1}+\norm{w_2(s)}_{H^1}
 \leq 2C_*s^{-3/4}.
\]
Define
\[
 Y_{T_1}:=\left\{h\in C\bigl([T_1,\infty);H^1(\R)\bigr):
 \sup_{t\geq T_1}t^{3/4}\norm{h(t)}_{H^1}<\infty\right\},
 \qquad
 \norm h_{Y_{T_1}}
 :=\sup_{t\geq T_1}t^{3/4}\norm{h(t)}_{H^1}.
\]
The bound \eqref{eq:tail-uniqueness-class} implies $d\in Y_{T_1}$.
The linearized cubic term contributes
\[
 C\mathfrak a_\beta(A)^2\abs\beta
 \sup_{t\geq T_1}t^{3/4}
 \int_t^\infty s^{-1}s^{-3/4}\dd s\,\norm d_{Y_{T_1}}
 =\frac{4}{3}C\mathfrak a_\beta(A)^2\abs\beta
 \norm d_{Y_{T_1}}.
\]
Choose the universal constant $\eta_0$ in \eqref{eq:effective-smallness}
small enough that this contribution is at most
$\frac14\norm d_{Y_{T_1}}$. Every remaining term in
\eqref{eq:Lambda-bound} contains either an extra factor $s^{-1}$ from the
localized quintic linearization or a positive power of $\mu(s)$. Its
contribution is $o_{T_1\to\infty}(1)\norm d_{Y_{T_1}}$. Increasing $T_1$
gives
\[
 \norm d_{Y_{T_1}}\leq\frac12\norm d_{Y_{T_1}},
\]
so $d=0$ on a common tail $[T_1,\infty)$. Local $H^1$ uniqueness then propagates the equality backward from $T_1$ to the original time $T_*$.
\end{proof}

\begin{corollary}\label{cor:wave-operator-properties}
Define the open admissible profile set
\[
\fkX_{\mathrm{adm}}
 :=\left\{A\in\fkX:
 \abs\beta\mathfrak a_\beta(A)^2<\eta_0,\ 
 \norm A_2<m_0\right\}.
\]
The modified wave operator
\[
 \Omega_+:\fkX_{\mathrm{adm}}\longrightarrow H^1(\R)
\]
is continuous, injective, and gauge equivariant:
\[
 \Omega_+(e^{i\vartheta}A)=e^{i\vartheta}\Omega_+(A),
 \qquad \vartheta\in\R.
\]
\end{corollary}

\begin{proof}
Equation \eqref{eq:a-beta-continuity} and the continuity of the $L^2$ norm
show first that $\fkX_{\mathrm{adm}}$ is open. Fix
$A\in\fkX_{\mathrm{adm}}$. The strict inequalities at $A$ and
\eqref{eq:a-beta-continuity} give $\varepsilon_A>0$ and a bounded ball
$\mathcal U$ about $A$ such that
\[
 \sup_{B\in\mathcal U}\abs\beta\,\mathfrak a_\beta(B)^2
 \leq\eta_0-\varepsilon_A,
 \qquad
 \sup_{B\in\mathcal U}\norm B_2\leq m_0-\varepsilon_A.
\]
The construction in
\autoref{thm:wave-operator} permits one starting time $T_0$, one contraction
ball, and one contraction constant for every $B\in\mathcal U$. Let
$u^{(B)}$ denote the corresponding solution and set
$w_B=u^{(B)}-u_B^{\mathrm{as}}$, where
$u_B^{\mathrm{as}}$ is obtained from the formula for $u_A$ by replacing
$A$ with $B$. Define $\cM_B$ by the formula for $\cM_A$ in the proof of
\autoref{thm:wave-operator}, again with $A$ replaced by $B$.

The explicit formulas for the background and the residual give a constant
$C_{\mathcal U}>0$ such that, for $B,C\in\mathcal U$ and $t\geq T_0$,
\begin{align}
 \norm{u_B^{\mathrm{as}}(t)-u_C^{\mathrm{as}}(t)}_{W^{1,\infty}}
 &\leq C_{\mathcal U}t^{-1/2}\ell_\beta(t)
 \norm{B-C}_{\fkX},
 \label{eq:background-parameter-difference}\\
 \norm{\cR_B(t)-\cR_C(t)}_{H^1}
 &\leq C_{\mathcal U}t^{-2}\ell_\beta(t)^3
 \norm{B-C}_{\fkX}.
 \label{eq:residual-parameter-difference}
\end{align}
To compare the corrections, define
\begin{align*}
 \mathcal Y_{T_0}
 &:=\left\{h\in C([T_0,\infty);H^1(\R)):
 \norm h_{\mathcal Y_{T_0}}<\infty\right\},\\
 \norm h_{\mathcal Y_{T_0}}
 &:=\sup_{t\geq T_0}\frac{t}{\ell_\beta(t)^3}\norm{h(t)}_{H^1}.
\end{align*}
Every correction in the common $\mathcal X_{T_0}$ ball belongs to
$\mathcal Y_{T_0}$. If $w$ lies in this common ball, the polynomial
identities used in \autoref{lem:nonlinear-H1} and
\eqref{eq:background-parameter-difference} give
\begin{equation}\label{eq:nonlinearity-parameter-difference}
 \norm{\cM_B(w)(t)-\cM_C(w)(t)}_{H^1}
 \leq C_{\mathcal U}t^{-2}\ell_\beta(t)^3
 \norm{B-C}_{\fkX}.
\end{equation}
Indeed, every cubic term on the left contains one factor from
$u_B^{\mathrm{as}}-u_C^{\mathrm{as}}$ and at least one factor from $w$.
The quintic terms have two additional decaying factors. We split
\[
 \cM_B(w_B)-\cM_C(w_C)
 =\bigl[\cM_B(w_B)-\cM_B(w_C)\bigr]
 +\bigl[\cM_B(w_C)-\cM_C(w_C)\bigr].
\]
Subtracting the two final-state integral equations and using
\eqref{eq:residual-parameter-difference}--
\eqref{eq:nonlinearity-parameter-difference}, together with the difference
estimate in \autoref{lem:nonlinear-H1}, gives a constant
$0\leq\vartheta_*<1$ such that, after increasing $C_{\mathcal U}$,
\begin{equation}\label{eq:wave-operator-parameter-stability}
 \norm{w_B-w_C}_{\mathcal Y_{T_0}}
 \leq \vartheta_*\norm{w_B-w_C}_{\mathcal Y_{T_0}}
 +C_{\mathcal U}\norm{B-C}_{\fkX}.
\end{equation}
Here the long-range term is controlled by
\[
 \sup_{t\geq T_0}\frac{t}{\ell_\beta(t)^3}
 \int_t^\infty s^{-2}\ell_\beta(s)^3\dd s<\infty.
\]
The effective smallness condition makes its coefficient strictly smaller
than one, and increasing $T_0$ controls every term with additional decay.
The terms containing the profile difference are bounded by the same
integral. This proves
\eqref{eq:wave-operator-parameter-stability}.

It follows that $w_B(T_0)\to w_A(T_0)$ in $H^1$ when $B\to A$ in
$\fkX$. The background is continuous at the fixed time $T_0$, so
$u^{(B)}(T_0)\to u^{(A)}(T_0)$ in $H^1$. Finite-time continuous dependence from
\autoref{prop:local-theory} propagates the convergence to time $1$. The
fixed translation between $u(1)$ and $v(1)$ proves continuity of
$\Omega_+$.

Suppose that $\Omega_+(A)=\Omega_+(B)$. Uniqueness for the Cauchy problem
gives one solution with both asymptotic profiles. After the self-similar
change of variables, the two estimates \eqref{eq:wave-operator-est} imply
\[
 \norm{Ae^{-i\beta\abs A^2\log t}
 -Be^{-i\beta\abs B^2\log t}}_2\longrightarrow0.
\]
The reverse triangle inequality first gives $\abs A=\abs B$ almost
everywhere. The logarithmic phase factors then agree, and the same limit
reduces to $\norm{A-B}_2=0$. Hence $\Omega_+$ is injective.

Finally, multiplication by a constant phase preserves the equation and
the final-state asymptotic. Gauge covariance follows from the uniqueness in
\autoref{lem:tail-uniqueness}.
\end{proof}

\section{Small-data forward modified scattering}\label{sec:forward-scattering}

The direct self-similar high-derivative argument used for constant-coefficient cubic--quintic equations does not extend uniformly to \eqref{eq:comoving}, since
\[
 \partial_\xi^j\left[s^{-2}\cW\bigl(s(\xi+\gamma)\bigr)\right]
 =s^{j-2}\cW^{(j)}\bigl(s(\xi+\gamma)\bigr),
\]
which grows near $\xi=-\gamma$ when $j\geq3$. We instead adapt the wave-packet testing method of Ifrim and Tataru \cite{IT2015,IT2024} to the quadratic Schr\"odinger dispersion. We rederive the packet residual, comparison estimates, and effective ordinary differential equation in the present normalization, with a separate estimate for the localized quintic term.

We use the weighted space $\Sigma$ defined by \eqref{def:Sigma} and the Galilean vector field $L(t)$
introduced in \autoref{subsec:local-theory}. The commutator and
gauge-covariant identities needed below are
\eqref{eq:L-commute}--\eqref{eq:L-W}.

\subsection{Energy and vector-field estimates}

By \autoref{prop:local-theory}, mass is conserved throughout the
$H^1$ lifespan:
\[
 \norm{u(t)}_2=\norm{u(1)}_2.
\]

\begin{lemma}\label{lem:energy-packet-section}
Let $T\geq1$, $D_{\mathrm{boot}}\geq1$, and $\varepsilon>0$. Let $u$ be a
smooth solution to
\eqref{eq:comoving} on $[1,T]$. Assume
\begin{equation}\label{eq:bootstrap-Linf}
 \norm{u(t)}_{L^\infty_y}\leq D_{\mathrm{boot}}\varepsilon t^{-1/2},
 \qquad 1\leq t\leq T.
\end{equation}
If $\norm{u(1)}_\Sigma\leq\varepsilon$, then
\begin{equation}\label{eq:energy-growth-packet}
 \norm{u(t)}_{H^1_y}+\norm{Lu(t)}_{L^2_y}
 \leq C\varepsilon t^{\delta_E},
 \qquad
 \delta_E=C_E\abs\beta D_{\mathrm{boot}}^2\varepsilon^2,
 \qquad 1\leq t\leq T,
\end{equation}
provided
\begin{equation}\label{eq:energy-lemma-smallness}
 \norm{\cW}_\infty D_{\mathrm{boot}}^4\varepsilon^4
 +\norm{\cW'}_2D_{\mathrm{boot}}^5\varepsilon^4\leq1.
\end{equation}
The constants $C$ and $C_E$ are independent of
$D_{\mathrm{boot}}$, $\varepsilon$, and $\gamma$. The integrable quintic
contributions affect only the time-independent prefactor, not the exponent
$\delta_E$.
\end{lemma}

\begin{proof}
Mass is conserved because both nonlinear coefficients are real. Differentiate \eqref{eq:comoving} once in $y$. The standard $L^2$ energy identity gives
\begin{align}
 \frac{\dd}{\dd t}\norm{\partial_yu(t)}_2
 &\leq C\abs\beta\norm u_\infty^2\norm{\partial_yu}_2
 +C\norm{\cW}_\infty\norm u_\infty^4\norm{\partial_yu}_2
 +C\norm{\cW'}_2\norm u_\infty^5.
 \label{eq:dy-energy-detailed}
\end{align}
The last term follows from
\[
 \norm{\cW'(y+\gamma t)\abs{u}^4u}_2
 \leq\norm{\cW'}_2\norm u_\infty^5.
\]
Under \eqref{eq:bootstrap-Linf}, the three contributions in \eqref{eq:dy-energy-detailed} have sizes
\[
 C\abs\beta D_{\mathrm{boot}}^2\varepsilon^2t^{-1},
 \qquad
 CD_{\mathrm{boot}}^4\varepsilon^4t^{-2},
 \qquad
 C\norm{\cW'}_2D_{\mathrm{boot}}^5\varepsilon^5t^{-5/2}.
\]
Apply $L$ to \eqref{eq:comoving}. Using \eqref{eq:L-commute}, \eqref{eq:L-cubic}, \eqref{eq:L-quintic}, and \eqref{eq:L-W}, we obtain
\begin{align*}
 (i\partial_t+\partial_y^2)Lu
 &=\beta\left(2\abs{u}^2Lu-u^2\overline{Lu}\right)\\
 &\quad-\cW(y+\gamma t)
 \left(3\abs{u}^4Lu-2\abs{u}^2u^2\overline{Lu}\right)
 -2it\cW'(y+\gamma t)\abs{u}^4u.
\end{align*}
Therefore
\begin{align}
 \frac{\dd}{\dd t}\norm{Lu(t)}_2
 &\leq 3\abs\beta\norm u_\infty^2\norm{Lu}_2
 +5\norm{\cW}_\infty\norm u_\infty^4\norm{Lu}_2
 +2t\norm{\cW'}_2\norm u_\infty^5.
 \label{eq:L-energy-detailed}
\end{align}
The two differential inequalities hold for almost every $t$. Equivalently,
one may regularize each norm by
$(\norm{\cdot}_2^2+\nu^2)^{1/2}$ and let $\nu\downarrow0$.
The last term is bounded by
$C\norm{\cW'}_2D_{\mathrm{boot}}^5\varepsilon^5t^{-3/2}$, which is integrable. Since
\[
 \norm{L(1)u(1)}_2
 \leq\norm{yu(1)}_2+2\norm{\partial_yu(1)}_2
 \leq3\varepsilon,
\]
Gronwall's inequality applied to \eqref{eq:dy-energy-detailed} and
\eqref{eq:L-energy-detailed} gives
\[
 \norm{u(t)}_{H^1}+\norm{Lu(t)}_2
 \leq C\left(\varepsilon+\norm{\cW'}_2D_{\mathrm{boot}}^5\varepsilon^5\right)
 \exp\left(C\norm{\cW}_\infty D_{\mathrm{boot}}^4\varepsilon^4\right)
 t^{C\abs\beta D_{\mathrm{boot}}^2\varepsilon^2}.
\]
Condition \eqref{eq:energy-lemma-smallness} bounds the prefactor by a
constant multiple of $\varepsilon$. This proves
\eqref{eq:energy-growth-packet} after fixing $C_E$ large enough. In
particular, $\delta_E=0$ when $\beta=0$.
\end{proof}

The vector field also gives a useful pointwise estimate. If
\begin{equation}\label{eq:g-osc}
 g(t,y)=e^{-iy^2/(4t)}u(t,y),
\end{equation}
then
\begin{equation}\label{eq:g-L-relation}
 \partial_yg(t,y)=\frac{1}{2it}e^{-iy^2/(4t)}Lu(t,y).
\end{equation}
Consequently the one-dimensional Sobolev inequality gives
\begin{equation}\label{eq:pseudoconformal-Sobolev}
 \norm{u(t)}_\infty^2
 \leq t^{-1}\norm{u(t)}_2\norm{Lu(t)}_2.
\end{equation}
The estimate \eqref{eq:pseudoconformal-Sobolev} alone loses the small power $t^{\delta_E/2}$. The wave-packet coefficient introduced below removes this loss.

\subsection{Explicit Schr\"odinger wave packets}

Fix a real function $\chi\in C_c^\infty(\R)$ with
\begin{equation*}
 \int_\R\chi(z)\dd z=1.
\end{equation*}
For $t\geq1$ and $\xi\in\R$, define
\begin{equation}\label{eq:packet-def-explicit}
 z=z(t,y,\xi)=\frac{y-\xi t}{\sqrt{2t}},
 \qquad
 \Psi_\xi(t,y)=\frac1{\sqrt2}\chi(z)e^{iy^2/(4t)}.
\end{equation}
The cutoff confines $\Psi_\xi$ to a tube of width comparable to $t^{1/2}$ around the ray $y=\xi t$. The oscillatory factor is centered at frequency $\xi/2$, as required by the group-velocity relation for the quadratic Schr\"odinger phase.

Define the packet coefficient by
\begin{equation}\label{eq:Gamma-def}
 \Gamma(t,\xi)=\int_\R u(t,y)\overline{\Psi_\xi(t,y)}\dd y.
\end{equation}

The normalization in \eqref{eq:packet-def-explicit} makes the packet
coefficient agree asymptotically with the self-similar amplitude. Let $A$ be
bounded and continuous. If
\[
 u(t,y)=t^{-1/2}e^{iy^2/(4t)}A(y/t),
\]
then the change of variables $y=t\xi+\sqrt{2t}z$ gives
\[
 \Gamma(t,\xi)=\int A\left(\xi+\sqrt{2/t}\,z\right)\chi(z)\dd z,
\]
which converges to $A(\xi)$ by dominated convergence. Thus $\Gamma$ is a
smoothed self-similar profile at resolution $t^{-1/2}$.

\begin{lemma}\label{lem:packet-residual}
The packet in \eqref{eq:packet-def-explicit} satisfies
\begin{equation}\label{eq:packet-residual-exact}
 (i\partial_t+\partial_y^2)\Psi_\xi
 =\frac{e^{iy^2/(4t)}}{2\sqrt2\,t}K_\chi(z),
\end{equation}
where
\begin{equation*}
 K_\chi(z)=\chi''(z)+i\left(z\chi'(z)+\chi(z)\right).
\end{equation*}
Moreover
\begin{equation}\label{eq:Kchi-zero-mean}
 \int_\R K_\chi(z)\dd z=0.
\end{equation}
\end{lemma}

\begin{proof}
Write $\Phi_{\mathrm{Sch}}(t,y)=y^2/(4t)$. For a scalar amplitude
$b_{\mathrm{pkt}}$, direct differentiation gives
\begin{equation}\label{eq:conjugated-P}
 (i\partial_t+\partial_y^2)(e^{i\Phi_{\mathrm{Sch}}}b_{\mathrm{pkt}})
 =e^{i\Phi_{\mathrm{Sch}}}\left[
 i\left(\partial_tb_{\mathrm{pkt}}+\frac yt\partial_yb_{\mathrm{pkt}}
 +\frac{b_{\mathrm{pkt}}}{2t}\right)
 +\partial_y^2b_{\mathrm{pkt}}
 \right].
\end{equation}
For $b_{\mathrm{pkt}}=2^{-1/2}\chi(z)$, one has
\[
 \partial_tz=-\frac{\xi}{\sqrt{2t}}-\frac{z}{2t},
 \qquad
 \partial_yz=\frac1{\sqrt{2t}}.
\]
Substitution into \eqref{eq:conjugated-P} yields \eqref{eq:packet-residual-exact}. Finally,
\[
 \int\chi''=0,
 \qquad
 \int z\chi'=-\int\chi,
\]
which proves \eqref{eq:Kchi-zero-mean}.
\end{proof}

Using \eqref{eq:g-osc} and the change of variables $y=t\xi+\sqrt{2t}\,z$, we obtain the exact averaging identity
\begin{equation}\label{eq:Gamma-average}
 \Gamma(t,\xi)
 =\sqrt t\int_\R g\left(t,t\xi+\sqrt{2t}\,z\right)\chi(z)\dd z.
\end{equation}

\begin{lemma}\label{lem:packet-comparison}
Let $t\geq1$ and suppose that $u(t)\in\Sigma$. Define $\Gamma(t)$ by
\eqref{eq:Gamma-def}. Then
\begin{align}
 \norm{\Gamma(t)}_{L^\infty_\xi}
 &\leq C\sqrt t\norm{u(t)}_{L^\infty_y},
 \notag\\
 \norm{\Gamma(t)}_{L^2_\xi}
 &\leq C\norm{u(t)}_{L^2_y},
 \label{eq:Gamma-L2}\\
 \norm{\partial_\xi\Gamma(t)}_{L^2_\xi}
 &\leq C\norm{Lu(t)}_{L^2_y},
 \label{eq:Gamma-derivative}\\
 \norm{u(t,t\xi)-t^{-1/2}e^{it\xi^2/4}\Gamma(t,\xi)}_{L^\infty_\xi}
 &\leq Ct^{-3/4}\norm{Lu(t)}_{L^2_y},
 \label{eq:packet-pointwise}\\
 \norm{u(t,t\xi)-t^{-1/2}e^{it\xi^2/4}\Gamma(t,\xi)}_{L^2_\xi}
 &\leq Ct^{-1}\norm{Lu(t)}_{L^2_y}.
 \label{eq:packet-L2xi}
\end{align}
\end{lemma}

\begin{proof}
The first bound follows directly from \eqref{eq:Gamma-def} because $\norm{\Psi_\xi}_1\leq C\sqrt t$. The averaging identity \eqref{eq:Gamma-average} and Minkowski's inequality give
\[
 \norm{\Gamma(t)}_{L^2_\xi}
 \leq\sqrt t\int\abs{\chi(z)}
 \norm{g(t,t\xi+\sqrt{2t}\,z)}_{L^2_\xi}\dd z
 \leq C\norm{u(t)}_2.
\]
Differentiating \eqref{eq:Gamma-average} in $\xi$ and using \eqref{eq:g-L-relation} yields
\[
 \partial_\xi\Gamma(t,\xi)
 =\frac{\sqrt t}{2i}
 \int e^{-i(t\xi+\sqrt{2t}z)^2/(4t)}
 Lu(t,t\xi+\sqrt{2t}z)\chi(z)\dd z.
\]
Minkowski's inequality gives \eqref{eq:Gamma-derivative}.

Set
\[
 g_\chi(t,\xi)=t^{-1/2}\Gamma(t,\xi).
\]
From \eqref{eq:Gamma-average},
\begin{equation}\label{eq:gchi-average}
 g_\chi(t,\xi)=\int g(t,t\xi+\sqrt{2t}z)\chi(z)\dd z.
\end{equation}
The fundamental theorem of calculus and Cauchy--Schwarz imply
\[
 \abs{g(t,x)-g(t,x_0)}
 \leq\abs{x-x_0}^{1/2}\norm{\partial_yg(t)}_2.
\]
With \eqref{eq:g-L-relation} and $x_0=t\xi$, this gives
\begin{equation}\label{eq:g-average-pointwise}
 \abs{g(t,t\xi)-g_\chi(t,\xi)}
 \leq Ct^{-3/4}\norm{Lu(t)}_2.
\end{equation}
Multiplication by the central phase gives \eqref{eq:packet-pointwise}.

For the $L^2_\xi$ estimate, translation and the fundamental theorem of calculus give, for each fixed $z$,
\begin{align*}
 &\norm{g(t,t\xi+\sqrt{2t}z)-g(t,t\xi)}_{L^2_\xi}\\
 &\qquad=t^{-1/2}
 \norm{g(t,\cdot+\sqrt{2t}z)-g(t,\cdot)}_{L^2_y}
 \leq C\abs z\norm{\partial_yg(t)}_2
 \leq Ct^{-1}\abs z\norm{Lu(t)}_2.
\end{align*}
Average this inequality against $\chi$ and use \eqref{eq:gchi-average}. This proves \eqref{eq:packet-L2xi}.
\end{proof}

\subsection{Testing the equation and extracting the resonant cubic term}

The next lemma gives the effective dynamics of the packet coefficient and separates the localized quintic error from the cubic resonant term.

\begin{lemma}\label{lem:Gamma-ODE}
Let $T\geq1$, and let $u$ be a smooth solution in $\Sigma$ of
\eqref{eq:comoving} on $[1,T]$. Then
\begin{equation}\label{eq:Gamma-ODE}
 \partial_t\Gamma(t,\xi)
 =-i\beta t^{-1}\abs{\Gamma(t,\xi)}^2\Gamma(t,\xi)
 +\mathcal E_\Gamma(t,\xi),
\end{equation}
where
\begin{align*}
 \mathcal E_\Gamma
 &=\mathcal E_{\mathrm{lin}}+\mathcal E_{\mathrm{cub}}+\mathcal E_{\cW},\\
 \mathcal E_{\mathrm{lin}}(t,\xi)
 &=i\int_\R u(t,y)
 \overline{(i\partial_t+\partial_y^2)\Psi_\xi(t,y)}\dd y,\\
 \mathcal E_{\mathrm{cub}}(t,\xi)
 &=-i\beta\left[
 \int_\R\abs{u(t,y)}^2u(t,y)\overline{\Psi_\xi(t,y)}\dd y
 -t^{-1}\abs{\Gamma(t,\xi)}^2\Gamma(t,\xi)
 \right],\\
 \mathcal E_{\cW}(t,\xi)
 &=i\int_\R\cW(y+\gamma t)\abs{u(t,y)}^4u(t,y)
 \overline{\Psi_\xi(t,y)}\dd y.
\end{align*}
For every $1\leq t\leq T$,
\begin{align}
 \norm{\mathcal E_{\mathrm{lin}}(t)}_{L^\infty_\xi}
 &\leq Ct^{-5/4}\norm{Lu(t)}_2,
 \label{eq:Elinear-inf}\\
 \norm{\mathcal E_{\mathrm{lin}}(t)}_{L^2_\xi}
 &\leq Ct^{-3/2}\norm{Lu(t)}_2,
 \label{eq:Elinear-L2}\\
 \norm{\mathcal E_{\mathrm{cub}}(t)}_{L^\infty_\xi}
 &\leq C\abs\beta t^{-1/4}\norm{u(t)}_\infty^2\norm{Lu(t)}_2,
 \label{eq:Ecubic-inf}\\
 \norm{\mathcal E_{\mathrm{cub}}(t)}_{L^2_\xi}
 &\leq C\abs\beta t^{-1/2}\norm{u(t)}_\infty^2\norm{Lu(t)}_2,
 \label{eq:Ecubic-L2}\\
 \norm{\mathcal E_{\cW}(t)}_{L^\infty_\xi}
 &\leq C\norm{\cW}_1\norm{u(t)}_\infty^5,
 \label{eq:EW-inf}\\
 \norm{\mathcal E_{\cW}(t)}_{L^2_\xi}
 &\leq Ct^{-1/4}\norm{\cW}_1\norm{u(t)}_\infty^5.
 \label{eq:EW-L2}
\end{align}
\end{lemma}

\begin{proof}
Let $P=i\partial_t+\partial_y^2$. Differentiating \eqref{eq:Gamma-def}, integrating by parts in the linear term, and using
\[
 Pu=\beta\abs{u}^2u-\cW(y+\gamma t)\abs{u}^4u
\]
gives
\begin{align*}
 \partial_t\Gamma
 &=-i\beta\int\abs{u}^2u\overline{\Psi_\xi}\dd y
 +i\int\cW(y+\gamma t)\abs{u}^4u\overline{\Psi_\xi}\dd y
 +i\int u\overline{P\Psi_\xi}\dd y.
\end{align*}
We estimate the three terms separately.

For the linear packet error, use \eqref{eq:packet-residual-exact} and write the last integral in terms of $g$. The change of variables from \eqref{eq:Gamma-average} gives
\[
 \mathcal E_{\mathrm{lin}}(t,\xi)
 =\frac{i}{2\sqrt t}
 \int g(t,t\xi+\sqrt{2t}z)\overline{K_\chi(z)}\dd z.
\]
The zero-mean identity \eqref{eq:Kchi-zero-mean} allows us to subtract $g(t,t\xi)$:
\begin{equation}\label{eq:Elinear-cancel}
 \mathcal E_{\mathrm{lin}}(t,\xi)
 =\frac{i}{2\sqrt t}
 \int\left[g(t,t\xi+\sqrt{2t}z)-g(t,t\xi)\right]
 \overline{K_\chi(z)}\dd z.
\end{equation}
For the pointwise estimate, the fundamental theorem of calculus, \eqref{eq:g-L-relation}, and Cauchy--Schwarz give
\[
 \abs{g(t,t\xi+\sqrt{2t}z)-g(t,t\xi)}
 \leq C t^{-3/4}\abs z^{1/2}\norm{Lu(t)}_2.
\]
Inserting this into \eqref{eq:Elinear-cancel} proves \eqref{eq:Elinear-inf}, since $K_\chi$ is compactly supported and smooth. For the $L^2_\xi$ estimate, translation in the physical variable gives
\begin{align*}
 \norm{g(t,t\xi+\sqrt{2t}z)-g(t,t\xi)}_{L^2_\xi} &=t^{-1/2}\norm{g(t,\cdot+\sqrt{2t}z)-g(t,\cdot)}_{L^2_y}\\
 &\leq t^{-1/2}\sqrt{2t}\abs z\norm{\partial_yg(t)}_2
 \leq Ct^{-1}\abs z\norm{Lu(t)}_2.
\end{align*}
Multiplication by the prefactor $t^{-1/2}$ in \eqref{eq:Elinear-cancel} yields \eqref{eq:Elinear-L2}.

For the cubic term, the oscillatory phases cancel, so
\begin{equation}\label{eq:cubic-packet-average}
 \int\abs{u}^2u\overline{\Psi_\xi}\dd y
 =\sqrt t\int \cN_3\left(g(t,t\xi+\sqrt{2t}z)\right)\chi(z)\dd z.
\end{equation}
The leading term is
\[
 \sqrt t\,\cN_3(g_\chi)=t^{-1}\abs{\Gamma}^2\Gamma.
\]
On the support of $\chi$, the triangle inequality, the fundamental theorem
of calculus, and \eqref{eq:g-average-pointwise} give
\begin{equation*}
 \begin{split}
 \abs{g(t,t\xi+\sqrt{2t}z)-g_\chi(t,\xi)}
 \leq\abs{g(t,t\xi+\sqrt{2t}z)-g(t,t\xi)}+\abs{g(t,t\xi)-g_\chi(t,\xi)}
 \leq Ct^{-3/4}\norm{Lu(t)}_2.
 \end{split}
\end{equation*}
Moreover, \eqref{eq:gchi-average} gives
$\abs{g_\chi(t,\xi)}\leq\norm{\chi}_1\norm{u(t)}_\infty$.
Since
\[
 \abs{\cN_3(\mathfrak z_1)-\cN_3(\mathfrak z_2)}
 \leq C\left(\abs{\mathfrak z_1}^2+\abs{\mathfrak z_2}^2\right)
 \abs{\mathfrak z_1-\mathfrak z_2},
\]
and $\abs{g}=\abs{u}$, the difference between
\eqref{eq:cubic-packet-average} and $\sqrt t\,\cN_3(g_\chi)$ is bounded
pointwise by
\[
 C\sqrt t\,\norm{u(t)}_\infty^2
 t^{-3/4}\norm{Lu(t)}_2
 =Ct^{-1/4}\norm{u(t)}_\infty^2\norm{Lu(t)}_2.
\]
After multiplication by $\abs\beta$, this proves \eqref{eq:Ecubic-inf}. For
the $L^2_\xi$ estimate, the same decomposition and the argument used for
\eqref{eq:packet-L2xi} yield, uniformly for $z\in\supp\chi$,
\begin{equation*}
 \norm{g(t,t\xi+\sqrt{2t}z)-g_\chi(t,\xi)}_{L^2_\xi}
 \leq Ct^{-1}\norm{Lu(t)}_2.
\end{equation*}
Hence Minkowski's inequality and the same Lipschitz estimate for $\cN_3$ give
\[
 \sqrt t\,\norm{u(t)}_\infty^2
 \sup_{z\in\supp\chi}
 \norm{g(t,t\xi+\sqrt{2t}z)-g_\chi(t,\xi)}_{L^2_\xi}
 \leq Ct^{-1/2}\norm{u(t)}_\infty^2\norm{Lu(t)}_2,
\]
which proves \eqref{eq:Ecubic-L2}.

Finally,
\[
 \mathcal E_{\cW}(t,\xi)
 =i\int\cW(y+\gamma t)\abs{u(t,y)}^4u(t,y)
 \overline{\Psi_\xi(t,y)}\dd y.
\]
Since $\norm{\Psi_\xi}_\infty\leq C$,
\[
 \abs{\mathcal E_{\cW}(t,\xi)}
 \leq C\norm{u(t)}_\infty^5
 \int_\R\abs{\cW(y+\gamma t)}\dd y
 =C\norm{\cW}_1\norm{u(t)}_\infty^5,
\]
which is \eqref{eq:EW-inf}. For fixed $y$, the condition $z=(y-\xi t)/\sqrt{2t}\in\supp\chi$ restricts $\xi$ to an interval of length $O(t^{-1/2})$. Since the packet is uniformly bounded,
\begin{equation*}
 \norm{\Psi_\xi(t,y)}_{L^2_\xi}\leq Ct^{-1/4}.
\end{equation*}
Minkowski's inequality then yields
\begin{align*}
 \norm{\mathcal E_{\cW}(t)}_{L^2_\xi}
 &\leq\int_\R\abs{\cW(y+\gamma t)}\abs{u(t,y)}^5
 \norm{\Psi_\xi(t,y)}_{L^2_\xi}\dd y\\
 &\leq Ct^{-1/4}\norm{\cW}_1\norm{u(t)}_\infty^5,
\end{align*}
which proves \eqref{eq:EW-L2}. This estimate uses spatial integrability of $\cW$ and is uniform in the packet velocity, including $\xi=-\gamma$.
\end{proof}

Under the bootstrap assumption \eqref{eq:bootstrap-Linf} and the energy estimate \eqref{eq:energy-growth-packet}, the remainder bounds become
\begin{align}
 \norm{\mathcal E_\Gamma(t)}_{L^\infty_\xi}
 &\leq C\varepsilon\left(1+\abs\beta D_{\mathrm{boot}}^2\varepsilon^2\right)
 t^{-5/4+\delta_E}+C_{\cW}D_{\mathrm{boot}}^5\varepsilon^5t^{-5/2},
 \label{eq:EGamma-bootstrap-inf}\\
 \norm{\mathcal E_\Gamma(t)}_{L^2_\xi}
 &\leq C\varepsilon\left(1+\abs\beta D_{\mathrm{boot}}^2\varepsilon^2\right)
 t^{-3/2+\delta_E}+C_{\cW}D_{\mathrm{boot}}^5\varepsilon^5t^{-11/4}.
 \label{eq:EGamma-bootstrap-L2}
\end{align}
The cubic and linear packet errors have the same time exponent. The localized quintic error decays faster.

\subsection{Global decay and convergence of the modified profile}

\begin{theorem}\label{thm:forward-scattering}
Let $\beta,\gamma\in\R$, and let $\cW$ be real-valued and satisfy
\eqref{eq:W-assumption}. There is
$\varepsilon_0=\varepsilon_0(\beta,\cW)>0$, independent of $\gamma$, such that if
\begin{equation*}
 u(1)=u_1\in\Sigma,
 \qquad
 \norm{u_1}_\Sigma\leq\varepsilon\leq\varepsilon_0,
\end{equation*}
then \eqref{eq:comoving} has a unique global solution
\[
 u\in C\bigl([1,\infty);\Sigma\bigr)\cap C^1\bigl([1,\infty);H^{-1}(\R)\bigr).
\]
There is a fixed constant $C_\delta>0$ such that
$\delta=C_\delta\abs\beta\varepsilon^2<1/16$ and
\begin{equation}\label{eq:global-decay}
 \norm{u(t)}_{L^\infty_y}\leq C\varepsilon t^{-1/2},
 \qquad
 \norm{u(t)}_{H^1_y}+\norm{Lu(t)}_{L^2_y}
 \leq C\varepsilon t^\delta.
\end{equation}
All estimates in this theorem hold for every $t\geq1$.
There is a unique profile $A\in L^2\cap L^\infty$ satisfying
\begin{equation}\label{eq:A-size}
 \norm A_2+\norm A_\infty\leq C\varepsilon
\end{equation}
and, more sharply,
\begin{equation}\label{eq:A-reg}
 A\in H^\varsigma(\R)\qquad\text{for every }\varsigma<1-2\delta.
\end{equation}
The packet coefficient satisfies
\begin{align}
 \norm{\Gamma(t)-A e^{-i\beta\abs{A}^2\log t}}_{L^\infty_\xi}
 &\leq C\varepsilon t^{-1/4+2\delta},
 \label{eq:Gamma-asymptotic-inf}\\
 \norm{\Gamma(t)-A e^{-i\beta\abs{A}^2\log t}}_{L^2_\xi}
 &\leq C\varepsilon t^{-1/2+2\delta}.
 \label{eq:Gamma-asymptotic-L2}
\end{align}
Consequently,
\begin{align}
 &\norm{u(t,y)-t^{-1/2}e^{iy^2/(4t)}A(y/t)
 e^{-i\beta\abs{A(y/t)}^2\log t}}_{L^\infty_y}
 \leq C\varepsilon t^{-3/4+2\delta},
 \notag\\
 &\norm{u(t,y)-t^{-1/2}e^{iy^2/(4t)}A(y/t)
 e^{-i\beta\abs{A(y/t)}^2\log t}}_{L^2_y}
 \leq C\varepsilon t^{-1/2+2\delta}.
 \label{eq:asymptotic-L2}
\end{align}
The global bounds and the two physical-space asymptotic estimates also hold
for $v(t,x)=u(t,x-\gamma t)$, with $y$ replaced by $x-\gamma t$. The
vector-field bound in the original variables uses
$L_\gamma=x-\gamma t+2it\partial_x$.
\end{theorem}

\begin{proof}
By \autoref{prop:local-theory}, the datum $u_1\in\Sigma$ determines a
unique maximal solution
\[
 u\in C\bigl([1,T_+);\Sigma\bigr)
 \cap C^1\bigl([1,T_+);H^{-1}(\R)\bigr),
\]
and a finite $T_+$ can occur only if
$\norm{u(t)}_{H^1}\to\infty$ along a sequence with $t\uparrow T_+$.
The a priori estimates below are first proved for smooth approximations
and then passed to this solution by the continuous dependence supplied
by \autoref{prop:local-theory}. By the Sobolev embedding at time $t=1$,
\[
 \norm{u(1)}_\infty\leq C\varepsilon.
\]
Choose a constant $C_{\mathrm{boot}}$, independent of $D_{\mathrm{boot}}$, that dominates
the constants in the Sobolev, energy, packet comparison, and time-integration
estimates above. Fix $D_{\mathrm{boot}}\geq4C_{\mathrm{boot}}$. Decrease
$\varepsilon_0$ so that the smallness requirements in
\autoref{lem:energy-packet-section} hold and
\begin{equation}\label{eq:bootstrap-smallness-explicit}
 \abs\beta D_{\mathrm{boot}}^2\varepsilon_0^2
 +C_{\cW}D_{\mathrm{boot}}^5\varepsilon_0^4\leq1.
\end{equation}
We choose $C_{\cW}$ large enough that
\eqref{eq:bootstrap-smallness-explicit} implies
\eqref{eq:energy-lemma-smallness}.
Fix $C_\delta\geq C_ED_{\mathrm{boot}}^2+1$ and set
\[
 \delta=C_\delta\abs\beta\varepsilon^2,
 \qquad
 \delta_E=C_E\abs\beta D_{\mathrm{boot}}^2\varepsilon^2\leq\delta.
\]
Decrease $\varepsilon_0$ once more so that $\delta<1/16$. Define
\[
 T_{\mathrm{boot}}
 :=\sup\left\{T\in(1,T_+):
 \norm{u(t)}_\infty\leq D_{\mathrm{boot}}\varepsilon t^{-1/2}
 \text{ for every }t\in[1,T]\right\}.
\]
Local well-posedness and continuity imply $T_{\mathrm{boot}}>1$. On every
compact subinterval of $[1,T_{\mathrm{boot}})$,
\autoref{lem:energy-packet-section} and $\delta_E\leq\delta$ give the
second estimate in \eqref{eq:global-decay}. From \eqref{eq:Gamma-ODE},
\[
 \partial_t\abs{\Gamma(t,\xi)}
 \leq\abs{\mathcal E_\Gamma(t,\xi)},
\]
because the resonant cubic coefficient in \eqref{eq:Gamma-ODE} is purely
imaginary. Integrating \eqref{eq:EGamma-bootstrap-inf} gives
\begin{equation}\label{eq:Gamma-uniform}
 \norm{\Gamma(t)}_\infty\leq C_{\mathrm{boot}}\varepsilon.
\end{equation}
The packet comparison estimate \eqref{eq:packet-pointwise} therefore yields
\[
 \norm{u(t)}_\infty
 \leq Ct^{-1/2}\norm{\Gamma(t)}_\infty
 +Ct^{-3/4}\norm{Lu(t)}_2
 \leq C_{\mathrm{boot}}\varepsilon t^{-1/2}
 +C_{\mathrm{boot}}\varepsilon t^{-3/4+\delta}.
\]
Since $\delta<1/4$, the second term is at most
$C_{\mathrm{boot}}\varepsilon t^{-1/2}$. Hence
\[
 \norm{u(t)}_\infty\leq 2C_{\mathrm{boot}}\varepsilon t^{-1/2}
 \leq \frac {D_{\mathrm{boot}}}{2}\varepsilon t^{-1/2},
 \qquad 1\leq t<T_{\mathrm{boot}},
\]
a strict improvement of the bootstrap assumption. Continuity gives
$T_{\mathrm{boot}}=T_+$. If $T_+<\infty$, then
\autoref{lem:energy-packet-section} gives
$\sup_{1\leq t<T_+}\norm{u(t)}_{H^1}\leq C\varepsilon T_+^\delta$.
This contradicts \eqref{eq:H1-blowup-alternative}. Hence $T_+=\infty$,
and \eqref{eq:global-decay} holds for every $t\geq1$.

It remains to identify the asymptotic profile. Set
\begin{equation}\label{eq:modified-packet-profile}
 Z(t,\xi)=\Gamma(t,\xi)
 \exp\left(i\beta\abs{\Gamma(t,\xi)}^2\log t\right).
\end{equation}
Since
\[
 \partial_t\abs{\Gamma}^2
 =2\Re\left(\overline\Gamma\mathcal E_\Gamma\right),
\]
a direct differentiation of \eqref{eq:modified-packet-profile} gives
\begin{equation*}
 \partial_tZ
 =e^{i\beta\abs{\Gamma}^2\log t}
 \left[
 \mathcal E_\Gamma
 +2i\beta\log t\,\Re(\overline\Gamma\mathcal E_\Gamma)\Gamma
 \right].
\end{equation*}
By \eqref{eq:Gamma-uniform},
\begin{equation*}
 \norm{\partial_tZ(t)}_{L^q_\xi}
 \leq C\left(1+\abs\beta\varepsilon^2\log t\right)
 \norm{\mathcal E_\Gamma(t)}_{L^q_\xi},
 \qquad q\in\{2,\infty\}.
\end{equation*}
Put $\mu_{\mathrm{lr}}=\abs\beta\varepsilon^2$. The right-hand side is integrable
by \eqref{eq:EGamma-bootstrap-inf} and \eqref{eq:EGamma-bootstrap-L2}. Hence $Z(t)$
is Cauchy in both $L^2_\xi$ and $L^\infty_\xi$. On every bounded
$\xi$-interval, the two limits agree almost everywhere. Denote the common limit by
$A\in L^2\cap L^\infty$. Integrating from $t$ to infinity and keeping the
logarithmic factor gives
\begin{align}
 \norm{Z(t)-A}_\infty
 &\leq C\varepsilon(1+\mu_{\mathrm{lr}}\log t)t^{-1/4+\delta}
 +C\varepsilon^5(1+\mu_{\mathrm{lr}}\log t)t^{-3/2},
 \label{eq:Z-conv-inf}\\
 \norm{Z(t)-A}_2
 &\leq C\varepsilon(1+\mu_{\mathrm{lr}}\log t)t^{-1/2+\delta}
 +C\varepsilon^5(1+\mu_{\mathrm{lr}}\log t)t^{-7/4}.
 \label{eq:Z-conv-L2}
\end{align}
The estimates \eqref{eq:A-size} follow from \eqref{eq:Gamma-L2} and \eqref{eq:Gamma-uniform}.

Because $\abs{Z}=\abs{\Gamma}$, \eqref{eq:Z-conv-inf} and \eqref{eq:Z-conv-L2} also control $\abs{\Gamma}^2-\abs{A}^2$. Using $\abs{e^{i\varphi}-1}\leq\abs{\varphi}$, one obtains
\begin{align*}
 &\norm{\Gamma(t)-A e^{-i\beta\abs{A}^2\log t}}_{L^q_\xi}\leq\norm{Z(t)-A}_{L^q_\xi}
 +C\abs\beta\log t\,\norm A_\infty
 \norm{\abs{\Gamma}^2-\abs{A}^2}_{L^q_\xi},
 \qquad q\in\{2,\infty\}.
\end{align*}
The last expression is bounded by
$C(1+\mu_{\mathrm{lr}}\log t)\norm{Z(t)-A}_{L^q_\xi}$. The choice of
$C_\delta$ gives
$(1+\mu_{\mathrm{lr}}\log t)^2\leq C t^\delta$. When $\beta=0$, both
$\mu_{\mathrm{lr}}$ and
$\delta$ vanish. Estimates \eqref{eq:Z-conv-inf}--
\eqref{eq:Z-conv-L2} therefore yield
\eqref{eq:Gamma-asymptotic-inf}--\eqref{eq:Gamma-asymptotic-L2} with the
stated $2\delta$ loss.

The pointwise physical-space estimate follows from \eqref{eq:packet-pointwise}
and \eqref{eq:Gamma-asymptotic-inf}. For every $f\in L^2(\R)$,
\[
 \norm{t^{-1/2}f(y/t)}_{L^2_y}
 =\norm{f}_{L^2_\xi}.
\]
Furthermore \eqref{eq:packet-L2xi} becomes, after the change of variables $y=t\xi$,
\begin{equation}\label{eq:packet-L2y}
 \norm{u(t,y)-t^{-1/2}e^{iy^2/(4t)}\Gamma(t,y/t)}_{L^2_y}
 \leq Ct^{-1/2}\norm{Lu(t)}_2.
\end{equation}
Combining \eqref{eq:packet-L2y} and \eqref{eq:Gamma-asymptotic-L2} proves \eqref{eq:asymptotic-L2}.

We finish with the regularity of $A$. From \eqref{eq:Gamma-derivative} and \eqref{eq:global-decay},
\begin{equation}\label{eq:Gamma-H1-growth}
 \norm{\partial_\xi\Gamma(t)}_2\leq C\varepsilon t^\delta.
\end{equation}
Differentiate \eqref{eq:modified-packet-profile} in $\xi$. Using \eqref{eq:Gamma-uniform} and \eqref{eq:Gamma-H1-growth},
\begin{equation*}
 \norm{Z(t)}_{H^1_\xi}
 \leq C\varepsilon\left(1+\abs\beta\varepsilon^2\log t\right)t^\delta.
\end{equation*}
For the following interpolation argument, the Fourier transform in $\xi$ is
\[
 \widehat h(\xi^\ast)=\int_\R e^{-i\xi^\ast\xi}h(\xi)\dd\xi.
\]
Suppose $f(t)\to f_\infty$ in $L^2$ and
\[
 \norm{f(t)-f_\infty}_2\leq Ct^{-\theta_1},
 \qquad
 \norm{f(t)}_{H^1}\leq Ct^{\theta_0}.
\]
For $R_0\geq1$, Fourier truncation gives
\[
 \norm{\mathbf 1_{\{\abs{\xi^\ast}>R_0\}}\widehat{f_\infty}}_2
 \leq Ct^{-\theta_1}+CR_0^{-1}t^{\theta_0}.
\]
Choosing $t=R_0^{1/(\theta_0+\theta_1)}$ shows
\begin{equation}\label{eq:frequency-tail-interpolation}
 \norm{\mathbf 1_{\{\abs{\xi^\ast}>R_0\}}\widehat{f_\infty}}_2
 \leq CR_0^{-\theta_1/(\theta_0+\theta_1)}.
\end{equation}
Decomposing frequency space into the annuli
$2^j\leq\abs{\xi^\ast}<2^{j+1}$ and using
\eqref{eq:frequency-tail-interpolation}, we obtain
\[
 \norm{\widehat{f_\infty}}_{L^2(2^j\leq\abs{\xi^\ast}<2^{j+1})}
 \leq C2^{-j\theta_1/(\theta_0+\theta_1)}.
\]
Therefore
\[
 \sum_{j\geq0}2^{2j\varsigma}
 \norm{\widehat{f_\infty}}_{L^2(2^j\leq\abs{\xi^\ast}<2^{j+1})}^2<\infty
\]
whenever $\varsigma<\theta_1/(\theta_0+\theta_1)$. Hence
$f_\infty\in H^\varsigma$ for every $ \varsigma<\frac{\theta_1}{\theta_0+\theta_1}$.
If $\beta\neq0$, the preceding $H^1$ bound is $O(t^{\theta_0})$ for every
$\theta_0>\delta$. If $\beta=0$, the same conclusion holds for every
$\theta_0>0$. The unabsorbed time integral permits every
$\theta_1<1/2-\delta$ in the $L^2$ convergence estimate. Letting
$\theta_0\downarrow\delta$ and $\theta_1\uparrow1/2-\delta$ gives
\[
 \sup_{\substack{\theta_0>\delta\\0<\theta_1<1/2-\delta}}
 \frac{\theta_1}{\theta_0+\theta_1}=1-2\delta,
\]
which proves \eqref{eq:A-reg}.

Finally, the profile is unique and therefore independent of the packet
cutoff $\chi$. Suppose $B\in L^2\cap L^\infty$ gives the same
$L^2_y$ modified asymptotic. After the self-similar change of variables,
\[
 \norm{Ae^{-i\beta\abs{A}^2\log t}
 -Be^{-i\beta\abs{B}^2\log t}}_2\longrightarrow0.
\]
The reverse triangle inequality gives $\abs{A}=\abs{B}$ almost
everywhere. The two logarithmic phases then coincide, and the same limit
reduces to $\norm{A-B}_2=0$.
\end{proof}

Uniqueness of the modified scattering profile defines a forward scattering
map on the small $\Sigma$-ball.

\begin{corollary}\label{cor:scattering-map}
After decreasing $\varepsilon_0$ if necessary, the assignment
\[
 S_+:\{u_1\in\Sigma:\norm{u_1}_\Sigma<\varepsilon_0\}
 \longrightarrow L^2(\R)\cap L^\infty(\R),
 \qquad S_+(u_1)=A[u_1],
\]
is well defined, continuous, injective, and gauge equivariant:
$S_+(e^{i\vartheta}u_1)=e^{i\vartheta}S_+(u_1)$ for
every $\vartheta\in\R$. Here $A[u_1]$ is the unique profile furnished by
\autoref{thm:forward-scattering}, and the target is equipped with the sum norm
$\norm{h}_2+\norm{h}_\infty$.
\end{corollary}

\begin{proof}
Well-definedness and gauge equivariance follow from
\autoref{thm:forward-scattering} and uniqueness of the profile. For continuity,
restrict the data to a smaller $\Sigma$-ball on which all tail estimates are
uniform. Given $u_{1,n}\to u_1$ in $\Sigma$, finite-time continuous
dependence gives corresponding solutions $u_n$ satisfying
$u_n\to u$ in $C([1,T];\Sigma)$ for every fixed $T$. Denote their packet
coefficients, modified packet profiles, and limits by $\Gamma_n$, $Z_n$,
and $A_n$, respectively. Then $\Gamma_n(T)\to\Gamma(T)$ and
$Z_n(T)\to Z(T)$ in
both $L^2$ and $L^\infty$ for each fixed $T$. First choose $T$ so that the
uniform tails
\[
 \norm{Z_n(T)-A_n}_2+\norm{Z_n(T)-A_n}_\infty
 \quad\text{and}\quad
 \norm{Z(T)-A}_2+\norm{Z(T)-A}_\infty
\]
are small, and then let $n\to\infty$.

For injectivity, put
\[
 \delta_0=C_\delta\abs\beta\varepsilon_0^2,
 \qquad \sigma_{\mathrm d}=\frac12-2\delta_0>0.
\]
Let two solutions have the same profile and write their difference as
$d$. Their uniform $L^2$ asymptotics give
$\norm{d(t)}_2\leq Ct^{-\sigma_{\mathrm d}}$. Subtract the equations and
integrate from $t$ to $T_{\mathrm{fin}}$. Since
$d(T_{\mathrm{fin}})\to0$ in $L^2$ and
$s^{-1-\sigma_{\mathrm d}}$ is integrable, we may let
$T_{\mathrm{fin}}\to\infty$.
The pointwise decay then gives
\[
 \norm{d(t)}_2
 \leq C\int_t^\infty
 \left(\abs\beta\varepsilon_0^2s^{-1}
 +C_{\cW}\varepsilon_0^4s^{-2}\right)\norm{d(s)}_2\dd s.
\]
For $T\geq1$, define the finite quantity
\[
 \mathsf d_T
 :=\sup_{s\geq T}s^{\sigma_{\mathrm d}}\norm{d(s)}_2.
\]
Then
\[
 \mathsf d_T\leq
 \left(\frac{C\abs\beta\varepsilon_0^2}{\sigma_{\mathrm d}}
 +\frac{C_{\cW}\varepsilon_0^4}{T}\right)\mathsf d_T.
\]
Choose the data radius smaller and then $T$ larger so that the coefficient
is below one. Thus $d=0$ on a tail, and local $H^1$ uniqueness propagates
the equality back to $t=1$.
\end{proof}

\begin{corollary}\label{cor:sharp-linear-decay}
If $u_1$ is a nonzero datum in the small $\Sigma$-ball of
\autoref{cor:scattering-map}, then
\[
 \lim_{t\to\infty}t^{1/2}\norm{u(t)}_\infty
 =\norm{A[u_1]}_\infty>0.
\]
Thus the upper bound in \eqref{eq:global-decay} has the optimal time power
for every nonzero small solution.
\end{corollary}

\begin{proof}
The pointwise asymptotic in \autoref{thm:forward-scattering} and the reverse
triangle inequality give
\[
 \left|t^{1/2}\norm{u(t)}_\infty-\norm{A[u_1]}_\infty\right|
 \leq C\varepsilon t^{-1/4+2\delta}\longrightarrow0.
\]
Since $S_+(0)=0$ and $S_+$ is injective,
$A[u_1]\neq0$ whenever $u_1\neq0$. Its continuous representative therefore
has positive $L^\infty$ norm.
\end{proof}

The next proposition identifies the derivative of the scattering map at the
origin. In particular, it shows that the distinguished value of the profile
is nonzero for an open set of small data.

\begin{proposition}\label{prop:scattering-map-linearization}
Let $S_+$ be the forward scattering map from
\autoref{cor:scattering-map}. Define
\begin{equation}\label{eq:linear-scattering-map}
 (S_{\mathrm{lin}} f)(\xi)
 :=\frac{e^{-i\pi/4}}{2\sqrt\pi}
 e^{i\xi^2/4}\widehat f(\xi/2).
\end{equation}
The operator $S_{\mathrm{lin}}:\Sigma\to
L^2(\R)\cap L^\infty(\R)$ is bounded.
More precisely,
\begin{equation}\label{eq:linear-scattering-map-bounded}
 \norm{S_{\mathrm{lin}}f}_2=\norm f_2,
 \qquad
 \norm{S_{\mathrm{lin}}f}_\infty
 \leq\frac1{2\sqrt\pi}\norm f_1
 \leq C\norm f_\Sigma.
\end{equation}
There are constants $\varepsilon_{\mathrm{lin}}>0$ and
$C_{\mathrm{lin}}>0$, depending only on $\beta$ and $\cW$, such that
\begin{equation}\label{eq:scattering-map-linearization}
 \norm{S_+(f)-S_{\mathrm{lin}}f}_2
 +\norm{S_+(f)-S_{\mathrm{lin}}f}_\infty
 \leq C_{\mathrm{lin}}\norm f_\Sigma^3
\end{equation}
whenever $\norm f_\Sigma\leq\varepsilon_{\mathrm{lin}}$. Thus
$S_+$ is Fr\'echet differentiable at the origin from $\Sigma$ to
$L^2\cap L^\infty$, with derivative
$DS_+(0)=S_{\mathrm{lin}}$.
The constants are independent of $\gamma$.
\end{proposition}

\begin{proof}
Plancherel's theorem and the change of variables $\xi=2\zeta$ give
$\norm{S_{\mathrm{lin}}f}_2=\norm f_2$. The Fourier $L^1$ bound and
$\Sigma\hookrightarrow L^1$ give the two $L^\infty$ estimates in
\eqref{eq:linear-scattering-map-bounded}.

Let
\[
 u^{\mathrm{lin}}(t)=e^{i(t-1)\partial_y^2}f
\]
and denote its packet coefficient by $\Gamma^{\mathrm{lin}}$. Since $L(t)$
commutes with the free Schr\"odinger operator,
\[
 \norm{L(t)u^{\mathrm{lin}}(t)}_2
 =\norm{L(1)f}_2\leq3\norm f_\Sigma.
\]
The linear estimates in \autoref{lem:Gamma-ODE} give
\[
 \norm{\partial_t\Gamma^{\mathrm{lin}}(t)}_\infty
 \leq Ct^{-5/4}\norm f_\Sigma,
 \qquad
 \norm{\partial_t\Gamma^{\mathrm{lin}}(t)}_2
 \leq Ct^{-3/2}\norm f_\Sigma.
\]
Hence $\Gamma^{\mathrm{lin}}(t)$ converges in $L^2\cap L^\infty$.

We identify its limit from the free kernel. Since $f\in\Sigma\subset L^1$,
for each $\xi\in\R$,
\begin{align*}
 &\sqrt t\,e^{-it\xi^2/4}u^{\mathrm{lin}}(t,t\xi)=\frac{\sqrt t}{\sqrt{4\pi i(t-1)}}
 \int_\R\exp\left\{i\left[
 \frac{(t\xi-y)^2}{4(t-1)}-\frac{t\xi^2}{4}
 \right]\right\}f(y)\dd y\longrightarrow
 \frac{e^{-i\pi/4}}{2\sqrt\pi}
 e^{i\xi^2/4}\widehat f(\xi/2).
\end{align*}
The dominated convergence theorem applies because the oscillatory factor has
unit modulus and $f\in L^1$. On the other hand,
\eqref{eq:packet-pointwise} gives
\[
 \left|\Gamma^{\mathrm{lin}}(t,\xi)
 -\sqrt t\,e^{-it\xi^2/4}u^{\mathrm{lin}}(t,t\xi)\right|
 \leq Ct^{-1/4}\norm{L(1)f}_2.
\]
The pointwise calculation identifies the $L^2\cap L^\infty$ limit almost
everywhere. The continuous expression in \eqref{eq:linear-scattering-map}
selects its representative. Hence
\begin{equation}\label{eq:free-packet-limit}
 \Gamma^{\mathrm{lin}}(t)\longrightarrow S_{\mathrm{lin}}f
 \quad\text{in }L^2\cap L^\infty.
\end{equation}

Let $u$ be the nonlinear solution with $u(1)=f$, and set
\[
 d(t)=u(t)-u^{\mathrm{lin}}(t),
 \qquad \varepsilon=\norm f_\Sigma.
\]
After decreasing $\varepsilon_{\mathrm{lin}}$, the estimates in
\autoref{thm:forward-scattering} give
\begin{equation*}
 \norm{u(t)}_\infty\leq C\varepsilon t^{-1/2},
 \qquad
 \norm{L(t)u(t)}_2\leq C\varepsilon t^\delta,
 \qquad 0\leq\delta\leq\frac1{16}.
\end{equation*}
The difference has zero initial value and satisfies
\[
 (i\partial_t+\partial_y^2)d
 =\beta\abs u^2u-\cW(y+\gamma t)\abs u^4u.
\]
Applying $L$ and using \eqref{eq:L-cubic}--\eqref{eq:L-W} gives
\begin{align}
 \norm{L(t)d(t)}_2
 &\leq C\abs\beta\int_1^t
 \norm{u(s)}_\infty^2\norm{L(s)u(s)}_2\dd s +C\norm{\cW}_\infty\int_1^t
 \norm{u(s)}_\infty^4\norm{L(s)u(s)}_2\dd s\notag\\
 &\quad\quad+C\norm{\cW'}_2\int_1^t
 s\norm{u(s)}_\infty^5\dd s\notag\\
 &\leq C\varepsilon^3(1+\log t)t^\delta.
 \label{eq:L-nonlinear-linear-difference}
\end{align}
The localized quintic integrals in the second and third lines are
$O(\varepsilon^5)$.

Let $\Gamma$ be the packet coefficient of $u$, let $\mathcal E_\Gamma$ be
the error in \eqref{eq:Gamma-ODE}, and set
\[
 Z(t,\xi)=\Gamma(t,\xi)
 e^{i\beta\abs{\Gamma(t,\xi)}^2\log t}.
\]
For $q\in\{2,\infty\}$, put $a_2=3/2$ and $a_\infty=5/4$.
The estimates in \autoref{lem:Gamma-ODE} and
\eqref{eq:L-nonlinear-linear-difference} imply
\begin{align}
 \norm{\mathcal E_{\mathrm{lin}}[d](t)}_q
 &\leq C\varepsilon^3(1+\log t)t^{-a_q+\delta},
 \label{eq:linearization-error-d}\\
 \norm{\mathcal E_{\mathrm{cub}}[u](t)}_q
 &\leq C\varepsilon^3t^{-a_q+\delta},
 \notag\\
 \norm{\mathcal E_{\cW}[u](t)}_q
 &\leq C\varepsilon^5t^{-a_q},
 \label{eq:linearization-error-W}\\
 \norm{\mathcal E_\Gamma(t)}_q
 &\leq C\varepsilon t^{-a_q+\delta}.
 \label{eq:linearization-error-total}
\end{align}
For example, the localized bounds
$t^{-5/2}$ in $L^\infty$ and $t^{-11/4}$ in $L^2$ have been weakened in
\eqref{eq:linearization-error-W}.

Set $\theta=\beta\abs{\Gamma}^2\log t$. The packet bound gives
$\norm\Gamma_\infty\leq C\varepsilon$ and
$\norm\theta_\infty\leq C\varepsilon^2\log t$. Since
\[
 \partial_tZ=e^{i\theta}\left[
 \mathcal E_\Gamma+2i\beta\log t\,
 \Re(\overline\Gamma\mathcal E_\Gamma)\Gamma\right]
\]
and
$\partial_t\Gamma^{\mathrm{lin}}
=\mathcal E_{\mathrm{lin}}[u^{\mathrm{lin}}]$, linearity of
$\mathcal E_{\mathrm{lin}}$ yields
\begin{align*}
 \partial_tZ-\partial_t\Gamma^{\mathrm{lin}}
 &=e^{i\theta}\mathcal E_{\mathrm{lin}}[d]
 +(e^{i\theta}-1)\mathcal E_{\mathrm{lin}}[u^{\mathrm{lin}}] +e^{i\theta}\left(
 \mathcal E_{\mathrm{cub}}[u]+\mathcal E_{\cW}[u]\right) +2i\beta e^{i\theta}\log t\,
 \Re(\overline\Gamma\mathcal E_\Gamma)\Gamma.
\end{align*}
Using $\abs{e^{i\theta}-1}\leq\abs\theta$ and
\eqref{eq:linearization-error-d}--\eqref{eq:linearization-error-total},
we obtain
\[
 \norm{\partial_tZ(t)-\partial_t\Gamma^{\mathrm{lin}}(t)}_q
 \leq C\varepsilon^3(1+\log t)t^{-a_q+\delta}.
\]
This is integrable because $\delta\leq1/16$. Moreover,
$Z(1)=\Gamma(1)=\Gamma^{\mathrm{lin}}(1)$. Integrating from $1$ to
infinity and using \eqref{eq:free-packet-limit} proves
\eqref{eq:scattering-map-linearization} for smooth data. Approximation in
$\Sigma$, together with \autoref{cor:scattering-map}, proves it for every
$f\in\Sigma$.
\end{proof}

\begin{corollary}\label{cor:nonvanishing-distinguished-profile}
Assume
\[
 M_0=\int_\R\cW(x)\dd x\neq0.
\]
Then the set of sufficiently small data $u_1\in\Sigma$ for which
$M_0A[u_1](-\gamma)\neq0$ is nonempty and open in $\Sigma$.
\end{corollary}

\begin{proof}
Consider $f_\gamma(y)=e^{-i\gamma y/2}e^{-y^2}$. Then
$f_\gamma\in\cS(\R)$ and
\[
 \widehat f_\gamma(\zeta)
 =\sqrt\pi\,e^{-(\zeta+\gamma/2)^2/4}.
\]
Consequently,
\[
 (S_{\mathrm{lin}}f_\gamma)(-\gamma)
 =\frac12e^{i\gamma^2/4-i\pi/4}\neq0.
\]
For sufficiently small nonzero real $\rho_{\mathrm g}$,
\autoref{prop:scattering-map-linearization} gives
\[
 A[\rho_{\mathrm g} f_\gamma](-\gamma)
 =\frac{\rho_{\mathrm g}}2e^{i\gamma^2/4-i\pi/4}
 +O(\rho_{\mathrm g}^3)\neq0.
\]
The profiles have continuous representatives by \eqref{eq:A-reg}.
Continuity of $S_+$ in $L^\infty$ shows that the nonvanishing
condition persists on a $\Sigma$-neighborhood of
$\rho_{\mathrm g} f_\gamma$.
\end{proof}

This gives forward modified scattering and an injective scattering map for
small $\Sigma$ data. It does not assert surjectivity onto
$L^2\cap L^\infty$. The global wave operator in
\autoref{thm:wave-operator} supplies the complementary final-state
construction on the smoother profile class $\fkX$.

\section{The distinguished ray and an exact cubic-renormalized profile}
\label{sec:distinguished}

The coefficient $\cW\bigl(s(\xi+\gamma)\bigr)$ is localized near
$\xi=-\gamma$. On the scale $\abs{\xi+\gamma}\sim s^{-1}$, however,
$\partial_\xi$ has size $s$, so $s^{-2}\partial_\xi^2$ remains of order
one. A scalar eikonal integration that neglects dispersion therefore does
not describe this layer.

We remove the full cubic Duhamel increment. The resulting
history-dependent interaction profile has a tail generated exactly by the
localized quintic term. We use the Fourier transform convention recalled in
\eqref{eq:fourier-convention}.

Let $u$ be the global solution from \autoref{thm:forward-scattering}, and let
$A$ be its modified scattering profile. By \eqref{eq:A-reg} and
$\delta<1/16$, the profile has a unique continuous representative.

The exact cubic counterterm defines the interaction profile used in the inner
analysis.

\begin{definition}
\label{def:exact-renormalized-profile}
For $t\geq1$ and $\zeta\in\R$, define
\begin{equation}\label{eq:exact-renormalized-profile}
 \cP_u(t,\zeta)
 :=e^{it\zeta^2}\widehat u(t,\zeta)
 +i\beta\int_1^t e^{is\zeta^2}
 \widehat{\abs{u(s)}^2u(s)}(\zeta)\dd s.
\end{equation}
\end{definition}
Here, the integral is understood in $L^2_\zeta$ and has a continuous pointwise
representative. By construction, the counterterm cancels the full cubic
Duhamel increment, including its nonresonant part. Thus $\cP_u$ is an exact
cubic-renormalized interaction profile rather than a phase-only profile.

Time integrability of the localized quintic source gives convergence and an
exact representation of the tail.

\begin{proposition}
\label{prop:exact-renormalized-tail}
Let $0<\varepsilon\leq\varepsilon_0$, and let $u$ be the solution furnished
by \autoref{thm:forward-scattering} for data satisfying
$\norm{u(1)}_\Sigma\leq\varepsilon$. There is a unique
\[
 \cP_{u,+}\in L^2(\R)\cap C_0(\R)
\]
such that $\cP_u(t)\to\cP_{u,+}$ both in $L^2_\zeta$ and uniformly in $\zeta$ as
$t\to\infty$. More precisely, for $T\geq1$,
\begin{align}
 \norm{\cP_{u,+}-\cP_u(T)}_{L^2_\zeta}
 &\leq C_{\beta,\cW}\varepsilon^5T^{-7/4},
 \label{eq:P-tail-L2}\\
 \norm{\cP_{u,+}-\cP_u(T)}_{L^\infty_\zeta}
 &\leq C_{\beta,\cW}\varepsilon^5T^{-3/2}.
 \label{eq:P-tail-Linfty}
\end{align}
Its tail satisfies the exact identity
\begin{equation}\label{eq:localized-Duhamel-tail}
 \mathfrak D_{\cW}(T,\zeta)
 :=\cP_{u,+}(\zeta)-\cP_u(T,\zeta)
 =i\int_T^\infty e^{is\zeta^2}
 \widehat{\cW(\cdot+\gamma s)\abs{u(s)}^4u(s)}(\zeta)\dd s.
\end{equation}
\end{proposition}

\begin{proof}
On each compact time interval, the cubic integral is well defined in
$L^2_\zeta$ because
\[
 \norm{\abs{u}^2u}_2\leq\norm{u}_\infty^2\norm{u}_2.
\]
Moreover, $u(t)\in L^2\cap L^\infty$ gives
$\norm{\abs{u}^2u}_1\leq\norm{u}_2^2\norm{u}_\infty$, so its Fourier
transform belongs to $C_0$. Since persistence of $\Sigma$ gives
$u(t)\in\Sigma\subset L^1$, the same is true of
$e^{it\zeta^2}\widehat{u(t)}$. Thus $\cP_u(t)\in L^2\cap C_0$ for every finite
$t$. On compact time intervals, the map
$s\mapsto\abs{u(s)}^2u(s)$ is continuous into $L^1\cap L^2$. Hence the
finite-time cubic integral in \eqref{eq:exact-renormalized-profile} is a
Bochner integral with values in $C_0\cap L^2$.

The Duhamel formula for \eqref{eq:comoving} gives, as an identity in $L^2_\zeta$,
\begin{equation}\label{eq:P-derivative}
 \partial_t\cP_u(t,\zeta)
 =i e^{it\zeta^2}
 \widehat{\cW(\cdot+\gamma t)\abs{u(t)}^4u(t)}(\zeta).
\end{equation}
Equivalently,
\begin{equation}\label{eq:P-exact-representation}
 \cP_u(t,\zeta)
 =e^{i\zeta^2}\widehat{u(1)}(\zeta)
 +i\int_1^t e^{is\zeta^2}
 \widehat{\cW(\cdot+\gamma s)\abs{u(s)}^4u(s)}(\zeta)\dd s.
\end{equation}
The decay estimate in \autoref{thm:forward-scattering} implies
\begin{align}
 \norm{\cW(\cdot+\gamma t)\abs{u}^4u}_2
 &\leq\norm{\cW}_2\norm{u}_\infty^5
 \leq C_{\beta,\cW}\varepsilon^5t^{-5/2},
 \label{eq:P-source-L2}\\
 \norm{\cW(\cdot+\gamma t)\abs{u}^4u}_1
 &\leq\norm{\cW}_1\norm{u}_\infty^5
 \leq C_{\beta,\cW}\varepsilon^5t^{-5/2}.
 \label{eq:P-source-L1}
\end{align}
The first bound makes the right-hand side of \eqref{eq:P-derivative}
integrable in $L^2_\zeta$. The second bound and the Fourier
$L^1\to L^\infty$ estimate show that its $L^\infty_\zeta$ norm is integrable
in time. The map
\[
 t\longmapsto \cW(\cdot+\gamma t)\abs{u(t)}^4u(t)
\]
is continuous with values in $L^1\cap L^2$. The right-hand side of
\eqref{eq:P-derivative} is therefore Bochner integrable with values in both
$L^2_\zeta$ and $C_0(\R)$.
The $L^2$ and uniform limits agree almost everywhere, so their continuous
representatives coincide. Thus $\cP_u(t)$ has a common limit in both
spaces. Integrating \eqref{eq:P-derivative} from $T$ to infinity proves
\eqref{eq:localized-Duhamel-tail}. Integrating
\eqref{eq:P-source-L1} proves \eqref{eq:P-tail-Linfty}.

For the sharper $L^2$ rate, set
\[
 F(s,y)=\cW(y+\gamma s)\abs{u(s,y)}^4u(s,y).
\]
Plancherel's theorem and the dual form of the one-dimensional homogeneous
Strichartz estimate give
\begin{align*}
 \norm{\cP_{u,+}-\cP_u(T)}_2
 &\leq C\norm{F}_{L_s^{4/3}([T,\infty);L_y^1)}\\
 &\leq C\norm{\cW}_1\varepsilon^5
 \left(\int_T^\infty s^{-10/3}\dd s\right)^{3/4}
 \leq C_{\beta,\cW}\varepsilon^5T^{-7/4},
\end{align*}
which proves \eqref{eq:P-tail-L2}.
Here we used
\[
 \norm{e^{is\partial_y^2}\phi}_{L_s^4L_y^\infty}
 \leq C\norm\phi_2,
 \qquad
 \norm{F(s)}_1\leq\norm{\cW}_1\norm{u(s)}_\infty^5.
\]
\end{proof}

The terminal variable has complementary descriptions in terms of the
initial datum and the modified-scattering amplitude. The second description
isolates the convergent nonresonant cubic correction.

\begin{proposition}\label{prop:terminal-profile-characterization}
Let $u$, $A$, and $\cP_{u,+}$ be as in
\autoref{prop:exact-renormalized-tail}, and set $c_{\mathrm{sp}}=2\sqrt\pi\,e^{i\pi/4}$.
Then
\begin{equation}\label{eq:P-terminal-initial-data}
 \cP_{u,+}(\zeta)
 =e^{i\zeta^2}\widehat{u(1)}(\zeta)
 +i\int_1^\infty e^{is\zeta^2}
 \widehat{\cW(\cdot+\gamma s)\abs{u(s)}^4u(s)}(\zeta)\dd s.
\end{equation}
The integral converges absolutely in $L^2_\zeta\cap C_0(\R)$, and
\begin{equation}\label{eq:P-terminal-fifth-order}
 \norm{\cP_{u,+}-e^{i\zeta^2}\widehat{u(1)}}_2
 +\norm{\cP_{u,+}-e^{i\zeta^2}\widehat{u(1)}}_\infty
 \leq C_{\beta,\cW}\varepsilon^5.
\end{equation}
The terminal profile is also a cubic-order perturbation of the
stationary-phase image of $A$:
\begin{equation}\label{eq:P-terminal-cubic-comparison}
 \norm{\cP_{u,+}-c_{\mathrm{sp}}A(2\,\cdot)}_2
 +\norm{\cP_{u,+}-c_{\mathrm{sp}}A(2\,\cdot)}_\infty
 \leq C_{\beta,\cW}\varepsilon^3.
\end{equation}
Define
\begin{align}
 \Phi_A(t,\zeta)
 &:=c_{\mathrm{sp}}A(2\zeta)
 e^{-i\beta\abs{A(2\zeta)}^2\log t},
 \label{eq:Phi-A-Fourier}\\
 \cN_A^{\mathrm{res}}(t,\zeta)
 &:=t^{-1}\abs{A(2\zeta)}^2\Phi_A(t,\zeta).
 \notag
\end{align}
For every $0<\sigma<1-2\delta$ and $T\geq1$,
\begin{align}
 &\left\|\cP_{u,+}-c_{\mathrm{sp}}A(2\,\cdot)
 -i\beta\int_1^T\left[
 e^{is\zeta^2}\widehat{\abs{u(s)}^2u(s)}
 -\cN_A^{\mathrm{res}}(s)\right]\dd s\right\|_2\notag\\
 &\quad\leq C_{\beta,\cW}\varepsilon^5T^{-7/4}
 +C\varepsilon T^{-1/2+2\delta}
 +C_\sigma T^{-\sigma/2}
 \left(1+\abs\beta\norm A_\infty^2\log T\right)\norm A_{H^\sigma}.
 \label{eq:P-terminal-normal-form-finite}
\end{align}
The following integral converges absolutely in $L^2_\zeta$, for every $\beta\in\R$,
\begin{equation}\label{eq:P-terminal-normal-form}
 \cP_{u,+}=c_{\mathrm{sp}}A(2\,\cdot)
 +i\beta\lim_{T\to\infty}\int_1^T\left[
 e^{is\zeta^2}\widehat{\abs{u(s)}^2u(s)}
 -\cN_A^{\mathrm{res}}(s)\right]\dd s.
\end{equation}
Thus $\cP_{u,+}$ is the stationary-phase image of $A$ corrected by the
convergent nonresonant part of the cubic interaction. If $\beta=0$, then
\begin{equation}\label{eq:P-terminal-beta-zero}
 \cP_{u,+}(\zeta)=c_{\mathrm{sp}}A(2\zeta)\quad \text{for every } \zeta\in\R.
\end{equation}
 Finally,
\begin{equation}\label{eq:P-terminal-mass-comparison}
 \left|\norm{\cP_{u,+}}_2-\sqrt{2\pi}\norm A_2\right|
 \leq C_{\beta,\cW}\varepsilon^5.
\end{equation}
\end{proposition}

\begin{proof}
Equation \eqref{eq:P-terminal-initial-data} follows by letting $t$ tend to
infinity in \eqref{eq:P-exact-representation}. The bounds
\eqref{eq:P-source-L2}--\eqref{eq:P-source-L1} show absolute convergence in
$L^2\cap C_0$ and prove \eqref{eq:P-terminal-fifth-order}.
The normalization in \eqref{eq:linear-scattering-map} gives
\[
 c_{\mathrm{sp}}
 \bigl(S_{\mathrm{lin}}u(1)\bigr)(2\zeta)
 =e^{i\zeta^2}\widehat{u(1)}(\zeta).
\]
Combining this identity with
\eqref{eq:scattering-map-linearization} and
\eqref{eq:P-terminal-fifth-order}, and using the change of variables
$\xi=2\zeta$, proves \eqref{eq:P-terminal-cubic-comparison} after decreasing
$\varepsilon_0$ so that $\varepsilon_0\leq1$.

For the stationary-phase normalization, put
\[
 B_t(\xi)=A(\xi)e^{-i\beta\abs{A(\xi)}^2\log t},\quad \text{ and }
\]
\[
 u_{\mathrm{as}}(t,y)
 =t^{-1/2}e^{iy^2/(4t)}B_t(y/t).
\]
Completing the square in the Fourier integral gives the exact identity
\begin{equation*}
 e^{it\zeta^2}\widehat{u_{\mathrm{as}}(t)}(\zeta)
 =c_{\mathrm{sp}}
 \left[e^{it^{-1}\partial_\xi^2}B_t\right](2\zeta).
\end{equation*}
For $0<\sigma<1$, the Fourier multiplier estimate gives
\[
 \norm{(e^{it^{-1}\partial_\xi^2}-1)B_t}_2
 \leq Ct^{-\sigma/2}\norm{B_t}_{H^\sigma}.
\]
The map $z\mapsto ze^{-i\beta\abs z^2\log t}$ is Lipschitz on
$\{\abs z\leq\norm A_\infty\}$, with Lipschitz constant at most
$1+C\abs\beta\norm A_\infty^2\log t$. The fractional difference-quotient
characterization of $H^\sigma$ therefore implies
\[
 \norm{B_t}_{H^\sigma}
 \leq C\left(1+\abs\beta\norm A_\infty^2\log t\right)\norm A_{H^\sigma}.
\]
Combining these estimates with \eqref{eq:asymptotic-L2} yields
\begin{align}
 \norm{e^{it\zeta^2}\widehat{u(t)}-\Phi_A(t)}_2\leq C\varepsilon t^{-1/2+2\delta}
 +C_\sigma t^{-\sigma/2}
 \left(1+\abs\beta\norm A_\infty^2\log t\right)\norm A_{H^\sigma}.
 \label{eq:Fourier-profile-stationary-limit}
\end{align}

By \eqref{eq:Phi-A-Fourier},
\[
 \partial_t\Phi_A(t)=-i\beta\cN_A^{\mathrm{res}}(t).
\]
Hence
\[
 \Phi_A(T)+i\beta\int_1^T\cN_A^{\mathrm{res}}(s)\dd s
 =\Phi_A(1)=c_{\mathrm{sp}}A(2\,\cdot).
\]
Subtracting this identity from the definition of $\cP_u(T)$ gives
\begin{align*}
 \cP_u(T)-c_{\mathrm{sp}}A(2\,\cdot)
 =e^{iT\zeta^2}\widehat{u(T)}-\Phi_A(T)+i\beta\int_1^T\left[
 e^{is\zeta^2}\widehat{\abs{u(s)}^2u(s)}
 -\cN_A^{\mathrm{res}}(s)\right]\dd s.
\end{align*}
The estimates for $\mathcal P_{u,+}-\mathcal P_u(T)$ and
$e^{iT\zeta^2}\widehat{u(T)}-\Phi_A(T)$ prove the asserted finite-$T$
bound. To justify the limiting identity for every $\beta\in\mathbb R$,
including $\beta=0$, we establish absolute integrability of the
nonresonant cubic remainder with values in $L^2_\zeta$.

Fix $0<\sigma<1-2\delta$ and, with $B_s$ and $u_{\mathrm{as}}$ as above,
set
\[
  C_s(\xi):=|B_s(\xi)|^2B_s(\xi)
  =|A(\xi)|^2A(\xi)e^{-i\beta|A(\xi)|^2\log s},
\]
and
\[
  E_{\mathrm{nr}}(s,\zeta)
  :=e^{is\zeta^2}\widehat{|u(s)|^2u(s)}(\zeta)
    -\mathcal N_A^{\mathrm{res}}(s,\zeta).
\]
Applying the same stationary-phase identity to the cubic leading state
gives, as an identity in $L^2_\zeta$,
\[
  e^{is\zeta^2}
  \widehat{|u_{\mathrm{as}}(s)|^2u_{\mathrm{as}}(s)}(\zeta)
  =\frac{c_{\mathrm{sp}}}{s}
    \bigl[e^{is^{-1}\partial_\xi^2}C_s\bigr](2\zeta),
  \qquad
  \mathcal N_A^{\mathrm{res}}(s,\zeta)
  =\frac{c_{\mathrm{sp}}}{s}C_s(2\zeta).
\]

On the disk $|z|\leq\|A\|_\infty$, the map
$z\mapsto |z|^2z e^{-i\beta|z|^2\log s}$ vanishes at zero and has
Lipschitz constant at most
$C\|A\|_\infty^2(1+|\beta|\|A\|_\infty^2\log s)$.
Since $0<\sigma<1$, the fractional difference-quotient
characterization of $H^\sigma$, together with
$\|A\|_\infty\leq C\varepsilon$, therefore yields
\[
  \|C_s\|_{H^\sigma}
  \leq C_\sigma\varepsilon^2
       (1+|\beta|\varepsilon^2\log s)\|A\|_{H^\sigma}.
\]
The decay and asymptotic estimates for $u$ also give
\begin{align*}
  \bigl\||u(s)|^2u(s)
       -|u_{\mathrm{as}}(s)|^2u_{\mathrm{as}}(s)\bigr\|_{L^2_y}
  &\leq C\bigl(\|u(s)\|_\infty^2
                  +\|u_{\mathrm{as}}(s)\|_\infty^2\bigr)
           \|u(s)-u_{\mathrm{as}}(s)\|_{L^2_y} \\
  &\leq C\varepsilon^3s^{-3/2+2\delta}.
\end{align*}
Consequently, Plancherel's theorem and the fractional Fourier
multiplier estimate imply
\begin{align*}
  \|E_{\mathrm{nr}}(s)\|_{L^2_\zeta}
  &\leq C\varepsilon^3s^{-3/2+2\delta}
       +\frac{C}{s}
        \bigl\|(e^{is^{-1}\partial_\xi^2}-1)C_s\bigr\|_{L^2_\xi} \\
  &\leq C\varepsilon^3s^{-3/2+2\delta}
       +C_\sigma\varepsilon^2\|A\|_{H^\sigma}
        (1+|\beta|\varepsilon^2\log s)s^{-1-\sigma/2}.
\end{align*}
Both terms are integrable on $[1,\infty)$ because
$\delta<1/16$ and $\sigma>0$. Thus
\[
  \int_1^\infty
  \|E_{\mathrm{nr}}(s)\|_{L^2_\zeta}\,ds<\infty,
\]
so the nonresonant cubic integral converges absolutely as an
$L^2_\zeta$-valued Bochner integral, including when $\beta=0$.
Letting $T\to\infty$ in the preceding finite-time identity now gives
\[
  \mathcal P_{u,+}
  =c_{\mathrm{sp}}A(2\cdot)
   +i\beta\int_1^\infty E_{\mathrm{nr}}(s,\cdot)\,ds.
\]
When $\beta=0$, the term multiplied by $\beta$ vanishes, and hence
$\mathcal P_{u,+}=c_{\mathrm{sp}}A(2\cdot)$ in $L^2_\zeta$.
Both sides have continuous representatives, so this equality also
holds pointwise.

Mass conservation and \eqref{eq:asymptotic-L2} imply
$\norm A_2=\norm{u(1)}_2$. By the value of $c_{\mathrm{sp}}$, the change of
variables $\xi=2\zeta$, and Plancherel's theorem,
\[
 \norm{c_{\mathrm{sp}}A(2\,\cdot)}_2
 =\sqrt{2\pi}\norm A_2
 =\norm{e^{i\zeta^2}\widehat{u(1)}}_2.
\]
The reverse triangle inequality and \eqref{eq:P-terminal-fifth-order}
give \eqref{eq:P-terminal-mass-comparison}.
\end{proof}

We record two invariance properties of the normalization.

\begin{remark}
\label{rem:P-normalization}
The terminal value $\cP_{u,+}$ is the limit of the exact interaction variable
\eqref{eq:exact-renormalized-profile}. Proposition
\ref{prop:terminal-profile-characterization} relates it to the
modified-scattering amplitude $A$ through a generally nonzero nonresonant
cubic correction. Thus no pointwise identity between
$\abs{\cP_{u,+}}$ and $\abs{A(2\,\cdot)}$ is asserted. Replacing the lower limit $1$ in
\eqref{eq:exact-renormalized-profile} by any fixed base time $t_0\geq1$ changes
$\cP_u(t)$ and $\cP_{u,+}$ by the same time-independent function, so the tail
$\cP_{u,+}-\cP_u(T)$ is unchanged. The construction is also gauge covariant.
If $u_\vartheta=e^{i\vartheta}u$, then
$\cP_{u_\vartheta}=e^{i\vartheta}\cP_u$.
\end{remark}

The tail concentrates near the distinguished frequency on the scale
$T^{-1/2}$.

\begin{theorem}
\label{thm:exact-boundary-layer}
Let $0<\varepsilon\leq\varepsilon_0$, and let $u$ be the solution furnished
by \autoref{thm:forward-scattering} for data satisfying
$\norm{u(1)}_\Sigma\leq\varepsilon$. Let $A$ be its modified scattering
profile. Set
\begin{equation*}
 \zeta_0=-\frac\gamma2,
 \qquad A_0=A(-\gamma),
 \qquad \alpha=\beta\abs{A_0}^2,
 \qquad M_0=\int_\R\cW(x)\dd x,
\end{equation*}
and define
\begin{equation}\label{eq:F-alpha}
 \mathfrak F_\alpha(\eta)
 :=\int_1^\infty \tau^{-5/2-i\alpha}e^{i\eta^2\tau}\dd\tau.
\end{equation}
Set $\delta=C_\delta\abs\beta\varepsilon^2$ as in
\eqref{eq:global-decay}, and set $a_*=\frac14-2\delta>0$.
Then, for every $R_{\mathrm{in}}>0$ and every $0\leq\kappa<1/2$, there is
$C_{R_{\mathrm{in}},\kappa}<\infty$ such that, for $T\geq2$,
\begin{equation}\label{eq:boundary-layer-limit}
 \sup_{\abs\eta\leq R_{\mathrm{in}}T^\kappa}
 \left|
 \begin{aligned}
 &T^{3/2}e^{i\alpha\log T}
 \left[\cP_{u,+}-\cP_u(T)\right]
 \left(\zeta_0+\frac{\eta}{\sqrt T}\right)\\
 &\qquad-iM_0\abs{A_0}^4A_0\,\mathfrak F_\alpha(\eta)
 \end{aligned}
 \right|
 \leq C_{R_{\mathrm{in}},\kappa}
 \left(T^{-a_*}+T^{-1/2+\kappa}\right).
\end{equation}
The constant $C_{R_{\mathrm{in}},\kappa}$ is independent of $T$ and may depend on the
fixed solution, $\beta$, and $\cW$. The inner frequency scale is
$T^{-1/2}$ and is centered at $\zeta=-\gamma/2$. The estimate is uniform on
the expanding $\eta$-windows in \eqref{eq:boundary-layer-limit}, which
correspond to $\zeta$-windows of width $O(T^{-1/2+\kappa})$. Since the
Schr\"odinger group velocity is $2\zeta$, the center corresponds to the
self-similar ray $\xi=-\gamma$. For $\kappa=0$, the convergence is uniform on
every fixed compact $\eta$-interval at rate $O(T^{-a_*})$.

Estimate \eqref{eq:boundary-layer-limit} is an absolute uniform
approximation for every $\kappa<1/2$. At the edge
$\abs\eta\asymp T^\kappa$, equation \eqref{eq:F-alpha-outer} shows that the
limiting profile has its natural size $T^{-2\kappa}$. The error is smaller
than the natural edge scale only when
\[
 2\kappa<\min\{a_*,1/2-\kappa\},
 \qquad\text{equivalently}\qquad
 \kappa<\frac18-\delta.
\]
Thus $\kappa<1/8-\delta$ is the effective range in which the estimate
resolves the leading profile at the edge. For
$1/8-\delta\leq\kappa<1/2$, it remains a uniform absolute estimate.
\end{theorem}

\begin{proof}
By \eqref{eq:localized-Duhamel-tail}, the profile difference in
\eqref{eq:boundary-layer-limit} is exactly $\mathfrak D_{\cW}(T,\zeta)$. It
therefore remains to analyze the integral on the right-hand side of
\eqref{eq:localized-Duhamel-tail}.

Fix $R_{\mathrm{in}}>0$ and $0\leq\kappa<1/2$. Set
\begin{equation*}
 \theta_A=\frac38-2\delta,
 \qquad
 s_A=\frac{15}{16}-2\delta.
\end{equation*}
Since $\delta<1/16$, these exponents satisfy
\[
 0<\theta_A<\frac12,
 \qquad
 \frac12<s_A<1-2\delta,
 \qquad
 a_*<\theta_A<s_A-\frac12.
\]
By \eqref{eq:A-reg}, $A\in H^{s_A}\hookrightarrow C^{0,\theta_A}$.
Set $\delta_1=2\delta$. Write the physical asymptotic expansion from
\autoref{thm:forward-scattering} as
\begin{equation}\label{eq:u-asymp-section4}
 u(s,y)=s^{-1/2}e^{iy^2/(4s)}
 A(y/s)e^{-i\beta\abs{A(y/s)}^2\log s}
 +\mathcal E_{\mathrm{as}}(s,y),
\end{equation}
where
\begin{equation}\label{eq:Eas-section4}
 \norm{\mathcal E_{\mathrm{as}}(s)}_\infty
 \leq C\varepsilon s^{-3/4+\delta_1},
 \qquad \delta_1<\frac18.
\end{equation}

Set
\[
 \zeta=\zeta_0+\varrho,
 \qquad y=x-\gamma s.
\]
The quadratic phases satisfy the exact identity
\begin{equation}\label{eq:phase-identity-section4}
 s\zeta^2-\zeta y+\frac{y^2}{4s}
 =s\varrho^2-\varrho x+\frac{x^2}{4s}.
\end{equation}
Let
\[
 U_{\mathrm{as}}(s,y)
 :=s^{-1/2}e^{iy^2/(4s)}
 A(y/s)e^{-i\beta\abs{A(y/s)}^2\log s}.
\]
The polynomial Lipschitz bound for $z\mapsto\abs{z}^4z$, together with
\eqref{eq:u-asymp-section4}, \eqref{eq:Eas-section4}, and
\eqref{eq:global-decay}, gives
\begin{equation*}
 \left|\abs{u}^4u-\abs{U_{\mathrm{as}}}^4U_{\mathrm{as}}\right|
 \leq C\varepsilon^5s^{-11/4+\delta_1}.
\end{equation*}
Therefore its contribution to \eqref{eq:localized-Duhamel-tail} is bounded
uniformly in $\zeta$ by
\begin{equation}\label{eq:Duhamel-error-section4}
 C\varepsilon^5\norm{\cW}_1
 \int_T^\infty s^{-11/4+\delta_1}\dd s
 \leq C\varepsilon^5T^{-7/4+\delta_1}.
\end{equation}
After multiplication by $T^{3/2}$ this is
$O(T^{-1/4+\delta_1})=O(T^{-a_*})$.

It remains to evaluate the contribution of $U_{\mathrm{as}}$. Define
$\widetilde A_s(x)=A(-\gamma+x/s)$. Equation
\eqref{eq:phase-identity-section4} gives
\begin{align}
 \mathfrak D_{\cW}^{\mathrm{as}}(T,\zeta_0+\varrho)
 &:=i\int_T^\infty s^{-5/2}e^{is\varrho^2}
 \int_\R \cW(x)e^{-i\varrho x}e^{ix^2/(4s)}
 \abs{\widetilde A_s(x)}^4\widetilde A_s(x)
 e^{-i\beta\abs{\widetilde A_s(x)}^2\log s}\dd x\dd s.
 \label{eq:D-as-exact}
\end{align}
Now set
\[
 \varrho=\frac\eta{\sqrt T},
 \qquad s=T\tau.
\]
Multiplying \eqref{eq:D-as-exact} by $T^{3/2}e^{i\alpha\log T}$ yields
\begin{align}
 &i\int_1^\infty \tau^{-5/2}e^{i\eta^2\tau}
 \int_\R \cW(x)e^{-i\eta x/\sqrt T}e^{ix^2/(4T\tau)}
 \abs{\widetilde A_{T\tau}(x)}^4\widetilde A_{T\tau}(x)\notag\\
 &\qquad\qquad\times
 \exp\left(
 i\alpha\log T-i\beta\abs{\widetilde A_{T\tau}(x)}^2\log(T\tau)
 \right)\dd x\dd\tau.
 \label{eq:scaled-D-as}
\end{align}
The exponent in the final factor of \eqref{eq:scaled-D-as} decomposes as
\begin{align*}
 &i\alpha\log T-i\beta
 \abs{\widetilde A_{T\tau}(x)}^2\log(T\tau)\\
 &\qquad=-i\alpha\log\tau
 -i\beta\left(\abs{\widetilde A_{T\tau}(x)}^2-\abs{A_0}^2\right)
 \log(T\tau).
\end{align*}
After factoring out $e^{i\eta^2\tau}$ in \eqref{eq:scaled-D-as}, define
\begin{align*}
 B_T(\tau,x,\eta)
 &=e^{-i\eta x/\sqrt T}e^{ix^2/(4T\tau)}
 \abs{\widetilde A_{T\tau}(x)}^4\widetilde A_{T\tau}(x)\\
 &\quad\times\exp\left(
 i\alpha\log T-i\beta
 \abs{\widetilde A_{T\tau}(x)}^2\log(T\tau)
 \right).
\end{align*}
Since $0<\theta_A<1/2$ and $p>2$,
the decay assumption on $\cW$ implies
\begin{equation}\label{eq:W-moments-boundary}
 \int_\R\left(1+\abs{x}+\abs{x}^{2\theta_A}\right)
 \abs{\cW(x)}\dd x<\infty.
\end{equation}
For $\abs\eta\leq R_{\mathrm{in}}T^\kappa$, the bounds
$\abs{e^{iz}-1}\leq\abs z$ and
$\abs{e^{iz}-1}\leq C_{\theta_A}\abs z^{\theta_A}$ give
\begin{align*}
 \abs{e^{-i\eta x/\sqrt T}-1}
 &\leq R_{\mathrm{in}}T^{-1/2+\kappa}\abs{x},\\
 \abs{e^{ix^2/(4T\tau)}-1}
 &\leq C_{\theta_A}T^{-\theta_A}\tau^{-\theta_A}
 \abs{x}^{2\theta_A}.
\end{align*}
Let $K_A$ be a global $\theta_A$-H\"older constant for $A$. Then
\[
 \abs{\widetilde A_{T\tau}(x)-A_0}
 \leq K_AT^{-\theta_A}\tau^{-\theta_A}\abs{x}^{\theta_A}.
\]
Using $\abs{e^{iz}-1}\leq\abs z$ once more in the nonlinear phase,
we obtain
\begin{align}
 &\sup_{\abs\eta\leq R_{\mathrm{in}}T^\kappa}
 \abs{B_T(\tau,x,\eta)
 -\abs{A_0}^4A_0\tau^{-i\alpha}}\notag\\
 &\quad\leq C_{R_{\mathrm{in}},\kappa}\left[
 T^{-1/2+\kappa}\abs{x}
 +T^{-\theta_A}\tau^{-\theta_A}\abs{x}^{2\theta_A}
 +T^{-\theta_A}\tau^{-\theta_A}\abs{x}^{\theta_A}
 \bigl(1+\log(T\tau)\bigr)
 \right].
 \label{eq:BT-quantitative}
\end{align}
Here $C_{R_{\mathrm{in}},\kappa}$ also absorbs $K_A$, $\norm{A}_\infty$, and
$\abs\beta$. Multiplying \eqref{eq:BT-quantitative} by
$\tau^{-5/2}\abs{\cW(x)}$ and integrating, using
\eqref{eq:W-moments-boundary}, proves
\begin{align}
 &\sup_{\abs\eta\leq R_{\mathrm{in}}T^\kappa}\left|
 T^{3/2}e^{i\alpha\log T}
 \mathfrak D_{\cW}^{\mathrm{as}}
 \left(T,\zeta_0+\frac\eta{\sqrt T}\right)
 -iM_0\abs{A_0}^4A_0\mathfrak F_\alpha(\eta)
 \right|\notag\\
 &\qquad\leq C_{R_{\mathrm{in}},\kappa}\left[T^{-1/2+\kappa}
 +T^{-\theta_A}(1+\log T)\right].
 \label{eq:D-as-quantitative}
\end{align}
Combining \eqref{eq:D-as-quantitative} with the scaled version of
\eqref{eq:Duhamel-error-section4}, we obtain the errors
$O(T^{-1/2+\kappa})$, $O(T^{-\theta_A}(1+\log T))$, and
$O(T^{-a_*})$. Since
$\theta_A-a_*=1/8$, the logarithmic term is bounded by $CT^{-a_*}$ for
$T\geq2$. This proves \eqref{eq:boundary-layer-limit}.
\end{proof}

\begin{corollary}\label{cor:profile-tail-sharpness}
Under the hypotheses of \autoref{thm:exact-boundary-layer}, suppose that
$M_0A_0\neq0$. There are constants $c_2,c_\infty>0$ and $T_*>0$, depending
on the fixed solution, such that
\begin{align*}
 \norm{\cP_{u,+}-\cP_u(T)}_2
 &\geq c_2\abs{M_0}\abs{A_0}^5T^{-7/4},\\
 \norm{\cP_{u,+}-\cP_u(T)}_\infty
 &\geq c_\infty\abs{M_0}\abs{A_0}^5T^{-3/2}
\end{align*}
for every $T\geq T_*$. Thus both powers in
\eqref{eq:P-tail-L2}--\eqref{eq:P-tail-Linfty} are optimal on the
nontrivial subclass $M_0A_0\neq0$.
\end{corollary}

\begin{proof}
Since
\[
 \mathfrak F_\alpha(0)=\frac1{\frac32+i\alpha},
\]
\eqref{eq:boundary-layer-limit} at $\eta=0$ gives
\[
 T^{3/2}e^{i\alpha\log T}\mathfrak D_{\cW}(T,\zeta_0)
 =\frac{iM_0\abs{A_0}^4A_0}{\frac32+i\alpha}+o(1).
\]
This proves the $L^\infty$ lower bound. For the $L^2$ bound, continuity
and nonvanishing of $\mathfrak F_\alpha$ at zero give $r>0$ such that
$\abs{\mathfrak F_\alpha(\eta)}$ is bounded below for $\abs\eta\leq r$.
The uniform convergence in \eqref{eq:boundary-layer-limit} with
$\kappa=0$, followed by the change of variables
$\zeta=\zeta_0+\eta/\sqrt T$, yields
\[
 \int_{\abs{\zeta-\zeta_0}\leq r/\sqrt T}
 \abs{\mathfrak D_{\cW}(T,\zeta)}^2\dd\zeta
 \geq c\abs{M_0}^2\abs{A_0}^{10}T^{-7/2}
\]
for all sufficiently large $T$. Taking square roots proves the claim.
\end{proof}

\begin{remark}
\label{rem:boundary-layer-scope}
By \eqref{eq:localized-Duhamel-tail}, the tail of $\cP_u$ consists exactly
of the localized quintic Duhamel term. A phase-only renormalization leaves
cubic corrections that need not be $O(T^{-3/2})$ under the
present hypotheses. Thus \eqref{eq:boundary-layer-limit} is not asserted for
that profile. \autoref{prop:F-alpha-regularity} concerns the
limiting inner profile. The later
\autoref{thm:terminal-profile-peano-singularity} gives a separate statement
about the terminal profile itself under stronger decay and data regularity.
\end{remark}

When $\beta=0$, the exact profile reduces to the ordinary free interaction
profile.

\begin{corollary}
\label{cor:beta-zero-full-profile}
Assume the hypotheses and notation of
\autoref{thm:exact-boundary-layer}. If $\beta=0$, then
\[
 \cP_u(t,\zeta)=e^{it\zeta^2}\widehat u(t,\zeta).
\]
Consequently, \eqref{eq:boundary-layer-limit} is the inner scaling limit
for the full tail of the ordinary free interaction profile, with
$\alpha=0$.
\end{corollary}

The limiting profile has a precise nonanalytic expansion at the center of
the inner scale.

\begin{proposition}\label{prop:F-alpha-regularity}
For every $\alpha\in\R$, the function $\mathfrak F_\alpha$ from
\eqref{eq:F-alpha} satisfies
\[
 \mathfrak F_\alpha\in C^{2,1}_{\mathrm{loc}}(\R)
 \cap C^\infty(\R\setminus\{0\}),
 \qquad
 \mathfrak F_\alpha\notin C^3(\R).
\]
With the complex-power convention introduced above, the following
expansion holds near $\eta=0$. In particular,
$\abs\eta^{3+2i\alpha}$ is assigned the value zero at $\eta=0$.
\begin{equation}\label{eq:F-alpha-expansion}
 \mathfrak F_\alpha(\eta)
 =\frac{1}{\frac32+i\alpha}
 +\frac{i}{\frac12+i\alpha}\eta^2
 +C_\alpha\abs{\eta}^{3+2i\alpha}
 +\eta^4G_\alpha(\eta^2),
\end{equation}
where $C_\alpha$ is defined in \eqref{eq:C-alpha} and is nonzero.
Here $G_\alpha\in C^\infty([0,\infty))$. Moreover,
\begin{equation}\label{eq:F-alpha-outer}
 \mathfrak F_\alpha(\eta)
 =\frac{i e^{i\eta^2}}{\eta^2}+O_\alpha(\abs\eta^{-4}),
 \qquad \abs\eta\to\infty.
\end{equation}
\end{proposition}

\begin{proof}
For $\eta\neq0$, repeated integration by parts in $\tau$ shows that
$\mathfrak F_\alpha$ is smooth. Put $\upsilon=\eta^2\geq0$. Subtracting the
first two Taylor terms of $e^{i\upsilon\tau}$ gives the following identity.
For $\upsilon>0$, all complex powers below use the real logarithm. They are
extended by zero at $\upsilon=0$ when their real exponent is positive. We have
\begin{align*}
 \mathfrak F_\alpha(\eta)
 -\int_1^\infty\tau^{-5/2-i\alpha}\dd\tau
 -i\upsilon\int_1^\infty\tau^{-3/2-i\alpha}\dd\tau
 &=\int_1^\infty\tau^{-5/2-i\alpha}
 \left(e^{i\upsilon\tau}-1-i\upsilon\tau\right)\dd\tau\\
 &=\upsilon^{3/2+i\alpha}
 \int_\upsilon^\infty\sigma^{-5/2-i\alpha}
 \left(e^{i\sigma}-1-i\sigma\right)\dd\sigma.
\end{align*}
Set
\[
 \mathfrak g(\sigma)=
 \begin{cases}
 \dfrac{e^{i\sigma}-1-i\sigma}{\sigma^2},&\sigma\neq0,\\[4pt]
 -\dfrac12,&\sigma=0.
 \end{cases}
\]
Then $\mathfrak g$ is smooth. Denote the last integral by
$J_\alpha(\upsilon)$. It
satisfies the exact identity
\begin{align*}
 J_\alpha(\upsilon)-C_\alpha
 =-\int_0^\upsilon\sigma^{-1/2-i\alpha}\mathfrak g(\sigma)\dd\sigma
 =-\upsilon^{1/2-i\alpha}
 \int_0^1z^{-1/2-i\alpha}
 \mathfrak g(\upsilon z)\dd z.
\end{align*}
Consequently \eqref{eq:F-alpha-expansion} holds with
\begin{equation*}
 G_\alpha(\upsilon)
 =-\int_0^1z^{-1/2-i\alpha}
 \mathfrak g(\upsilon z)\dd z,
\end{equation*}
which is smooth for $\upsilon\geq0$. We next prove that $C_\alpha$ does not
vanish. For
$-2<\Re z_{\mathrm M}<-1$, set
\[
 \mathfrak m(z_{\mathrm M})=\int_0^\infty
 \sigma^{z_{\mathrm M}-1}
 \left(e^{i\sigma}-1-i\sigma\right)\dd\sigma,
 \qquad
 \sigma^{z_{\mathrm M}-1}
 =e^{(z_{\mathrm M}-1)\log\sigma},
\]
where $\log\sigma$ is real on $(0,\infty)$. This integral is absolutely
convergent. The two subtractions give $O(\sigma^2)$ at the origin, while
the integrand is $O(\sigma^{\Re z_{\mathrm M}})$ at infinity. Two integrations by
parts have vanishing boundary terms throughout this strip and give
\begin{align*}
 \mathfrak m(z_{\mathrm M})
 =-\frac{i}{z_{\mathrm M}}\int_0^\infty
 \sigma^{z_{\mathrm M}}\left(e^{i\sigma}-1\right)\dd\sigma=-\frac{1}{z_{\mathrm M}(z_{\mathrm M}+1)}
 \int_0^\infty
 \sigma^{z_{\mathrm M}+1}e^{i\sigma}\dd\sigma.
\end{align*}
The last integral is understood as the Abel limit obtained by inserting
$e^{-h\sigma}$ and sending $h\downarrow0$. Let $\Gamma_{\mathrm E}$ denote
the Euler gamma function. Since
$0<\Re(z_{\mathrm M}+2)<1$, the Euler gamma integral
yields
\[
 \int_0^\infty \sigma^{z_{\mathrm M}+1}e^{i\sigma}\dd\sigma
 =e^{i\pi(z_{\mathrm M}+2)/2}
 \Gamma_{\mathrm E}(z_{\mathrm M}+2).
\]
Using
$\Gamma_{\mathrm E}(z_{\mathrm M}+2)
=z_{\mathrm M}(z_{\mathrm M}+1)
\Gamma_{\mathrm E}(z_{\mathrm M})$, we conclude that
\[
 \mathfrak m(z_{\mathrm M})
 =e^{i\pi z_{\mathrm M}/2}
 \Gamma_{\mathrm E}(z_{\mathrm M}).
\]
At $z_{\mathrm M}=-3/2-i\alpha$ this is exactly
\[
 C_\alpha
 =e^{\frac{i\pi}{2}(-3/2-i\alpha)}
 \Gamma_{\mathrm E}(-3/2-i\alpha)\neq0,
\]
because neither the exponential nor the Euler gamma function vanishes.

It remains to prove the regularity assertion. Set
$\psi_\alpha(\eta)=\abs{\eta}^{3+2i\alpha}$. Then
\[
 \psi_\alpha''(\eta)
 =(3+2i\alpha)(2+2i\alpha)\abs{\eta}^{1+2i\alpha}
 \quad(\eta\neq0),
 \qquad \psi_\alpha''(0)=0.
\]
On each half-line,
$\abs{(\psi_\alpha'')'(\eta)}
=\abs{(3+2i\alpha)(2+2i\alpha)(1+2i\alpha)}$ almost everywhere.
If two points lie on opposite sides of zero, the bound
$\abs{\psi_\alpha''(\eta)}\leq C(\alpha)\abs{\eta}$ gives the same
Lipschitz estimate across zero. Hence
$\psi_\alpha\in C^{2,1}_{\mathrm{loc}}$. For $\eta\neq0$, the expansion
gives
\begin{align}
 \mathfrak F_\alpha'''(\eta)
 &=C_\alpha(3+2i\alpha)(2+2i\alpha)(1+2i\alpha)
 \operatorname{sgn}(\eta)e^{2i\alpha\log\abs\eta}
 +O_\alpha(\abs\eta).\notag
\end{align}
If $\alpha=0$, the third derivative has distinct nonzero one-sided limits.
If $\alpha\neq0$, it oscillates logarithmically and has no limit at zero.
This proves the regularity assertion. Finally, one integration by parts in
\eqref{eq:F-alpha} gives \eqref{eq:F-alpha-outer}. A second integration by
parts bounds the remainder by $O_\alpha(\abs\eta^{-4})$.
\end{proof}

\begin{lemma}\label{lem:peano-oscillatory-tail}
Let $T\geq1$, $r_0>0$, and $\mu>0$. Suppose that
$\mathfrak q:[T,\infty)\times(-r_0,r_0)\to\C$ is three times continuously
differentiable in $r$ and
\[
 \sup_{\abs r<r_0}\abs{\partial_r^j\mathfrak q(s,r)}
 \leq C_js^{-\mu},
 \qquad 0\leq j\leq3.
\]
Define
\[
 \mathcal E_{\mathrm P}(r)
 =\int_T^\infty s^{-5/2}e^{isr^2}\mathfrak q(s,r)\dd s.
\]
Then there is a polynomial $Q_{\mathrm P}$ of degree at most three such that
\[
 \mathcal E_{\mathrm P}(r)=Q_{\mathrm P}(r)+o(\abs r^3)
 \qquad\text{as }r\to0.
\]
\end{lemma}

\begin{proof}
For $0\leq j\leq3$, set
$\mathfrak q_j(s)=\partial_r^j\mathfrak q(s,0)$ and
\[
 H_j(\upsilon_{\mathrm P})=\int_T^\infty
 s^{-5/2}\mathfrak q_j(s)e^{i\upsilon_{\mathrm P}s}\dd s.
\]
The assumed decay shows that $H_j$ is continuously differentiable, with
\[
 H_j'(\upsilon_{\mathrm P})=i\int_T^\infty
 s^{-3/2}\mathfrak q_j(s)e^{i\upsilon_{\mathrm P}s}\dd s.
\]
Choose
\[
 \frac12<\omega<\min\left\{1,\frac12+\mu\right\}.
\]
The inequality $\abs{e^{iz}-1}\leq C_\omega\abs z^\omega$ gives
\[
 \abs{H_j'(\upsilon_{\mathrm P})-H_j'(0)}
 \leq C\abs{\upsilon_{\mathrm P}}^\omega
 \int_T^\infty s^{-3/2-\mu+\omega}\dd s
 \leq C\abs{\upsilon_{\mathrm P}}^\omega.
\]
Consequently,
\begin{equation}\label{eq:Hj-peano}
 H_j(\upsilon_{\mathrm P})
 =H_j(0)+\upsilon_{\mathrm P}H_j'(0)
 +O(\abs{\upsilon_{\mathrm P}}^{1+\omega}).
\end{equation}
Taylor's formula with integral remainder gives
\[
 \mathfrak q(s,r)=\mathfrak q_0(s)+r\mathfrak q_1(s)
 +\frac{r^2}{2}\mathfrak q_2(s)
 +\frac{r^3}{2}\int_0^1(1-\vartheta)^2
 \partial_r^3\mathfrak q(s,\vartheta r)\dd\vartheta.
\]
Substitute this identity into $\mathcal E_{\mathrm P}(r)$ and apply
\eqref{eq:Hj-peano} with $\upsilon_{\mathrm P}=r^2$. The last term is treated by
dominated convergence. Since $1+\omega>3/2$, we obtain
\begin{align*}
 \mathcal E_{\mathrm P}(r)
 =H_0(0)+rH_1(0)
+r^2\left[H_0'(0)+\frac12H_2(0)\right]+r^3\left[H_1'(0)+\frac16H_3(0)\right]
 +o(\abs r^3).
\end{align*}
\end{proof}

\begin{theorem}\label{thm:terminal-profile-peano-singularity}
Assume $p>4$ in \eqref{eq:W-assumption}. Let $u$ be the small solution in
\autoref{thm:forward-scattering}, and assume in addition that
$u(1)\in\cS(\R)$. Use the notation
\[
 \zeta_0=-\frac\gamma2,
 \qquad A_0=A(-\gamma),
 \qquad \alpha=\beta\abs{A_0}^2,
 \qquad M_0=\int_\R\cW(x)\dd x.
\]
There are coefficients $b_0,b_1,b_2,b_3\in\C$ such that
\begin{equation}\label{eq:terminal-profile-peano-expansion}
 \cP_{u,+}(\zeta_0+r)
 =\sum_{j=0}^3b_jr^j
 +iM_0\abs{A_0}^4A_0C_\alpha
 \abs r^{3+2i\alpha}
 +o(\abs r^3)
\end{equation}
as $r\to0$. The complex power is assigned the value zero at $r=0$.
In particular, $\cP_{u,+}$ has a quadratic Peano approximation at
$\zeta_0$ with an $O(\abs r^3)$ remainder. If $M_0A_0\neq0$, it does not
admit a third-order Taylor expansion there. Hence it is not $C^3$ in any
neighborhood of $\zeta_0$. No $C^{2,1}$ regularity on a full neighborhood
is asserted.
\end{theorem}

\begin{proof}
Fix $T\geq2$. Finite-time persistence of Schwartz regularity gives
$u\in C([1,T];\cS)$. Indeed, one commutes the equation with finite products
of $y$ and $\partial_y$ and applies Gronwall's inequality on compact time
intervals. Hence $\cP_u(T)$ is a Schwartz function of $\zeta$.

Set $r=\zeta-\zeta_0$ and
\[
 V(s,x)=s^{1/2}e^{-i(x-\gamma s)^2/(4s)}u(s,x-\gamma s).
\]
The phase identity \eqref{eq:phase-identity-section4} and the exact tail
formula give
\begin{equation}\label{eq:terminal-tail-r}
 \cP_{u,+}(\zeta_0+r)-\cP_u(T,\zeta_0+r)
 =i\int_T^\infty s^{-5/2}e^{isr^2}a(s,r)\dd s,
\end{equation}
where
\[
 a(s,r)=\int_\R\cW(x)e^{-irx}e^{ix^2/(4s)}
 \abs{V(s,x)}^4V(s,x)\dd x.
\]
Let $c_0=\abs{A_0}^4A_0$ and define
\[
 \mathfrak q(s,r)=a(s,r)-c_0s^{-i\alpha}\widehat{\cW}(r).
\]
We claim that there is $\mu>0$ such that
\begin{equation}\label{eq:peano-remainder-derivative-decay}
 \sup_{\abs r\leq1}\abs{\partial_r^j\mathfrak q(s,r)}
 \leq C_js^{-\mu},
 \qquad 0\leq j\leq3.
\end{equation}

To prove the claim, write
\[
 V_A(s,x)=A\left(-\gamma+\frac{x}{s}\right)
 \exp\left[-i\beta
 \left|A\left(-\gamma+\frac{x}{s}\right)\right|^2\log s\right].
\]
The physical-space asymptotic in
\autoref{thm:forward-scattering} gives
\[
 \norm{V(s)-V_A(s)}_\infty
 \leq C\varepsilon s^{-a_*},
 \qquad a_*=\frac14-2\delta>0.
\]
Choose
\[
 0<\theta<\min\left\{\frac12-2\delta,p-4\right\}
\]
and then choose $s_{\mathrm P}$ such that
\[
 \frac12+\theta<s_{\mathrm P}<1-2\delta.
\]
By \eqref{eq:A-reg}, $A\in H^{s_{\mathrm P}}(\R)$. The one-dimensional Sobolev
embedding gives $A\in C^{0,\theta}(\R)$. Therefore
\begin{align*}
 \left|\abs{V_A(s,x)}^4V_A(s,x)-c_0s^{-i\alpha}\right|\leq Cs^{-\theta}(1+\log s)\abs x^\theta.
\end{align*}
The polynomial Lipschitz estimate for $z\mapsto\abs z^4z$ also gives
\[
 \left|\abs{V(s,x)}^4V(s,x)
 -\abs{V_A(s,x)}^4V_A(s,x)\right|
 \leq Cs^{-a_*}.
\]
Choose
\[
 0<\nu<\min\left\{1,\frac{p-4}{2}\right\}.
\]
Then
\[
 \left|e^{ix^2/(4s)}-1\right|
 \leq C_\nu s^{-\nu}\abs x^{2\nu}.
\]
For $0\leq j\leq3$, differentiation in $r$ contributes the factor
$(-ix)^j$. The condition $p>4$, the choices of $\theta$ and $\nu$, and
\eqref{eq:W-assumption} imply
\[
 \int_\R\left(\abs x^j+\abs x^{j+\theta}
 +\abs x^{j+2\nu}\right)\abs{\cW(x)}\dd x<\infty.
\]
After decreasing the exponent to absorb the logarithm, this proves
\eqref{eq:peano-remainder-derivative-decay} with any
\[
 0<\mu<\min\{a_*,\theta,\nu\}.
\]

Apply \autoref{lem:peano-oscillatory-tail} to the contribution of
$\mathfrak q$ in
\eqref{eq:terminal-tail-r}. It is a polynomial of degree at most three
plus $o(\abs r^3)$. The remaining contribution is
\[
 ic_0\widehat{\cW}(r)H_{\alpha,T}(r),
 \qquad
 H_{\alpha,T}(r)
 =\int_T^\infty s^{-5/2-i\alpha}e^{isr^2}\dd s.
\]
Scaling \eqref{eq:F-alpha-expansion} gives
\begin{align*}
 H_{\alpha,T}(r)
 &=\frac{T^{-3/2-i\alpha}}{\frac32+i\alpha}
 +\frac{iT^{-1/2-i\alpha}}{\frac12+i\alpha}r^2+C_\alpha\abs r^{3+2i\alpha}
 +r^4G_{\alpha,T}(r^2),
\end{align*}
where $G_{\alpha,T}$ is smooth near zero. Since $p>4$,
$\widehat{\cW}$ is three times continuously differentiable. Taylor's
theorem gives
\[
 \widehat{\cW}(r)=\sum_{j=0}^3w_jr^j+o(\abs r^3),
 \qquad
 w_j=\frac{\widehat{\cW}^{(j)}(0)}{j!},
 \qquad w_0=M_0.
\]
It follows that
\[
 ic_0\widehat{\cW}(r)H_{\alpha,T}(r)
 =Q_0(r)+iM_0c_0C_\alpha\abs r^{3+2i\alpha}
 +o(\abs r^3)
\]
for a polynomial $Q_0$ of degree at most three. Combining this identity
with the Taylor expansion of $\cP_u(T)$ proves
\eqref{eq:terminal-profile-peano-expansion}.

Assume now that $M_0A_0\neq0$. The coefficient of the complex power is
nonzero by \autoref{prop:F-alpha-regularity}. If $\cP_{u,+}$ admitted a
third-order Taylor expansion at $\zeta_0$, subtracting it from
\eqref{eq:terminal-profile-peano-expansion} and dividing by $r^3$ would
force
\[
 iM_0\abs{A_0}^4A_0C_\alpha
 \operatorname{sgn}(r)e^{2i\alpha\log\abs r}
\]
to have a limit as $r\to0$. For $\alpha=0$, its one-sided limits are
distinct. For $\alpha\neq0$, it oscillates logarithmically. This is a
contradiction.
\end{proof}

\begin{corollary}\label{cor:terminal-singularity-nonempty}
Assume $p>4$ and $M_0\neq0$. For the Schwartz function
$f_\gamma(y)=e^{-i\gamma y/2}e^{-y^2}$ from
\autoref{cor:nonvanishing-distinguished-profile}, every sufficiently small
nonzero real $\rho_{\mathrm g}$ gives a solution with datum
$u(1)=\rho_{\mathrm g} f_\gamma$ for which
\[
 \cP_{u,+}\notin C^3
\]
in every neighborhood of $\zeta_0=-\gamma/2$. The same conclusion holds
on a relatively open set of sufficiently small Schwartz data, with the topology
induced by $\Sigma$.
\end{corollary}

\begin{proof}
The expansion in
\autoref{cor:nonvanishing-distinguished-profile} gives
$A[\rho_{\mathrm g} f_\gamma](-\gamma)\neq0$. Apply
\autoref{thm:terminal-profile-peano-singularity}. Openness follows from
the $L^\infty$ continuity of the scattering map.
\end{proof}

\autoref{prop:F-alpha-regularity} determines the regularity of the limiting
inner profile, while \autoref{thm:terminal-profile-peano-singularity}
proves a singularity of the terminal profile attached to the solution.

\begin{corollary}
\label{cor:inner-profile-regularity}
Under the hypotheses of \autoref{thm:exact-boundary-layer}, assume also that
\[
 M_0A_0\neq0.
\]
Then the rescaled tails of the cubic-renormalized profile in
\eqref{eq:boundary-layer-limit} converge in $C^0_{\mathrm{loc}}$ to a
nonzero profile in $C^{2,1}_{\mathrm{loc}}$. Its third derivative does not
extend continuously to $\eta=0$. 
\end{corollary}
Note that the above conclusion concerns the inner limit. It does not assert
derivative convergence of the finite-$T$ tails or a loss of regularity for
the finite-time solution.

We finish the inner analysis by comparing it with the formal scalar eikonal
integral.

\begin{remark}\label{rem:formal-eikonal}
Deleting $s^{-2}\partial_\xi^2U$ from \eqref{eq:selfsimilar-intro} and
freezing the amplitude gives
\[
 \mathcal I_{\cW}(r):=\int_1^\infty s^{-2}\cW(sr)\dd s.
\]
Splitting the integral at $s=\abs{r}^{-1}$ and expanding $\cW$ near zero
gives
\[
 \mathcal I_{\cW}(r)
 =\cW(0)+\cW'(0)r\log\frac1{\abs r}+O(\abs r)
 \qquad(r\to0).
\]
The scalar calculation first lets $s\to\infty$ with fixed $r\neq0$ and then
lets $r\to0$. It does not apply in the joint regime $\abs r\sim s^{-1}$,
where the dispersive operator is leading order. The tail of $\cP_u$ is instead
organized on the scale $\abs{\zeta+\gamma/2}\sim T^{-1/2}$ and has the profile
\eqref{eq:F-alpha}.
\end{remark}

\begin{proposition}\label{prop:mean-zero-next-limit}
Assume the hypotheses of \autoref{thm:exact-boundary-layer}. Suppose in
addition that $M_0=0$ and $A\in W^{1,\infty}(\R)$. Define
\[
 u_A(s,y)=s^{-1/2}e^{iy^2/(4s)}A(y/s)
 e^{-i\beta\abs{A(y/s)}^2\log s}.
\]
Assume that, for some $\nu>0$,
\begin{equation}\label{eq:enhanced-local-asymptotic}
 \norm{u(s)-u_A(s)}_\infty\leq Cs^{-1-\nu},
 \qquad s\geq2.
\end{equation}
Set
\[
 M_1=\int_\R x\cW(x)\dd x.
\]
Then, for every $R>0$,
\begin{align}
 \sup_{\abs\eta\leq R}\left|
 T^2e^{i\alpha\log T}
 \mathfrak D_{\cW}\left(T,\zeta_0+\frac{\eta}{\sqrt T}\right)
 -M_1\abs{A_0}^4A_0\,\eta\mathfrak F_\alpha(\eta)
 \right|\longrightarrow0
 \label{eq:mean-zero-next-limit}
\end{align}
as $T\to\infty$. If $M_1A_0\neq0$, the limiting profile is nonzero.
\end{proposition}

\begin{proof}
Insert $u_A$ in the exact tail
\eqref{eq:localized-Duhamel-tail}. Since
$\norm u(s)_\infty+\norm{u_A(s)}_\infty\lesssim s^{-1/2}$,
\eqref{eq:enhanced-local-asymptotic} gives
\[
 \norm{\cW(\cdot+\gamma s)
 [\abs u^4u-\abs{u_A}^4u_A]}_1
 \leq Cs^{-3-\nu}.
\]
Its tail is $O(T^{-2-\nu})$ and vanishes after multiplication by $T^2$.

For the leading term, use \eqref{eq:D-as-exact}, set
$s=T\tau$, and multiply by $T^2e^{i\alpha\log T}$. Relative to
\eqref{eq:scaled-D-as}, this leaves an additional factor $\sqrt T$.
Because $M_0=0$,
\[
 \sqrt T\int_\R\cW(x)e^{-i\eta x/\sqrt T}\dd x
 \longrightarrow-i\eta M_1
\]
uniformly for $\eta$ in compact sets. The integrand is dominated by
$C_R\abs{x\cW(x)}$. Moreover,
\[
 \sqrt T\left|e^{ix^2/(4T\tau)}-1\right|
 \leq C\abs x\,\tau^{-1/2}.
\]
Put
\[
 G_{T,\tau}(z)=\abs z^4z
 \exp\left[i\alpha\log T-i\beta\abs z^2\log(T\tau)\right].
\]
On the bounded range of $A$, the polynomial Lipschitz estimate and the
mean value theorem give
\begin{align*}
 \sqrt T\left|
 G_{T,\tau}\left(A\left(-\gamma+\frac{x}{T\tau}\right)\right)
 -G_{T,\tau}(A_0)\right|\leq
 C\abs x\frac{1+\log(T\tau)}{\sqrt T\,\tau}.
\end{align*}
In particular, the nonlinear phase satisfies
\[
 \sqrt T\left|
 \left|A\left(-\gamma+\frac{x}{T\tau}\right)\right|^2
 -\abs{A_0}^2\right|\log(T\tau)
 \leq C\abs x\frac{\log(T\tau)}{\sqrt T\,\tau}.
\]
The last two expressions tend to zero pointwise and are bounded by
$C\abs x(1+\log\tau)/\tau$ for $T\geq2$. The spatial-phase difference is
bounded by $C\abs x\,\tau^{-1/2}$. The Fourier difference quotient is
bounded by $C_R\abs x$ for $\abs\eta\leq R$. After multiplication by
$\tau^{-5/2}\abs{\cW(x)}$, each dominator is integrable in $x$ and
$\tau$. Dominated convergence therefore gives
\[
 i(-i\eta M_1)\abs{A_0}^4A_0
 \int_1^\infty\tau^{-5/2-i\alpha}e^{i\eta^2\tau}\dd\tau,
\]
which is \eqref{eq:mean-zero-next-limit}.
\end{proof}

\begin{remark}\label{rem:mean-zero-scope}
The additional estimate \eqref{eq:enhanced-local-asymptotic} is essential
for this next-order statement. The general $\Sigma$ theory gives only the
$O(s^{-3/4+2\delta})$ pointwise remainder in
\autoref{thm:forward-scattering}. Its contribution to the tail is
$O(T^{-7/4+2\delta})$, which is larger than the $T^{-2}$ first-moment
scale. Thus \autoref{prop:mean-zero-next-limit} is a conditional
next-moment calculation. Estimate \eqref{eq:enhanced-local-asymptotic} is
not established here for any nonempty class of solutions. The new shape has the
expansion
\[
 \eta\mathfrak F_\alpha(\eta)
 =\frac{\eta}{\frac32+i\alpha}
 +\frac{i\eta^3}{\frac12+i\alpha}
 +C_\alpha\eta\abs\eta^{3+2i\alpha}+O(\eta^5).
\]
If $M_1A_0\neq0$, the third derivative of the limiting profile is locally
Lipschitz, but the profile has no fourth derivative at $\eta=0$.
\end{remark}

\section{Higher-order final-state expansions away from the distinguished ray}\label{sec:higher}

This section treats profiles that vanish in a neighborhood of the
distinguished ray. For such profiles, the localized quintic term is smaller
than the algebraic orders retained below. The homogeneous symbol classes
and order-by-order cancellation scheme are adaptations of the higher-order
construction of Lindblad and Soffer \cite[Section~5]{LS2006}. We formulate
the associated real-linear inversion explicitly in terms of the complex
profile $A$, including at its zeros, and apply the recursion only to the
translation-invariant cubic equation. The moving localized coefficient is
then controlled as a support-separated remainder; consequently, none of
the recursively generated coefficients depends on $\cW$. Arbitrary-order
expansions for translation-invariant gauge-invariant polynomial
Schr\"odinger equations were also obtained by space-time resonance methods
in \cite{JS2026}.

The coefficient bounds are not uniform across $\xi=-\gamma$, so this is a
global outer expansion for a separated profile class. Differentiating a
term of order $s^{-k}$ produces the factor $-ik$. This makes the recursive
real-linear operator invertible without division by the leading profile.

\subsection{Separated profiles and symbol classes}

Let $A\in C_c^\infty(\R;\C)$ satisfy
\begin{equation}\label{eq:support-separation}
 \dist\bigl(\supp A,\{-\gamma\}\bigr)>0.
\end{equation}
Choose a compact set $K\subset\R$ such that
\[
 \supp A\subset\operatorname{int}K,
 \qquad
 \dist(K,\{-\gamma\})>0.
\]
When $A=0$, any compact $K$ separated from $-\gamma$ is admissible. For
$s\geq1$, set
\[
 \omega_A(s,\xi)=e^{-i\beta\abs{A(\xi)}^2\log s},
 \qquad
 U_0(s,\xi)=\omega_A(s,\xi)A(\xi).
\]

\begin{definition}\label{def:higher-symbol-classes}
For each integer $k\geq1$, define the homogeneous class
\[
 \cH_k(K):=\left\{
 \omega_A(s,\xi)s^{-k}\sum_{j=0}^Jc_j(\xi)(\log s)^j:
 J\in\mathbb N_0,\ c_j\in C_c^\infty(\R;\C),
 \supp c_j\subset K\right\}.
\]
Define the filtered class of finite sums
\[
 \cG_k(K)
 :=\left\{\sum_{r=k}^{R}F_r:
 R\in\mathbb N,\ R\geq k,\ F_r\in\cH_r(K)\right\}.
\]
For $F\in\cG_k(K)$, write
$\pi_rF\in\cH_r(K)$ for its exact
order-$s^{-r}$ component.
\end{definition}

After division by $\omega_A$, distinct powers $s^{-r}$ with polynomial
coefficients in $\log s$ are linearly independent for large $s$. Hence the
sum in \autoref{def:higher-symbol-classes} is direct and $\pi_r$ is well
defined.

The filtered classes satisfy
\[
 (\omega_A^{-1}\cG_j(K))
 (\omega_A^{-1}\cG_k(K))
 \subset\omega_A^{-1}\cG_{j+k}(K),
 \qquad
 \partial_\xi\cG_j(K)\subset\cG_j(K),
 \qquad
 s^{-1}\cG_j(K)\subset\cG_{j+1}(K).
\]
Complex conjugation preserves $\omega_A^{-1}\cG_j(K)$. These rules
and gauge covariance give the order bookkeeping below.

The operator $\partial_\xi^2$ maps $\cG_k(K)$ into itself and
preserves the support in $K$, although it may increase the logarithmic
degree. For every integer $m\geq0$, separation from $-\gamma$ gives
\begin{equation}\label{eq:W-outer-derivative}
 \sup_{\xi\in K}
 \left|\partial_\xi^m\cW\bigl(s(\xi+\gamma)\bigr)\right|
 \leq C_{m,K}s^{-p},
 \qquad s\geq1.
\end{equation}
Indeed, the left-hand side is bounded by
$s^m C_{\cW,m}\la s(\xi+\gamma)\ra^{-p-m}$, which is at most
$C_{m,K}s^{-p}$ on $K$.

\subsection{The principal linearized inversion}

Let
\begin{align*}
 \mathfrak R_0(U)
 &=i\partial_sU+s^{-2}\partial_\xi^2U
 -\beta s^{-1}\abs U^2U,\\
 \mathfrak R(U)
 &=\mathfrak R_0(U)
 +s^{-2}\cW\bigl(s(\xi+\gamma)\bigr)\abs U^4U.
\end{align*}
The cubic principal linearization at $U_0$ is
\begin{equation}\label{eq:L0}
 \cL_0V=i\partial_sV-\frac\beta s
 \left[2\abs{U_0}^2V+U_0^2\overline V\right].
\end{equation}
The complex conjugate in \eqref{eq:L0} is essential because the linearized
map is real-linear rather than complex-linear.

For integers $k\geq1$ and $j\geq0$, and for $X=X(\xi)$, direct
differentiation gives
\begin{align}
 \cL_0\left[\omega_A\frac{(\log s)^j}{s^k}X\right]
 &=\omega_A\frac{(\log s)^j}{s^{k+1}}\mathfrak L_k[X]
 +ij\omega_A\frac{(\log s)^{j-1}}{s^{k+1}}X,
 \label{eq:L0-action}\\
 \mathfrak L_k[X]
 &=\left[-\beta\abs A^2-ik\right]X-\beta A^2\overline X.
 \notag
\end{align}
The second term on the right-hand side of \eqref{eq:L0-action} is omitted
when $j=0$.
The term $-ikX$ makes the algebraic operator invertible even where $A=0$.

The algebraic inversion and the subsequent comparison of logarithmic powers
make explicit, in the present complex-profile formulation, the inversion
mechanism underlying \cite[Lemma~5.1]{LS2006}.

\begin{lemma}\label{lem:Lk-inverse}
For every integer $k\geq1$ and every $\xi\in\R$, the real-linear map
\[
 X\longmapsto \left[-\beta\abs{A(\xi)}^2-ik\right]X
 -\beta A(\xi)^2\overline X
\]
is invertible on $\C$. For $Y\in\C$, its unique preimage is
\begin{equation}\label{eq:Lk-inverse}
 X=\frac{\left[-\beta\abs{A(\xi)}^2+ik\right]Y
 +\beta A(\xi)^2\overline Y}{k^2}.
\end{equation}
The formula is smooth at every zero of $A$.
\end{lemma}

\begin{proof}
Fix $\xi\in\R$. The equation and its complex conjugate form a two-by-two linear system for
$X$ and $\overline X$. Its determinant is
\[
 \left| -\beta\abs{A(\xi)}^2-ik\right|^2
 -\abs{\beta A(\xi)^2}^2=k^2.
\]
Cramer's rule gives \eqref{eq:Lk-inverse}.
\end{proof}

\begin{lemma}\label{lem:symbolic-inversion}
For the fixed $A$, $\beta$, and $K$ above, and every integer $k\geq1$, the map
$\cL_0:\cH_k(K)\to\cH_{k+1}(K)$ is a real-linear
isomorphism. More precisely, let $J\in\mathbb N_0$ and let
$B_0,\ldots,B_J\in C_c^\infty(\R;\C)$ be supported in $K$. If
\[
 F=\omega_A s^{-k-1}\sum_{j=0}^JB_j(\xi)(\log s)^j,
\]
then the unique
$\varphi=\omega_A s^{-k}\sum_{j=0}^JX_j(\xi)(\log s)^j$
satisfying $\cL_0\varphi=F$ is obtained by setting $X_{J+1}=0$ and solving
downward
\begin{equation}\label{eq:triangular-inversion}
 \mathfrak L_k[X_j]=B_j-i(j+1)X_{j+1},
 \qquad j=J,J-1,\ldots,0.
\end{equation}
\end{lemma}

\begin{proof}
Substitute $\varphi$ into \eqref{eq:L0-action} and compare
logarithmic powers. This gives \eqref{eq:triangular-inversion}.
\autoref{lem:Lk-inverse} determines each $X_j$ uniquely. Formula
\eqref{eq:Lk-inverse} preserves smoothness and support, including at zeros
of $A$.
\end{proof}

We record the first recursion step. By \eqref{eq:UA-second-derivative},
\begin{equation*}
 \partial_\xi^2U_0(s,\xi)
 =\omega_A(s,\xi)\sum_{j=0}^2D_{A,j}(\xi)(\log s)^j,
\end{equation*}
where $D_{A,0},D_{A,1},D_{A,2}\in C_c^\infty(\R;\C)$ are supported in $K$. The leading
dispersive residual therefore has the form
\[
 \omega_A(s,\xi)s^{-2}\sum_{j=0}^2D_{A,j}(\xi)(\log s)^j.
\]
We seek
\begin{equation*}
 \phi_1(s,\xi)=\omega_A(s,\xi)s^{-1}
 \sum_{j=0}^2X_j(\xi)(\log s)^j.
\end{equation*}
Apply \autoref{lem:symbolic-inversion} to the negative leading residual
$-s^{-2}\partial_\xi^2U_0$. Starting with the coefficient of
$(\log s)^2$, \eqref{eq:triangular-inversion} determines $X_2$, then
$X_1$, and finally $X_0$. Later steps have the same triangular form.

\subsection{Construction and error estimate}

The leading residual is
\begin{equation*}
 \mathfrak R(U_0)=s^{-2}\partial_\xi^2U_0
 +s^{-2}\cW\bigl(s(\xi+\gamma)\bigr)\abs{U_0}^4U_0.
\end{equation*}
The first term belongs to $\cG_2(K)$. For every integer $m\geq0$,
\eqref{eq:W-outer-derivative} and the derivatives of $\omega_A$ give a
constant $C_m>0$ and an integer $\mathfrak d_m\geq0$ such that
\begin{equation*}
 \sup_{\xi\in K}\left|
 \partial_\xi^m\left[
 s^{-2}\cW\bigl(s(\xi+\gamma)\bigr)\abs{U_0}^4U_0
 \right]\right|
 \leq C_m s^{-p-2}(1+\log s)^{\mathfrak d_m}.
\end{equation*}

\begin{proposition}\label{prop:UN}
Let $\beta\in\R$, let $A\in C_c^\infty(\R;\C)$, and let
$K\subset\R$ be compact with $\supp A\subset\operatorname{int}K$.
For every integer $N\geq1$, there exists
\begin{equation*}
 U_N=U_0+\sum_{k=1}^N\phi_k,
 \qquad
 \phi_k\in\cH_k(K),
\end{equation*}
such that, for every integer $m\geq0$, there are constants
$C^{(0)}_{N,m}>0$ and $\mathfrak d^{(0)}_{N,m}\in\mathbb N_0$ for which
\begin{equation}\label{eq:residual-R0-UN}
 \norm{\partial_\xi^m\mathfrak R_0(U_N)(s)}_2
 \leq C^{(0)}_{N,m}s^{-N-2}
 (1+\log s)^{\mathfrak d^{(0)}_{N,m}},
 \qquad s\geq1.
\end{equation}
The corrections $\phi_1,\ldots,\phi_N$ are determined by $\beta$ and $A$
alone. They are independent of $\gamma$ and $\cW$ and coincide with the
coefficients for $\cW\equiv0$. The constants in
\eqref{eq:residual-R0-UN} depend only on $N$, $m$, $\beta$, $A$, and $K$.

Assume in addition that $\gamma\in\R$, that
$\dist(K,\{-\gamma\})>0$, and that $\cW$ satisfies
\eqref{eq:W-assumption}. Then there are constants $C_{N,m}>0$ and
$\mathfrak d_{N,m}\in\mathbb N_0$ such that
\begin{equation}\label{eq:residual-UN}
 \norm{\partial_\xi^m\mathfrak R(U_N)(s)}_2
 \leq C_{N,m}\left[
 s^{-N-2}+s^{-p-2}\right](1+\log s)^{\mathfrak d_{N,m}},
 \qquad s\geq1.
\end{equation}
If $\cW\in\cS(\R)$, then for every prescribed $p_*>0$ the second term in
\eqref{eq:residual-UN} may be replaced by
$s^{-p_*-2}(1+\log s)^{\mathfrak d_{N,m}}$, with the constant allowed to depend on
$p_*$ and on finitely many Schwartz seminorms of $\cW$. The constants in
\eqref{eq:residual-UN} may also depend on $\gamma$, $\cW$, and the decay
parameters in \eqref{eq:W-assumption}.
\end{proposition}

\begin{proof}
We construct the corrections recursively using the translation-invariant
operator $\mathfrak R_0$. Its base
residual $\mathfrak R_0(U_0)=s^{-2}\partial_\xi^2U_0$ belongs to
$\cH_2(K)$. Suppose inductively that
\[
 U_n=U_0+\sum_{k=1}^n\phi_k,
 \qquad \phi_k\in\cH_k(K),
 \qquad \mathfrak R_0(U_n)\in\cG_{n+2}(K).
\]
Let $\mathscr E_{n+2}=\pi_{n+2}\mathfrak R_0(U_n)$. By
\autoref{lem:symbolic-inversion}, there is a unique
$\phi_{n+1}\in\cH_{n+1}(K)$ such that
\begin{equation}\label{eq:induction-principal-cancel}
 \cL_0\phi_{n+1}=-\mathscr E_{n+2}.
\end{equation}
For $V,H\in\C$, define the real-linear derivative and quadratic remainder
by
\begin{align*}
 D\cN_3(V)[H]&=2\abs V^2H+V^2\overline H,\\
 \cQ_3(V,H)&=\cN_3(V+H)-\cN_3(V)-D\cN_3(V)[H].
\end{align*}
Thus $\cQ_3(V,H)$ is a sum of terms containing at least two factors from
$H$ and $\overline H$. Set $U_{n+1}=U_n+\phi_{n+1}$. Then
\begin{align}
 \mathfrak R_0(U_{n+1})
 &=\mathfrak R_0(U_n)+\cL_0\phi_{n+1}
 +s^{-2}\partial_\xi^2\phi_{n+1}\notag\\
 &\quad-\frac\beta s
 \left\{[D\cN_3(U_n)-D\cN_3(U_0)][\phi_{n+1}]
 +\cQ_3(U_n,\phi_{n+1})\right\}.
 \notag
\end{align}

Write $U_n=\omega_A\widetilde U_n$. Gauge covariance gives
\[
 \cN_3(U_n)=\omega_A\abs{\widetilde U_n}^2\widetilde U_n,
 \qquad
 \cN_5(U_n)=\omega_A\abs{\widetilde U_n}^4\widetilde U_n.
\]
The filtered-algebra rules imply
\begin{align*}
 s^{-2}\partial_\xi^2\phi_{n+1}
 &\in\cG_{n+3}(K),\\
 s^{-1}[D\cN_3(U_n)-D\cN_3(U_0)][\phi_{n+1}]
 &\in\cG_{n+3}(K),\\
 s^{-1}\cQ_3(U_n,\phi_{n+1})
 &\in\cG_{2n+3}(K)
 \subset\cG_{n+3}(K).
\end{align*}
The second inclusion uses $U_n-U_0\in\cG_1(K)$. After
\eqref{eq:induction-principal-cancel}, these inclusions give
$\mathfrak R_0(U_{n+1})\in\cG_{n+3}(K)$ and close the induction.

The recursion has been applied only to $\mathfrak R_0$. It therefore uses
neither $\gamma$ nor $\cW$, which proves
\eqref{eq:residual-R0-UN} and the asserted independence of the
coefficients. Finally, \eqref{eq:W-outer-derivative}, the compact support in $K$, and
Leibniz's rule imply, for every fixed $m$,
\[
 \norm{\partial_\xi^m\left[
 s^{-2}\cW(s(\xi+\gamma))\abs{U_N}^4U_N\right]}_2
 \leq C_{N,m}s^{-p-2}(1+\log s)^{\mathfrak d_{N,m}}.
\]
Combining this with
$\mathfrak R_0(U_N)\in\cG_{N+2}(K)$ proves
\eqref{eq:residual-UN}. If $\cW$ is Schwartz,
\eqref{eq:W-outer-derivative} holds with any prescribed finite decay
exponent.
\end{proof}

Define the physical approximation
\begin{equation*}
 u_N(t,y)=t^{-1/2}e^{iy^2/(4t)}U_N(t,y/t)
\end{equation*}
and its counterpart in the original variables
\begin{equation}\label{eq:vN}
 v_N(t,x)=t^{-1/2}e^{i(x-\gamma t)^2/(4t)}
 U_N\left(t,\frac{x-\gamma t}{t}\right).
\end{equation}
The phase in \eqref{eq:vN} is written in the comoving coordinate required
by the translation $y=x-\gamma t$.

\begin{theorem}\label{thm:higher}
Let $\beta,\gamma\in\R$, let $\cW$ be real-valued and satisfy
\eqref{eq:W-assumption}, and
let $A\in C_c^\infty(\R;\C)$ satisfy
\eqref{eq:support-separation} and \eqref{eq:effective-smallness}. Let
$N\geq1$ be an integer satisfying
\[
 N\leq\lfloor p\rfloor
 \qquad\text{or}\qquad
 \cW\in\cS(\R).
\]
Choose a compact set $K$
such that $\supp A\subset\operatorname{int}K$ and
$\dist(K,\{-\gamma\})>0$. Let $U_N$ be the pure-cubic profile furnished by
\autoref{prop:UN}, and define $u_N$ and $v_N$ by the formulas above. Let $u$
be the solution furnished by \autoref{thm:wave-operator}, and set
$v(t,x)=u(t,x-\gamma t)$. Then there are constants $C_N>0$, $T_N\geq2$,
and an integer $\mathfrak d_N^{\mathrm{out}}\geq0$ such that
\begin{equation}\label{eq:higher-error}
 \norm{u(t)-u_N(t)}_{H^1_y}
 +\norm{u(t)-u_N(t)}_{L^\infty_y}
 \leq C_Nt^{-N-1}(1+\log t)^{\mathfrak d_N^{\mathrm{out}}}
\end{equation}
for every $t\geq T_N$. Translation to the original variables also gives
\[
 \norm{v(t)-v_N(t)}_{H^1_x}
 +\norm{v(t)-v_N(t)}_{L^\infty_x}
 \leq C_Nt^{-N-1}(1+\log t)^{\mathfrak d_N^{\mathrm{out}}}.
\]
For fixed $\beta$, $\gamma$, $A$, and $N$, the profiles $U_N$, $u_N$, and
$v_N$ are independent of $\cW$. The constants $C_N$ and $T_N$ may depend
on $\cW$.
\end{theorem}

\begin{proof}
Define the decay exponent
\[
 p_N=
 \begin{cases}
 p,&\cW\notin\cS(\R),\\
 N+1,&\cW\in\cS(\R).
 \end{cases}
\]
In the second case, the Schwartz conclusion of \autoref{prop:UN} applies
with exponent $p_N$.
Put $\mathcal E_N=\mathfrak R(U_N)$. The physical residual of $u_N$ is
\begin{equation*}
 R_N(t,y)=t^{-1/2}e^{iy^2/(4t)}\mathcal E_N(t,y/t).
\end{equation*}
Consequently
\[
 \norm{R_N(t)}_2=\norm{\mathcal E_N(t)}_2
\]
and, with $\xi=y/t$,
\begin{equation*}
 \partial_yR_N(t,y)=t^{-1/2}e^{iy^2/(4t)}
 \left[\frac{i\xi}{2}\mathcal E_N(t,\xi)
 +t^{-1}\partial_\xi\mathcal E_N(t,\xi)\right].
\end{equation*}
Every term in $\mathcal E_N$ is supported in the fixed compact set $K$. Thus
\eqref{eq:residual-UN} for $m=0,1$, together with the change of variables
$y=t\xi$, gives, with
$\mathfrak d_N^{\mathrm{res}}
:=\max\{\mathfrak d_{N,0},\mathfrak d_{N,1}\}$,
\begin{equation}\label{eq:RN-physical}
 \norm{R_N(t)}_{H^1_y}
 \leq C_N\left[t^{-N-2}(1+\log t)^{\mathfrak d_N^{\mathrm{res}}}
 +t^{-p_N-2}(1+\log t)^{\mathfrak d_N^{\mathrm{res}}}\right].
\end{equation}
Because $p_N\geq N$, the second term is bounded by the first after
increasing $C_N$.

Set
$\Xi_N=U_N-U_0=\sum_{k=1}^N\phi_k$. Since
$\Xi_N\in\cG_1(K)$, the symbol form of the corrections and their fixed compact
support imply, for some integer $J_N\geq0$,
\begin{align}
 \norm{u_N(t)-u_A(t)}_{H^1_y}
 &\leq C_Nt^{-1}(1+\log t)^{J_N},
 \notag\\
 \norm{u_N(t)-u_A(t)}_{W^{1,\infty}_y}
 &\leq C_Nt^{-3/2}(1+\log t)^{J_N}.
 \label{eq:uN-uA-W1inf}
\end{align}
Indeed, the $t^{-1/2}$ prefactor and the change of variables preserve the
$L^2$ size of $\Xi_N$, which is bounded by
$t^{-1}(1+\log t)^{J_N}$, whereas physical differentiation
produces
\[
 t^{-1/2}e^{iy^2/(4t)}
 \left(\frac{i\xi}{2}\Xi_N+t^{-1}\partial_\xi\Xi_N\right).
\]
The compact $\xi$-support controls the first term. Equations
\eqref{eq:a-beta-A} and \eqref{eq:uN-uA-W1inf} therefore give
\begin{equation*}
 \norm{u_N(t)}_{W^{1,\infty}}
 \leq \mathfrak a_\beta(A)t^{-1/2}
 +C_Nt^{-3/2}(1+\log t)^{J_N}.
\end{equation*}
Recall that $\cW_t(y)=\cW(y+\gamma t)$. Set
$a=\mathfrak a_\beta(A)$ and $\ell(t)=1+\log t$. Repeating the
algebra in \autoref{lem:nonlinear-H1} around $u_N$ gives the following
bound. For $t\geq2$ and $w_1,w_2\in H^1(\R)$, set
$\mu_N=\norm{w_1}_{H^1}+\norm{w_2}_{H^1}$. Then
\begin{align*}
 &\norm{\beta[\cN_3(u_N+w_1)-\cN_3(u_N+w_2)]
 -\cW_t[\cN_5(u_N+w_1)-\cN_5(u_N+w_2)]}_{H^1}\\
 &\hspace{30mm}\leq
 \Lambda_N(t,\mu_N)\norm{w_1-w_2}_{H^1},
\end{align*}
where
\begin{align}
 \Lambda_N(t,\mu_N)
 &\leq C\abs\beta\left[
 a^2t^{-1}+C_Nt^{-2}\ell(t)^{2J_N}
 +\left(at^{-1/2}+C_Nt^{-3/2}\ell(t)^{J_N}\right)\mu_N+\mu_N^2
 \right]\notag\\
 &\quad+C_N\left[
 t^{-2}\ell(t)^{4J_N}
 +t^{-3/2}\ell(t)^{3J_N}\mu_N
 +t^{-1}\ell(t)^{2J_N}\mu_N^2
 +t^{-1/2}\ell(t)^{J_N}\mu_N^3+\mu_N^4
 \right].
 \label{eq:uN-Lambda}
\end{align}
The only nonintegrable linear contribution to this bound is
$C\abs\beta a^2t^{-1}$. Every other term in
\eqref{eq:uN-Lambda} has an additional negative power of $t$ once the
decay of the correction is inserted.

We solve the final-state problem with background $u_N$. Set
\[
 \varpi_N(t)=t^{-N-1}\ell(t)^{\mathfrak d_N^{\mathrm{out}}},
\]
where $\mathfrak d_N^{\mathrm{out}}$ is chosen so that
\begin{equation}\label{eq:KN-choice}
 \mathfrak d_N^{\mathrm{out}}
 \geq\max\{\mathfrak d_N^{\mathrm{res}},4J_N\}.
\end{equation}
We use the elementary estimate
\begin{equation}\label{eq:higher-log-tail}
 \int_t^\infty s^{-b}\ell(s)^m\dd s
 \leq C_{b,m}t^{1-b}\ell(t)^m,
 \qquad b>1,\quad m\geq0,\quad t\geq2.
\end{equation}
Let $T_N\geq2$, to be chosen below, and define
\[
 \mathcal Y_N=\left\{w\in C([T_N,\infty);H^1(\R)):
 \norm w_{\mathcal Y_N}<\infty\right\},
 \qquad
 \norm w_{\mathcal Y_N}
 :=\sup_{t\geq T_N}\varpi_N(t)^{-1}\norm{w(t)}_{H^1_y}.
\]
Writing $u=u_N+w$, put
\[
 \mathcal M_N(w)=\beta\bigl[\cN_3(u_N+w)-\cN_3(u_N)\bigr]
 -\cW_t\bigl[\cN_5(u_N+w)-\cN_5(u_N)\bigr],
\]
and define the final-state map
\begin{equation*}
 \cT_Nw(t)=i\int_t^\infty e^{i(t-s)\partial_y^2}
 \bigl[\mathcal M_N(w)(s)-R_N(s)\bigr]\dd s.
\end{equation*}
Equations \eqref{eq:RN-physical}, \eqref{eq:KN-choice}, and
\eqref{eq:higher-log-tail} give the following bound for the normalized
source size when $T\geq2$.
\begin{align}
 \mathfrak B_N(T)
 &:=\sup_{t\geq T}\varpi_N(t)^{-1}
 \int_t^\infty\norm{R_N(s)}_{H^1}\dd s\notag\\
 &\leq C_N\sup_{t\geq T}
 \left[
 \ell(t)^{\mathfrak d_N^{\mathrm{res}}-\mathfrak d_N^{\mathrm{out}}}
 +t^{N-p_N}
 \ell(t)^{\mathfrak d_N^{\mathrm{res}}-\mathfrak d_N^{\mathrm{out}}}
 \right]
 \leq C_N.
 \notag
\end{align}
At the endpoint $p_N=N$, which can occur only when
$\cW\notin\cS(\R)$ and $p=N$ is an integer, the localized part
is uniformly bounded but has no algebraic gain. The choice
$\mathfrak d_N^{\mathrm{out}}\geq\mathfrak d_N^{\mathrm{res}}$ controls
its logarithmic factor.

Let $\mathsf R>0$ and suppose
$\norm{w_j}_{\mathcal Y_N}\leq\mathsf R$ for $j=1,2$. Set
\[
 d=w_1-w_2,
 \qquad
 \mu_N(s)=\norm{w_1(s)}_{H^1}+\norm{w_2(s)}_{H^1}.
\]
Then
\begin{equation}\label{eq:higher-ball-pointwise}
 \mu_N(s)\leq2\mathsf R\varpi_N(s),
 \qquad
 \norm{d(s)}_{H^1}
 \leq\norm d_{\mathcal Y_N}\varpi_N(s).
\end{equation}
Substituting \eqref{eq:higher-ball-pointwise} into
\eqref{eq:uN-Lambda}, using the unitarity of the free group on $H^1$,
and applying \eqref{eq:higher-log-tail}, we obtain
\begin{equation}\label{eq:higher-contraction-normalized}
 \norm{\cT_Nw_1-\cT_Nw_2}_{\mathcal Y_N}
 \leq
 \left[
 C\abs\beta a^2
 \left(\frac1{N+1}+o_{T_N\to\infty}(1)\right)
 +\varepsilon_N(T_N,\mathsf R)
 \right]\norm d_{\mathcal Y_N},
\end{equation}
where one may take
\begin{align}
 \varepsilon_N(T,\mathsf R)
 &\leq C_NT^{-1}\ell(T)^{4J_N}
+C\abs\beta a\mathsf R
 T^{-N-1/2}\ell(T)^{\mathfrak d_N^{\mathrm{out}}}+C_N\mathsf R
 T^{-N-3/2}\ell(T)^{3J_N+\mathfrak d_N^{\mathrm{out}}}\notag\\
 &\quad+C\abs\beta\mathsf R^2
 T^{-2N-1}\ell(T)^{2\mathfrak d_N^{\mathrm{out}}} + C_N\mathsf R^2
 T^{-2N-2}\ell(T)^{2J_N+2\mathfrak d_N^{\mathrm{out}}}\notag\\
 &\quad+C_N\mathsf R^3
 T^{-3N-5/2}\ell(T)^{J_N+3\mathfrak d_N^{\mathrm{out}}} + C_N\mathsf R^4
 T^{-4N-3}\ell(T)^{4\mathfrak d_N^{\mathrm{out}}}.
 \label{eq:higher-epsilon}
\end{align}
Indeed, the first term in the bracket in
\eqref{eq:higher-contraction-normalized} comes from the long-range cubic
linearization and the asymptotic identity
\[
 \sup_{t\geq T_N}\varpi_N(t)^{-1}
 \int_t^\infty s^{-1}\varpi_N(s)\dd s
 =\frac1{N+1}+o_{T_N\to\infty}(1).
\]
For all other terms, the powers in \eqref{eq:higher-epsilon} follow,
respectively, from the coefficients
\[
 s^{-2}\ell(s)^{4J_N},\quad
 s^{-1/2}\mu_N,\quad
 s^{-3/2}\ell(s)^{3J_N}\mu_N,\quad
 \mu_N^2,\quad
 s^{-1}\ell(s)^{2J_N}\mu_N^2,\quad
 s^{-1/2}\ell(s)^{J_N}\mu_N^3,\quad
 \mu_N^4
\]
in \eqref{eq:uN-Lambda}. Since $N\geq1$, every term on the
right-hand side of \eqref{eq:higher-epsilon} tends to zero as
$T\to\infty$ for each fixed $\mathsf R$.

After the same universal decrease of $\eta_0$ used in
\autoref{lem:tail-uniqueness}, condition
\eqref{eq:effective-smallness} makes the long-range term in
\eqref{eq:higher-contraction-normalized} at most $1/4$ for each fixed
$N\geq1$ after $T_N$ is sufficiently enlarged. To avoid any dependence of
the radius on a subsequently enlarged starting time, first choose a
preliminary $T_N^{(0)}\geq2$, set
\[
 \mathfrak B_N^{(0)}=\mathfrak B_N(T_N^{(0)}),
 \qquad \mathsf R=4\left(1+\mathfrak B_N^{(0)}\right),
\]
and then increase $T_N\geq T_N^{(0)}$ until
$\varepsilon_N(T_N,\mathsf R)\leq1/4$. Restricting the time interval can only
decrease the normalized source bound, so $ \norm{\cT_N0}_{\mathcal Y_N}
 \leq\mathfrak B_N^{(0)}<\frac{\mathsf R}{4}$.
Equation \eqref{eq:higher-contraction-normalized} now gives
\[
 \norm{\cT_Nw_1-\cT_Nw_2}_{\mathcal Y_N}
 \leq\frac12\norm{w_1-w_2}_{\mathcal Y_N}.
\]
Moreover, if $\norm w_{\mathcal Y_N}\leq\mathsf R$, then
\[
 \norm{\cT_Nw}_{\mathcal Y_N}
 \leq\norm{\cT_N0}_{\mathcal Y_N}
 +\norm{\cT_Nw-\cT_N0}_{\mathcal Y_N}
 \leq\frac{\mathsf R}{4}+\frac{\mathsf R}{2}
 =\frac{3\mathsf R}{4}.
\]
Thus $\cT_N$ is a contraction of the closed radius-$\mathsf R$ ball in
$\mathcal Y_N$ into itself. Banach's fixed-point theorem gives a unique
correction $w\in\mathcal Y_N$. Setting $\widetilde u=u_N+w$, we have
\[
 \norm{\widetilde u(t)-u_N(t)}_{H^1}
 \leq\mathsf R t^{-N-1}(1+\log t)^{\mathfrak d_N^{\mathrm{out}}},
 \qquad t\geq T_N.
\]
The one-dimensional Sobolev embedding gives the same bound in
$L^\infty$.
We identify $\widetilde u$ with the solution $u$ from
\autoref{thm:wave-operator}. The original solution satisfies
\[
 \norm{u(t)-u_A(t)}_{H^1}\lesssim
 t^{-1}(1+\abs\beta\log t)^2,
\]
while, by the construction above and the form of $u_N-u_A$,
\[
 \norm{\widetilde u(t)-u_A(t)}_{H^1}
 \lesssim t^{-1}(1+\log t)^{K'}
\]
for some finite $K'$. Hence both differences are $O(t^{-3/4})$ on a
sufficiently large common tail. \autoref{lem:tail-uniqueness} applies and
yields $u=\widetilde u$ there. This proves \eqref{eq:higher-error} for the
solution from \autoref{thm:wave-operator}. Since translation is an isometry
on $H^1$ and $L^\infty$, the asserted estimate for $v-v_N$ follows.
\end{proof}

\begin{corollary}\label{cor:schwartz-invisibility}
Let $\cW_1,\cW_2\in\cS(\R)$ be real-valued. For $j=1,2$, let $u^{(j)}$
be the solution produced by the modified wave operator with coefficient
$\cW_j$. Assume that the parameters $\beta$ and $\gamma$ and the separated
profile $A$ are the same for both solutions, and that $A$ satisfies
\eqref{eq:effective-smallness} for both coefficients.
Then, for every $\nu_{\mathrm{dec}}>0$, there are
$C_{\nu_{\mathrm{dec}}},T_{\nu_{\mathrm{dec}}}>0$ such that
\[
 \norm{u^{(1)}(t)-u^{(2)}(t)}_{H^1}
 +\norm{u^{(1)}(t)-u^{(2)}(t)}_\infty
 \leq C_{\nu_{\mathrm{dec}}}t^{-\nu_{\mathrm{dec}}},
 \qquad t\geq T_{\nu_{\mathrm{dec}}}.
\]
\end{corollary}

\begin{proof}
Choose an integer $N$ with $N+1>\nu_{\mathrm{dec}}$. The pure-cubic
approximation $u_N$ is
the same for both coefficients. Apply \autoref{thm:higher} to each solution,
use the triangle inequality, and absorb the logarithmic factors into the
strict power gap $t^{N+1-\nu_{\mathrm{dec}}}$.
\end{proof}

\begin{remark}\label{rem:higher-scope}
The support condition in \autoref{thm:higher} implies $A(-\gamma)=0$.
Thus this theorem and the nontrivial case of the inner theorem do not apply
to the same profile. The distinction is not merely one of scales. The inner
theorem requires interaction with the distinguished ray, while the global
outer estimate assumes that this interaction is absent. We do not claim a
matched inner--outer expansion for one solution.

For a general smooth profile, the same algebraic recursion determines
formal coefficients on every compact subset of
$\R\setminus\{-\gamma\}$ after multiplication by a cutoff that equals one
near that subset. The global error estimate proved here requires the full
support separation \eqref{eq:support-separation}. This observation is
only a formal coefficient identity and does not provide an asymptotic
approximation to a solution without support separation.
\end{remark}

\end{document}